\documentclass[a4paper]{article}

\usepackage{a4wide}
\usepackage[english]{babel}
\usepackage[T1]{fontenc}

\usepackage{amsmath,amsfonts,amssymb,amsthm}

\usepackage{mathrsfs}

\usepackage{enumitem}

\usepackage[svgnames]{xcolor}

\usepackage{bbold}

\usepackage{tikz}
\usetikzlibrary{fit}
\usetikzlibrary{arrows.meta}

\usepackage{tikz-cd}

\usepackage{witharrows}

\usepackage{hyperref}

\usepackage{graphicx}
\usepackage{caption}

\usetikzlibrary{babel}

\theoremstyle{definition}
\newtheorem{definition}{Definition}[section]
\newtheorem{proposition}[definition]{Proposition}
\newtheorem{example}[definition]{Example}
\newtheorem{lemma}[definition]{Lemma}
\newtheorem{theorem}[definition]{Theorem}

\newtheorem{conjecture}[definition]{Conjecture}

\theoremstyle{remark}
\newtheorem{remark}{Remark}

\definecolor{ColBlack}{RGB}{0,0,0} 
\definecolor{ColWhite}{RGB}{255,255,255} 
\definecolor{ColBlue}{RGB}{13,20,120} 
\definecolor{ColCyan}{RGB}{51,122,121} 
\definecolor{ColLCyan}{RGB}{217,253,253} 

\tikzstyle{Centering}=[{baseline={([yshift=-0.5ex]current bounding box.center)}}]
\tikzstyle{NodeGraph}=[circle,draw=ColCyan,fill=ColLCyan,inner sep=1pt,minimum size=4mm, thick,font=\scriptsize]
\tikzstyle{UnlabeledNodeGraph}=[NodeGraph,minimum size=2mm]
\tikzstyle{RootGraph}=[NodeGraph,rectangle]
\tikzstyle{EdgeGraph}=[ColBlue,cap=round,very thick]
\tikzstyle{ArcGraph}=[EdgeGraph,->]
\tikzstyle{NodeHyper}=[circle, draw=black, fill= black!20,inner sep=1pt,minimum size=3mm, thick, font=\scriptsize]
\tikzstyle{EdgeHyper}=[black, cap=roud, thick]
\tikzstyle{NodeFree}=[circle,draw=black,fill=white,inner sep=1pt,minimum size=6mm, thick,font=\scriptsize]
\tikzstyle{EdgeFree}=[black,cap=round,very thick]

\newcommand{\OEIS}[1]{\href{http://oeis.org/#1}{{\bf #1}}}              

\newcommand{\noemp}{\not = \emptyset}                                   
\newcommand{\take}{\leftarrow}                                          

\newcommand{\set}[1]{\left\{#1\right\}}                                 
\newcommand{\spe}[1]{\textsc{#1}}                                               
\newcommand{\op}[1]{\mathbf{#1}}                                        
\newcommand{\card}[1]{|#1|}                                             
\newcommand{\iso}[1]{\left[#1\right]}                                   

\newcommand{\N}{\mathbb{N}}                                             

\newcommand{\K}{\mathbb{K}}

\newcommand{\p}{\mathcal{P}}                                            
\newcommand{\symgrp}{\mathfrak{S}}                                      
\newcommand{\point}[1]{#1^{\bullet}}                                    

\newcommand{\id}{Id}                                                    
\newcommand{\twist}{\tau}                                               
\newcommand{\hilbert}[1]{\mathcal{H}_{#1}}                              
\newcommand{\comp}[1][\ast]{\circ_{#1}}                                 
\newcommand{\sign}{\text{sign}}                                         
\newcommand{\scalar}[2]{\langle #1 \,|\, #2\rangle}                     
\newcommand{\ori}{\texttt{o}}                                           

\newcommand{\T}{\spe{T}}                                                
\newcommand{\G}{\spe{G}}                                                
\newcommand{\MG}{\spe{MG}}                                              
\newcommand{\MHG}{\spe{MHG}}                                            
\newcommand{\Poly}{\spe{Pol}_+}                                         
\newcommand{\stree}[1]{\mathscr{S}_{{#1}}}                              
\newcommand{\augm}[2]{#1\text{-}#2}                                     
\newcommand{\func}[2]{\mathcal{F}_{#1}^{#2}}                            
\newcommand{\Sym}{\spe{E}}                                              
\newcommand{\Comp}{\spe{Comp}}                                          

\newcommand{\Operad}{\mathcal{O}}                                       
\newcommand{\PLie}{\op{PLie}}                                           
\newcommand{\NAP}{\op{NAP}}                                             
\newcommand{\Com}{\op{Com}}                                             
\newcommand{\ComMag}{\op{ComMag}}                                       
\newcommand{\free}[1]{\op{Free}_{#1}}                                   
\newcommand{\Ope}{\mathrm{Ope}}                                         
\newcommand{\Lie}{\op{Lie}}                                             
\newcommand{\ST}{\op{ST}}                                               
\newcommand{\Seg}{\op{Seg}}                                             
\newcommand{\SP}{\op{SP}}                                               
\newcommand{\LP}{\op{LP}}                                               

\newcommand{\Points}[2]{                                                
    \begin{tikzpicture}[Centering,scale=.6]
        \node[NodeGraph](a)at(0,0){$#1$};
        \node[NodeGraph](b)at(1,0){$#2$};
    \end{tikzpicture}}

\newcommand{\Segment}[2]{                                               
    \begin{tikzpicture}[Centering,scale=.6]
        \node[NodeGraph](a)at(0,0){$#1$};
        \node[NodeGraph](b)at(1,0){$#2$};
        \draw[EdgeGraph](a)--(b);
    \end{tikzpicture}}

\author{J.-C. Aval\footnote{Univ. Bordeaux, Bordeaux INP, CNRS, LaBRI, UMR 5800, F-33400 Talence, France},\hspace{0.2cm} S. Giraudo\footnote{Univ. du Québec à Montréal, LaCIM,
    Montréal, H2X~3Y7, Canada}\hspace{0.2cm} T. Karaboghossian\footnotemark[1]\hspace{0.2cm} and\hspace{0.2cm} A. Tanasa\footnotemark[1] 
\footnote{H. Hulubei Nat. Inst. Phys. Nucl. Engineering Magurele, Romania, EU} \\
{\small aval@labri.fr, giraudo.samuele@uqam.ca, theo.karaboghossian@u-bordeaux.fr, ntanasa@u-bordeaux.fr}}

\title{Operads on graphs: extending the pre-Lie operad and general construction}
\date{}

\begin{document}
\maketitle

\begin{abstract}
The overall aim of this paper is to define a structure of graph operads, thus generalizing the 
celebrated pre-Lie operad on rooted trees.
More precisely, we define two 
operads on multigraphs, and 
exhibit
a 
non trivial link between them and the pre-Lie and Kontsevich-Willwacher operads. We study 
one of these 
operads in more detail. 
While its structure is too involved to 
exhibit
a description 
by generators and relations,  we show that it has interesting finitely generated sub-operads, 
with links with the commutative and the magmatic commutative operads. In particular, one of 
them is Koszul
and this allows us to 
compute its Koszul dual. Finally, we introduce a new framework 
on species and operads and a general way to define operads on multigraphs.

\end{abstract}
\bigskip
\small{
\noindent\textbf{MSC classes:} 5E99, 05C76, 05C25.
\\
\noindent\textbf{Keywords:} Graphs, Species, Operads, Pre-Lie Operad, Koszul duality.}

\tableofcontents

\section*{Introduction}

Operads are mathematical structures which were first introduced as a way to formalize the 
notion of type of algebra: given a set of multilinear maps and relations between them,
the associated operad is composed of all the multilinear maps obtained by composing those in the
initial set.
These form in fact the generators of the operad. As detailed in \cite{MSJ02}, operad 
theory was first used in algebraic topology in the 1960s but had a 'renaissance' in the 1990s where 
it began to be used in many other fields, see for example \cite{LV12} for a very general algebraic 
approach of the theory. In particular, in combinatorics, operads provide the right framework to 
define the embedding of a combinatorial object in another \cite{CSLC, Gir18}. In this context, operads are defined 
by directly describing their objects and how to embed them. This contrast with the original way to define
them by generators and relations, and passing from one definition to the other is often a difficult
question and provides a lot of insight on the structure of the operad. 

A particularly successful example
is the pre-Lie operad which was first defined by being generated by the pre-Lie operator i.e. an antisymmetric
bilinear map $x\vartriangleleft y$ such that 
$(x\vartriangleleft y)\vartriangleleft z -x \vartriangleleft(y\vartriangleleft z) = (x\vartriangleleft z)\vartriangleleft y - x\vartriangleleft(z\vartriangleleft y)$.
Chapoton and Livernet later proved in \cite{CL01} that the elements of the pre-Lie operad could be seen 
as rooted trees and described the composition of two elements in a purely combinatorial way.

The pre-Lie operad along with the nonassociative permutative operad~\cite{Liv06} are two examples
of interesting operads for which the combinatorial interpretation is given by trees. 
While there
exist also in the literature
interesting operads on graphs \cite{Kreimer:2000ja,Kon99,Wil05,MV19}, none 
of them 
is
studied in a purely combinatorial framework. We propose to do so in this paper by
extending 
the pre-Lie operad 
from trees 
to graphs. 
Within our approach, we find it is more natural to work
on multigraphs and define an operad structure $\MG$ on the species of multigraphs, and an operad structure
$\point{\MG}_{or}$ on the species of pointed oriented multigraphs. The operad $\MG$ in particular
restricts to the Kontsevich-Willwacher operad~\cite{MV19} on the species of graphs $\G$.
We show that these operads and the pre-Lie operad $\PLie$ relate to each other by the following 
commutative diagram:
\begin{equation*}
    \begin{tikzcd}
        \T \arrow[r, "\sim"] \arrow[d, hook]
        & \PLie\cap\K\Operad/\mathcal{I} \arrow[r, equal] \arrow[d, hook]
        & \PLie\cap\Operad \arrow[d,hook] \arrow[r, hook]
        & \PLie \arrow[d,hook]\\
        \G_c \arrow[r, "\sim"] \arrow[d, hook]
        & \Operad\cap\point{\G}_{orc}/\mathcal{I}\cap\point{\G_{orc}} \arrow[d, hook]
        & \point{\G}_{orc}\cap\Operad \arrow[l, two heads] \arrow[r, hook] \arrow[d, hook]
        & \point{\G}_{orc}\cap\ST \arrow[d, hook] \\
        \MG_c \arrow[r, "\sim"]
        & \Operad/\mathcal{I}
        & \Operad \arrow[l, two heads] \arrow[r, hook]
        & \MG\times\PLie
    \end{tikzcd},
\end{equation*}
where $\T$, $\G_c$ and $\MG_c$ are respectively the sub-operads of trees, connected graphs and connected
multigraphs of $\MG$, $\point{\G_{orc}}$ is the sub-operad of pointed connected oriented graphs of $\point{\MG}_{or}$, and
$\ST$ is an operad on spanning trees with $\Operad$ and $\mathcal{I}$ sub-species of $\ST$ generated by
ad hoc sums of spanning trees. 
%
We show that while, contrarily to $\PLie$, the operad $\MG$ does not admit a simple presentation by generators
and relation, its sub-operads generated by combinations of the empty graph over two vertices and
the segment graph, have interesting links with the commutative operad $\Com$ and the magmatic operad $\ComMag$. In particular,
one of them is Koszul,
which allows us to 
exhibit its Koszul dual.
We end our study by providing new constructions on species and operads. These constructions enable us to define a general way of
constructing operads on graphs, which we use to justify the operad structure of $\MG$ and $\point{\MG}_{or}$.
More specifically, we give a general way to construct operads on multigraphs where the partial composition of two 
elements $g_1\comp g_2$ is given by the following steps:
\begin{enumerate}
    \item take the disjoint union of $g_1$ and $g_2$,
    \item remove the vertex $\ast$ from $g_1$, the edges previously connected to $\ast$ have now one or more loose end,
    \item connect \underline{independently} each loose ends of $g_1$ to $g_2$ in a certain way,
\end{enumerate}
\noindent where 
independently 
means here
that the way of connecting one loose end does not depend on the way we connect the other loose ends.

This paper is organized as follows. In Section~\ref{sec_intro} 
we give some general definitions of
species theory
and of operad theory, as well as some general definitions on graphs and multigraphs. In Section~\ref{sec_extending_plie} 
we define the aforementioned two 
operads on multigraphs and pointed oriented multigraphs and we exhibit the link between
these operads and the pre-Lie and Kontsevich-Willwacher operads. We then proceed to study the operad
on multigraphs and some of its sub-operads. Finally, in Section~\ref{sec_graph_operads} we 
exhbit 
our 
proposal 
of constructing species and operads and we apply them in order to obtain our general construction of operads on multigraphs.

This paper is an extended version of the extended abstract for FPSAC 2020~\cite{AGKT20} with proofs and additional results.
%


\section{Context}\label{sec_intro}
In all this paper, otherwise stated, $V$ 
denotes a finite set and $V_1$ and $V_2$ two disjoint sets such that
$V=V_1\sqcup V_2$. The letter $n$ always denotes a non negative integer and we denote by $[n]$ the set $\set{1,\dots, n}$.
All vector spaces appearing in this paper are defined over a field of characteristic 0 denoted by $\K$.
Finally, for any set $E$, 
$\K E$ denotes the free vector space over $E$.


\subsection{Definitions and background on species} \label{sec_species}
We recall here basic definitions on species. We refer the interested reader to~\cite{BLL98} 
for a detailed presentation of combinatorial species.

\begin{definition}
    A {\em linear species} $\spe{S}$ consists of the following data:
    \begin{itemize}
       \item For each finite set $V$, a vector space $\spe{S}[V]$ of finite dimension,
       \item For each bijection of finite sets $\sigma: V\to V'$, a linear map $\spe{S}[\sigma]:\spe{S}[V]\to \spe{S}[V']$.
       These maps should be such that $\spe{S}[\sigma_1\circ\sigma_2] = \spe{S}[\sigma_1]\circ \spe{S}[\sigma_2]$ and $\spe{S}[\id] = \id$.
    \end{itemize}
    Furthermore if $\spe{S}[\emptyset]  = \set{0}$, then $\spe{S}$ is say to be {\em positive}.
\end{definition}

We will use the term {\em species} to refer to linear species. 

When defining a species, the maps $\spe{S}[\sigma]$ are often clear under context and we do not mention them.
Let $\spe{S}$ be a species. A {\em sub-species} of $\spe{S}$ is a species $\spe{R}$ such that $\spe{R}[V]$ is 
a sub-space of $\spe{S}[V]$ for every finite set $V$  and $\spe{R}[\sigma] = \spe{S}[\sigma]$ for every bijection of finite sets $\sigma$. 
For $E$ a set of elements of $\spe{S}$, the {\em species generated by $E$} is the smallest sub-species of $\spe{S}$ containing $E$.
We denote by $\spe{S}_{n}$ the sub-species of $\spe{S}$ defined by $\spe{S}_{n}[V]=\spe{S}[V]$ if $\card{V} = n$ and $\spe{S}_n[V]=\emptyset$ otherwise,
and by $\spe{S}_{n+}$ the sub-species of $\spe{S}$ defined by $\spe{S}_{n+}[V]=\spe{S}[V]$ if $\card{V}\geq n$ and $\spe{S}_{n+}[V]=\set{0}$ otherwise.
We also denote by $\spe{S}_+$ the species $\spe{S}_{1^+}$. The {\em Hilbert series} of $\spe{S}$ is the formal power series defined by 
$\hilbert{\spe{S}}(x)=\sum_{n\geq 0}\frac{\dim\spe{S}[[n]]}{n!}x^n$.

%

A {\em morphism of species} from a species $\spe{R}$ to $\spe{S}$ is a collection of linear maps $f_V : \spe{R}[V] \to \spe{S}[V]$ such that for 
each bijection $\sigma: V\to V'$, we have $f_{V'}\circ \spe{R}[\sigma] = \spe{S}[\sigma]\circ f_V$. For easier reading, we will often forget the 
index $V$.

\begin{example}\label{ex_species}
    \begin{itemize}
        \item The {\em Identity}, {\em exponential} and {\em singleton species} are respectively defined by $\spe{\id}[V] = \K V$, $\Sym[V] = \K\set{V}$
        and $\spe{X} = \Sym_1$.
        \item For $V$ a finite set, we denote by $\spe{Pol}[V]$ the \textit{set} (not the vector space) of polynomials with coefficients in $\N$,
        variables in $V$ and null constant coefficient. To consider the species $\K\spe{Pol}$, we must take into consideration the fact that we need to 
        differentiate the plus of polynomials and the addition of vectors. We will thus denote by $\oplus$ the former and keep $+$ for the latter 
        and we will denote by $0_V\in\spe{Pol}[V]$ the polynomial constant to 0 and keep the notation $0$ for the null vector.
        For example, $ab\oplus c$ is an element of $\spe{Pol}[\{a,b,c\}]$, but $a\oplus b+c$ is a vector in $\K \spe{Pol}[\{a,b,c\}]$.
        \item We have a natural morphism from $\spe{\id}$ to $\K\spe{Pol}$ given by $v\mapsto v$ and three natural morphisms from $\Sym$ to
        $\K\spe{Pol}$ respectively given by $V\mapsto \prod_{v\in V}v$, $V\mapsto\bigoplus_{v\in V}v$ and $V\mapsto\sum_{v\in V}v$. In
        particular, the image of $\spe{\id}$ by this morphism is the species $\K\spe{Pol}^1$ of homogeneous polynomials of degree 1.
    \end{itemize}
\end{example}

One strong point of species is the different operations on them which enable to construct new species 
from existing ones. Let $\spe{R}$ and $\spe{S}$ be two species. We can then construct new species 
which are defined as follows: {\em sum} $(\spe{R}+\spe{S})[V] = \spe{R}[V]\bigoplus \spe{S}[V]$, {\em product}
$\spe{R}\cdot \spe{S}[V] = \bigoplus_{V_1\sqcup V_2 = V} \spe{R}[V_1]\otimes \spe{S}[V_2]$,\linebreak[9] {\em Hadamard product}
$(\spe{R}\times \spe{S})[V] = \spe{R}[V]\otimes \spe{S}[V]$, {\em derivative} $\spe{S}'[V] = \spe{S}[V+\set{\ast}]$ (where
$\ast\not\in V$), {\em second derivative} $\spe{S}''[V] = \spe{S}[V+\set{\ast_1,\ast_2}]$ ($\ast_1,\ast_2\not\in V$) and 
{\em pointing} $\point{\spe{S}}[V] = \spe{S}[V]\otimes \K V$.
Furthermore if $\spe{S}$ is positive we can also define the {\em composition} of $\spe{R}$ and $\spe{S}$ by
$$\spe{R}(\spe{S})[V] = \bigoplus_{\text{$P$ partition of $V$}}\spe{R}[P]\bigotimes_{P_i\in P} \spe{S}[P_i],$$
where $\bigotimes_{P_i\in P} \spe{S}[P_i]=\left(\spe{S}[P_1]\otimes\dots\otimes\spe{S}(P_k)\right)_{\symgrp_k}$ 
should be seen as an unordered tensor product. We call {\em assemblies} of $\spe{S}$ the elements of
this unordered tensor product.

\begin{remark}\label{remark_derivation}
    When considering the product of two derivatives $\spe{R}'\cdot\spe{S}'$, to avoid confusion we will use the notations
    $\spe{R}'[V] = \spe{R}[V+\set{\ast_1}]$ and $\spe{S}'[V] = \spe{S}[V+\set{\ast_2}]$.
\end{remark}


\begin{example}\label{remark_derivation}
    \begin{itemize}
        \item Let $\T$ be the species of trees i.e. $\T[V]$ is the free vector space over abstract 
        trees with vertex set $V$. Then the species $\point{\T}$ is the species of rooted tree. We distinguish the root vertex 
        by changing its shape. For example, the element $\left(\set{\set{a,c},\set{c,b}},c\right)\in \point{\T}[\set{a,b,c}]$ 
        will be represented as follows:
        \begin{equation}
            \begin{tikzpicture}[Centering,scale=.7]
                \node[RootGraph](s)at(0,0){$c$};
                \node[NodeGraph](a)at(1,1){$a$};
                \node[NodeGraph](b)at(1,-1){$b$};
                \draw[EdgeGraph](a)--(s);
                \draw[EdgeGraph](s)--(b);
            \end{tikzpicture}
            \enspace = \enspace\left(
            \begin{tikzpicture}[Centering,scale=.7]
                \node[NodeGraph](s)at(0,0){$c$};
                \node[NodeGraph](a)at(1,1){$a$};
                \node[NodeGraph](b)at(1,-1){$b$};
                \draw[EdgeGraph](a)--(s);
                \draw[EdgeGraph](s)--(b);
            \end{tikzpicture}
            \enspace , \enspace c\right)
        \end{equation}
        \item Since $\Sym[V]$ is always reduced to a space of dimension 1, the species $\Sym(\spe{S})$ can be interpreted as
        the species of assemblies of $\spe{S}$.
    \end{itemize}
\end{example}

A particularly interesting species we can construct using the composition of positive species is the species
of Schröder trees enriched with another species.
For $\spe{S}$ a positive species such that $\spe{S} = \spe{X} + \spe{S}_{2^+}$,
the species $\stree{\spe{S}}$ of {\em Schröder trees enriched with $\spe{S}$} is defined as the species satisfying the equation
\begin{math}
    \stree{\spe{S}} = \spe{X} + \spe{S}_{2^+}(\stree{\spe{S}}).
\end{math}
The vector space $\stree{\spe{S}}[V]$ is then generated by abstract rooted trees with internal vertices labelled by elements of $\spe{S}$ and set of leaves $V$. 
In particular, remark that the internal node have a linear behaviour: if $W$ is a subset of $V$, $x$ and $y$ are two elements of $\spe{S}[W]$ and 
$k$ is an element of $\K$, a tree $t$ with a node labelled by $x+ky$ is equal to the sum $t_1+kt_2$ where $t_1$ and $t_2$ are identical to $t$ except with
the node which was labelled by $x+ky$ in $t$ being respectively labelled by $x$ and $y$.
In this context, the elements of $\spe{S}$ are identified with the corollas of $\stree{\spe{S}}$.

\begin{example}\label{ex_stree}
    Here is an example of an element in $\stree{\spe{S}}[[4]]$ where $\spe{S}[V] = \K{x}$ if $\card{V} = 3$, $\spe{S}[V] = \K{y}$ if $\card{V}=2$ and 
    $\spe{S}[V] = \set{0}$ otherwise.
    \begin{equation}
        \begin{tikzpicture}[Centering, scale=0.6]
            \node[NodeFree, minimum size=8mm](r)at(0,0){$x_{12,3,4}$};
            \node[NodeFree, minimum size=8mm](n)at(-1.5,2.75){$y_{1,2}$};
            \node[NodeFree, minimum size=8mm](s)at(0.25,2.5){$4$};
            \node[NodeFree, minimum size=8mm](3)at(1.75,1.75){$3$};
            \node[NodeFree, minimum size=8mm](1)at(-3.5,5){$1$};
            \node[NodeFree, minimum size=8mm](2)at(-1.5,5.5){$2$};
            \draw[EdgeFree](r)--(n);
            \draw[EdgeFree](r)--(s);
            \draw[EdgeFree](r)--(3);
            \draw[EdgeFree](n)--(1);
            \draw[EdgeFree](n)--(2);
        \end{tikzpicture}
    \end{equation}
\end{example}
%


\subsection{Definitions and background on operads}\label{sec_operads}
We give here basic definitions and results of the theory as well as some classical examples.
We refer the reader to~\cite{Men15} and \cite{LV12} for a more general approach to the theory of operads.

A {\em partial composition} of a species $\spe{S}$ is a collection of linear maps 
$\comp[v]: \spe{S}[V_1]\otimes\spe{S}[V_2]\to \spe{S}[V_1+V_2-\set{v}]$, for $V_1$ and $V_2$ disjoint finite
sets and $v\in V_1$. These maps must be natural: for $\sigma_1$ and $\sigma_2$ bijections with respective domain $V_1$ and $V_2$, we have
$\comp[\sigma_1(v)]\circ\spe{S}[\sigma_1]\otimes\spe{S}[\sigma_2] = \spe{S}[\sigma_1\circ\sigma_2]\circ\comp[v]$. In particular, a partial composition induces a morphism of linear species $\comp:\spe{S}'\cdot\spe{S}\to\spe{S}$, from which we can recover all the maps of the collection by naturality, and
hence can be defined as such a morphism.
\begin{definition}
A {\em symmetric linear operad} is a positive linear species $\Operad$ equipped with an {\em unity} $e: \spe{X}\to \Operad$
and a  {\em partial composition} $\comp:\Operad'\cdot\Operad\to \Operad$ such that the following diagrams commute
\begin{center}\label{def_op}
    \begin{tikzcd}
        \Operad''\cdot\Operad^2 \arrow[r, "{\comp[\ast_1]}"] \arrow[d, "{\comp[\ast_2]}\circ\id\cdot\twist"] & \Operad'\cdot\Operad \arrow[d, "{\comp[\ast_2]}"] & &
        \Operad'\cdot \Operad'\cdot\Operad \arrow[r, "{\comp[\ast_1]}\cdot\id"] \arrow[d, "\id\cdot{\comp[\ast_2]}"] & \Operad'\cdot\Operad \arrow[d, "{\comp[\ast_2]}"] \\
        \Operad'\cdot \Operad \arrow[r, "{\comp[\ast_1]}"] & \Operad & & 
        \Operad'\cdot \Operad \arrow[r, "{\comp[\ast_1]}"] & \Operad
    \end{tikzcd}
\end{center}
\begin{center}
    \begin{tikzcd}
        \Operad'\cdot \K X \arrow[r, "\Operad'\cdot e"] \arrow[rd, "p"] & \Operad'\cdot\Operad \arrow[d, "\comp"] & \K X'\cdot \Operad \arrow[l, "e'\cdot \Operad" swap]\arrow[dl, "\cong" swap]\\
        & \Operad 
    \end{tikzcd}
\end{center}
where $\tau:x\otimes y\mapsto y\otimes x$ and $p_V: x \mapsto \Operad[\sigma](x)$ with $\sigma$ the
bijection that sends $\ast$ on $v$ and is the identity on $V\setminus\{v\}$.
\end{definition}

%


A {\em sub-operad} of an operad $\Operad$ is a sub-species of $\Operad$ containing the image of $e$ and stable under
partial composition. For $\Operad$ an operad, the sub-operad of $\Operad$ {\em generated by} a sub-species $\spe{S}$ is 
the smallest sub-operad of $\Operad$ containing $\spe{S}$ and the sub-operad generated by a set $E$ of elements of $\Operad$ is the
sub-operad generated by the species generated by $E$. A {\em morphism of operads} $f:\Operad_1\to\Operad_2$ is a morphism of species stable under the 
structure maps: $f\circ e = e$ and $f(x\comp y) =f(x)\comp f(y)$.

In practice the map $e$ is often trivial and we do not mention it. Let us now give a series of examples of operads.

\begin{description}\label{ex_operad}
    \item[Identity.] The identity species has a natural operad structure given by $v_1\comp v_2 = v_1$ if $v_1=\ast$ and
    $v_1\comp v_2=v_2$ else.
    \item[Com.] The species $\Sym_+$ has a natural operad structure given by $(V_1+\set{\ast})\comp V_2 = V_1+V_2$.
    This operad is called the {\em commutative operad} and denoted by $\Com$. 
    In this context we denote by $\mu_V=V$ the basis element of $\Com[V]$.
    \item[NAP.] Let be $t_1$ and $t_2$ be two rooted trees and let $t_1\comp t_2$ be the rooted tree obtained by the following operation.
    \begin{enumerate}
        \item Take the disjoint union of $t_1$ and $t_2$

        \item Remove the vertex $\ast$ from $t_1$.

        \item Add an edge between the neighbours of $\ast$ in $t_1$ and the root of $t_2$.
    \end{enumerate}
    This operation is a partial composition and turns the species $\point{\T}$ of rooted trees into an operad. This operad is called {\em non associative
    permutative operad}~\cite{Liv06} and is denoted by $\NAP$. For instance we have:
    \begin{equation}
        \begin{tikzpicture}[Centering,scale=.7]
            \node[NodeGraph](s)at(0,0){$\ast$};
            \node[RootGraph](a)at(1,1){$a$};
            \node[NodeGraph](b)at(1,-1){$b$};
            \draw[EdgeGraph](a)--(s);
            \draw[EdgeGraph](s)--(b);
        \end{tikzpicture}
        \enspace \comp \enspace
        \begin{tikzpicture}[Centering,scale=.7]
            \node[RootGraph](c)at(0,0){$c$};
            \node[NodeGraph](d)at(1,0){$d$};
            \draw[EdgeGraph](c)--(d);
        \end{tikzpicture}
        \enspace = \enspace
        \begin{tikzpicture}[Centering,scale=.7]
            \node[RootGraph](a)at(1,1){$a$};
            \node[NodeGraph](b)at(1,-1){$b$};
            \node[NodeGraph](c)at(0,0){$c$};
            \node[NodeGraph](d)at(2,0){$d$};
            \draw[EdgeGraph](a)--(c);
            \draw[EdgeGraph](c)--(b);
            \draw[EdgeGraph](c)--(d);
        \end{tikzpicture}.
    \end{equation}
    \item[PreLie.]\label{ex_plie} Let be $t_1$ and $t_2$ be two rooted trees and let $t_1\comp t_2$ be the sum over all rooted trees obtained by the following 
    operation.
    \begin{enumerate}
        \item Take the disjoint union of $t_1$ and $t_2$.

        \item Remove the vertex $\ast$ from $t_1$.

        \item Add an edge between the parent of $\ast$ in $t_1$ and the root of $t_2$.

        \item For each child of $\ast$ in $t_1$, add an edge between this vertex and any vertex of $t_2$.
    \end{enumerate}
    This operation  is a partial composition and turns the species $\point{\T}$ into an operad. This operad is called {\em pre-Lie operad}~\cite{CL01}
    and is denoted by $\PLie$. Remark that the partial composition of $t_1$ and $t_2$ as elements of $\NAP$ is always in the support of the
    partial composition of $t_1$ and $t_2$ as elements of $\PLie$. For instance, we have:
    \begin{equation}
        \begin{tikzpicture}[Centering,scale=.7]
            \node[NodeGraph](s)at(0,0){$\ast$};
            \node[RootGraph](a)at(1,1){$a$};
            \node[NodeGraph](b)at(1,-1){$b$};
            \draw[EdgeGraph](a)--(s);
            \draw[EdgeGraph](s)--(b);
        \end{tikzpicture}
        \enspace \comp \enspace
        \begin{tikzpicture}[Centering,scale=.7]
            \node[RootGraph](c)at(0,0){$c$};
            \node[NodeGraph](d)at(1,0){$d$};
            \draw[EdgeGraph](c)--(d);
        \end{tikzpicture}
        \enspace = \enspace
        \begin{tikzpicture}[Centering,scale=.7]
            \node[RootGraph](a)at(1,1){$a$};
            \node[NodeGraph](b)at(1,-1){$b$};
            \node[NodeGraph](c)at(0,0){$c$};
            \node[NodeGraph](d)at(2,0){$d$};
            \draw[EdgeGraph](a)--(c);
            \draw[EdgeGraph](c)--(b);
            \draw[EdgeGraph](c)--(d);
        \end{tikzpicture}
        \enspace + \enspace
        \begin{tikzpicture}[Centering,scale=.7]
            \node[RootGraph](a)at(1,1){$a$};
            \node[NodeGraph](b)at(1,-1){$b$};
            \node[NodeGraph](c)at(0,0){$c$};
            \node[NodeGraph](d)at(2,0){$d$};
            \draw[EdgeGraph](a)--(c);
            \draw[EdgeGraph](d)--(b);
            \draw[EdgeGraph](c)--(d);
        \end{tikzpicture}.
    \end{equation}
     \item[Polynomials.] The species $\K\Poly$ has a natural partial composition given by the composition of polynomials: for 
    $p_1(v_1,\dots,v_k,\ast)$ and $p_2(v_1',\dots,v_l')$ two polynomials over disjoint sets of variables,
    \begin{equation}\label{pol_compositioin}
        (p_1\comp p_2)(v_1,\dots,v_k,v_1',\dots,v_l') = p_1|_{\ast \take p_2}
        = p_1(v_1,\dots, v_{k},p_2(v_1',\dots,v_l')).
    \end{equation}
    One can directly check that this partial composition satisfies the commutative diagrams of Definition~\ref{def_op}. This turns $\K\Poly$ into 
    an operad where the units are the singleton polynomials $v\in\Poly[\set{v}]$. Remark that since $\comp$ is a linear map, the product and
    addition of polynomials act as bilinear maps.
    Restricting the four species morphisms from Example~\ref{ex_species} to $\Sym_+$ and $\K\Poly$ when necessary makes them operad morphisms.
    \item[Hadamard product.] If $\Operad_1$ and $\Operad_2$ are two operads, then the species $\Operad_1\times\Operad_2$ is also an operad with partial composition:
    $(x_1\otimes x_2)\comp (y_1\otimes y_2) =(x_1\comp y_1)\otimes (x_2\comp y_2)$. 
    \item[Assemblies.] For $\Operad$ an operad, the species $\Com(\Operad)$ of assemblies of $\Operad$ has a natural operad structure. Let $V_1,\dots V_n$ and
    $W_1,\dots, W_k$ be $n+k$ disjoint sets such that $\ast\in V_1$. Let be $x_i\in \Operad[V_i]$ for $1\leq i\leq n$ and $y_j\in\Operad[W_j]$ for $1\leq j\leq k$.
    Then the partial composition of the assemblies $x_1\otimes\dots\otimes x_n$ and $y_1\otimes\dots\otimes y_k$ is defined by
    $$(x_1\otimes\dots\otimes x_n) \comp (y_1\otimes\dots\otimes y_k) 
    = \sum_{i=1}^k y_1\otimes\dots\otimes y_{i-1}\otimes x_1\comp y_i \otimes y_{i+1}\otimes\dots\otimes y_k \otimes x_2\otimes\dots\otimes x_n.$$
%
\end{description}

In all the above examples, we defined operads by explicitly describing the vector spaces $\Operad[V]$ and the partial compositions. There is another way to
present an operad which is as a quotient of a free operad.

Let $\spe{S}$ be a positive linear species such that $\spe{S}=\spe{X}+\spe{S}_{2^+}$. 
Recall that $\stree{\spe{S}}[V]$ is generated by trees with internal vertices decorated with elements of $\spe{S}$ and set of leaves $V$.
The species $\stree{\spe{S}}$ has then a natural operad structure given by the grafting of trees.
For $t_1\in\stree{\spe{S}}'[V_1]$ and $t_2\in\stree{\spe{S}}[V_2]$, the partial composition $t_1\comp t_2$ is the tree obtained by grafting $t_2$ on the leaf
$\ast$ of $t_1$ and relabeling the nodes of $t_1$ accordingly. 
This operad is called {\em free operad over $\spe{S}$} and we denote it by $\free{\spe{S}}$.
\begin{example}
We give here an example of partial composition in a free operad over the same species than in Example~\ref{ex_stree}.
    \begin{equation}
        \begin{tikzpicture}[Centering, scale=0.6]
            \node[NodeFree, minimum size=8mm](r)at(0,0){$x_{12,\ast,3}$};
            \node[NodeFree, minimum size=8mm](n)at(-1.5,2.75){$y_{1,2}$};
            \node[NodeFree, minimum size=8mm](s)at(0.25,2.5){$\ast$};
            \node[NodeFree, minimum size=8mm](3)at(1.75,1.75){$3$};
            \node[NodeFree, minimum size=8mm](1)at(-3.5,5){$1$};
            \node[NodeFree, minimum size=8mm](2)at(-1.5,5.5){$2$};
            \draw[EdgeFree](r)--(n);
            \draw[EdgeFree](r)--(s);
            \draw[EdgeFree](r)--(3);
            \draw[EdgeFree](n)--(1);
            \draw[EdgeFree](n)--(2);
        \end{tikzpicture}
        \enspace \comp \enspace
        \begin{tikzpicture}[Centering, scale=0.6]
            \node[NodeFree, minimum size=8mm](r)at(0,0){$y_{4,5}$};
            \node[NodeFree, minimum size=8mm](4)at(-1.5,2.75){$4$};
            \node[NodeFree, minimum size=8mm](5)at(1.75,1.75){$5$};
            \draw[EdgeFree](r)--(4);
            \draw[EdgeFree](r)--(5);
        \end{tikzpicture}
        \enspace = \enspace
        \begin{tikzpicture}[Centering, scale=0.6]
            \node[NodeFree, minimum size=8mm](r)at(0,0){$x_{12,45,3}$};
            \node[NodeFree, minimum size=8mm](n)at(-1.5,2.75){$y_{1,2}$};
            \node[NodeFree, minimum size=8mm](s)at(0.25,2.5){$y_{4,5}$};
            \node[NodeFree, minimum size=8mm](4)at(0,5.25){$4$};
            \node[NodeFree, minimum size=8mm](5)at(2,4){$5$};
            \node[NodeFree, minimum size=8mm](3)at(1.75,1.75){$3$};
            \node[NodeFree, minimum size=8mm](1)at(-3.5,5){$1$};
            \node[NodeFree, minimum size=8mm](2)at(-1.5,5.5){$2$};
            \draw[EdgeFree](r)--(n);
            \draw[EdgeFree](r)--(s);
            \draw[EdgeFree](r)--(3);
            \draw[EdgeFree](n)--(1);
            \draw[EdgeFree](n)--(2);
            \draw[EdgeFree](s)--(4);
            \draw[EdgeFree](s)--(5);
        \end{tikzpicture}
    \end{equation}
\end{example}
%

In the sequel, we consider free operads over species which are sub-species of an operad $\Operad$. When this happens, we denote by $\comp^\xi$ the partial
composition in the free operad in order to not confuse it with the partial composition in $\Operad$.

\begin{example}[$\ComMag$]
    The free operad over one symmetric generator $\free{\Sym_2}$ is the operad of abstract binary trees with partial
    composition the grafting of trees. This operad is called {\em commutative magmatic operad}~\cite{BL11} and is denoted by $\ComMag$. 
    In this context, for $V$ a set of size $2$, we denote by $s_V$ the generating element of $\Sym_2[V]$: $\Sym_2[V]=\free{\Sym_2}[V] = \K s_V$.
\end{example}

An {\em ideal} of an operad $\Operad$ is a sub-species $\mathcal{I}$ such that the image of the products $\Operad'\cdot \mathcal{I}$ and $\mathcal{I}'\cdot\Operad$ 
by the partial composition maps are in $\mathcal{I}$. The {\em quotient species} $\Operad/ \mathcal{I}$ defined by $(\Operad/\mathcal{I})[V] = \Operad[V]/\mathcal{I}[V]$ 
is then an operad with the natural partial composition and unit : $\iso{x}\comp\iso{y}=\iso{x\comp y}$ where $\iso{x}$ is the equivalence class of $x$.
For $\mathcal{G}$ a species, if $\mathcal{R}$ is a sub-species of $\free{\mathcal{G}}$, we denote by $(\mathcal{R})$ the smallest ideal of 
$\free{\mathcal{G}}$ containing $\mathcal{R}$ and write that $(\mathcal{R})$ is {\em generated by $\mathcal{R}$}.

Denote by $\free{\mathcal{G}}^{(2)}$ the sub-species of $\free{\mathcal{G}}$ of trees with two internal node.

\begin{definition}
    Let $\mathcal{G}$ be a species and $\mathcal{R}$ be a sub-species of $\free{\mathcal{G}}$. We denote by 
    $\Ope(\mathcal{G},\mathcal{R})=\free{\mathcal{G}}/(\mathcal{R})$ the operad {\em generated by $\mathcal{G}$ and with relation $\mathcal{R}$}. 
    The operad $\Ope(\mathcal{G},\mathcal{R})$ is {\em binary} if the species $\mathcal{G}$ of {\em generators} is concentrated in cardinality $2$ (i.e. 
    $\mathcal{G} = \mathcal{G}_2$). This operad is {\em quadratic} if the species $\mathcal{R}$ of {\em relations} is a sub-species of 
    $\free{\mathcal{G}}^{(2)}$.
\end{definition}

\begin{example}\label{ex_gen_com} 
    Denote by 
    $\mathcal{R}$ the sub-species of $\ComMag^{(2)}$ generated by the {\em associativity relation} 
    $s_{\set{a,\ast}}\comp^{\xi} s_{\set{b,c}} - s_{\set{c,\ast}}\comp^{\xi} s_{\set{a,b}}$. 
    Then $\Com=\Ope(\Sym_2,\mathcal{R})$ and $\Com$ is hence binary and quadratic. Remark that as a consequence we have that $\mu_V$ is the image of $s_V$
    under the projection $\ComMag\to\Com$.
\end{example}

 Two advantages of defining an operad by its generators and relations are
 that it is possible to construct (under some conditions) its Koszul dual and 
 that it is possible to check if the operad is Koszul.
These notions are too involved
to be presented in a simple reminder and we only refer to them in Proposition~\ref{sp_dual} and Proposition~\ref{sp_koszul}. We refer the reader not versed in 
operad theory to Appendix~\ref{ann_koszul} for more information on these. 

%


\subsection{Graphs and multigraphs}\label{intro_graphs}
A {\em graph} or {\em simple graph} over $V$ is a set of non ordered pairs of distinct elements of $V$ and a {\em multigraph} is a multiset of
non ordered pairs of elements of $V$. In this context, the elements of a (multi)graph are called {\em edges}, the elements of $V$ are called 
vertices and the vertices of an edge are called its {\em ends}. Graphs can be seen as multigraphs with at most one edge between two vertices and 
no {\em loop} edges, which are edges with equal ends. We denote by $\G$ the species of graphs and by $\MG$ the species of multigraphs. 
We have the usual definitions of {\em connectedness} and {\em restriction} to subset of vertices.
We respectively denote by $\G_c$ and by $\MG_c$ the species of connected graphs and connected multigraphs and by $g|_W\in\MG[W]$ the restriction of 
$g\in\MG[V]$ to $W\subseteq V$

The four species presented here as well as the species of trees $\T$ are related by the following diagram:
\begin{equation}\label{graph_diagram}
        \begin{tikzcd}[row sep=scriptsize,column sep=scriptsize]
        \T \arrow[r, hook] & \G_c \arrow[r, hook] \arrow[d, hook] & \MG_c \arrow[d, hook] \\
        & \G \arrow[r, hook] & \MG & 
    \end{tikzcd}
\end{equation}

An {\em orientation} of a multigraph $g$ is a pair of maps $\ori=(\texttt{s},\texttt{t})$ from $g$ to $\p(V)$ such that 
$e = \texttt{s}(e) \cup \texttt{t}(e)$ for any edge $e$ of $g$. The
maps $\texttt{s}$ and $\texttt{t}$ are respectively called {\em source} and {\em target}. An {\em oriented} multigraph is then a pair of a multigraph 
and an orientation. We often represent oriented multigraphs by adding arrow heads to the target ends. 
We add the index $or$ to the five preceding species to designate their oriented counterpart e.g.
$\G_{orc}$ is the species of oriented connected graphs.

Let us finish this presentation of graphs and multigraphs by mentioning that there is an monomorphism $\Phi$ from the species $\MG$ to $\K\spe{Pol}$. It is
defined as follows:
\begin{itemize}\label{graphpol}
    \item the empty graph $\emptyset_V\in \MG[V]$ is sent on the null polynomial $\Phi(\emptyset_V) = 0_V$;
    \item an edge $e=\set{v_1,v_2}$ is sent on the monomial $\Phi(e) = v_1v_2$;
    \item an element $g\in \MG[V]$ is sent on the polynomial $\Phi(g) = \bigoplus_{e\in g} \Phi(e)$.
\end{itemize}

This enables us to see the species of multigraphs as the sub-species $\K\spe{Pol}^2$ of homogeneous polynomials of degree 2. This will be useful in 
the following sections to do computations on multigraphs since it is easier to formally write operations, and in particular composition, on polynomials than 
on multigraphs.

\begin{example}
    With this identification, the following multigraph $g$ writes as the polynomial $\Phi(g) = a^2\oplus ab\oplus ad\oplus be\oplus ef\oplus 2df$.
    \begin{equation}
        \begin{tikzpicture}[Centering,scale=1]
            \tikzset{every loop/.style={}}
            \node[NodeGraph, minimum size = 5mm](d)at(0,0){$d$};
            \node[NodeGraph, minimum size = 5mm](f)at(1,0){$f$};
            \node[NodeGraph, minimum size = 5mm](e)at(1,1.5){$e$};
            \node[NodeGraph, minimum size = 5mm](b)at(0,1.5){$b$};
            \node[NodeGraph, minimum size = 5mm](a)at(-1,0.75){$a$};
            \draw[EdgeGraph](b)--(a);
            \draw[EdgeGraph](d)--(a);
            \draw[EdgeGraph](d)edge[bend right=40](f);
            \draw[EdgeGraph](d)edge[bend left=40](f);
            \draw[EdgeGraph](f)--(e);
            \draw[EdgeGraph](b)--(e);
            \draw[EdgeGraph](a)edge[loop above](a);
        \end{tikzpicture}
    \end{equation}
\end{example}


\section{Extending the pre-Lie operad}\label{sec_extending_plie}
As announced, we want to extend the pre-Lie operad structure to graphs. As we will see later, it is more natural to search an extension to multigraphs.
We recall from Example~\ref{ex_operad} how that the partial composition of $\PLie$ works: for $t_1$ and $t_2$ two rooted trees, $t_1\comp t_2$ the sum 
of all rooted tree obtained as follows:
\begin{enumerate}
        \item take the disjoint union of $t_1$ and $t_2$;

        \item remove the vertex $\ast$ from $t_1$;

        \item add an edge between the parent of $\ast$ in $t_1$ and the root of $t_2$;

        \item for each child of $\ast$ in $t_1$, add an edge between this vertex and any vertex of $t_2$.
\end{enumerate}


\subsection{Two canonical operads}\label{subsec_cano_op}

If we want to extends the above construction to multigraphs, we are faced with the problem that multigraphs do not have a root vertex.
Our first solution is then to decide that it is more natural to try to extend it
to the pointed multigraphs $\point{\MG}$. But this alone is not enough: indeed the partial composition of $\PLie$ also use the notion of parent and
child vertex/end which can not be defined on a general pointed multigraph. To replace this, we will consider oriented pointed multigraphs, where the
targets will play the same role than the parents ends in $\PLie$ and the sources the same role than the children ends.
Let then be $(g_1, v_1) \in (\point{\MG_{or}})'[V_1]$ and $(g_2, v_2) \in\point{\MG_{or}}[V_2]$ and define the partial composition $(g_1,v_1)\comp (g_2,v_2)$ 
as the sum over all elements obtained as follows:
\begin{enumerate}
    \item take the disjoint union of $g_1$ and $g_2$;

    \item remove the vertex $\ast$, we then have some edges with a loose end;

    \item connect each loose source end to $v_2$;

    \item connect each loose target end to any vertex in $V_2$;

    \item the new root is $v_1\comp v_2$, with the partial composition of the identity operad.
\end{enumerate}
For instance, we have:
\begin{equation}\label{canoor}
    \begin{tikzpicture}[Centering,scale=.7]
        \node[NodeGraph](s)at(0,0){$\ast$};
        \node[RootGraph](a)at(1,1){$a$};
        \node[NodeGraph](b)at(1,-1){$b$};
        \draw[ArcGraph](a)--(s);
        \draw[ArcGraph](s)--(b);
    \end{tikzpicture}
    \enspace \comp \enspace
    \begin{tikzpicture}[Centering,scale=1]
        \node[RootGraph](c)at(0,0){$c$};
        \node[NodeGraph](d)at(1,0){$d$};
        \draw[ArcGraph](c)--(d);
    \end{tikzpicture}
    \enspace = \enspace
    \begin{tikzpicture}[Centering,scale=.7]
        \node[RootGraph](a)at(1,1){$a$};
        \node[NodeGraph](b)at(1,-1){$b$};
        \node[NodeGraph](c)at(0,0){$c$};
        \node[NodeGraph](d)at(2,0){$d$};
        \draw[ArcGraph](a)--(c);
        \draw[ArcGraph](c)--(b);
        \draw[ArcGraph](c)--(d);
    \end{tikzpicture}
    \enspace + \enspace
    \begin{tikzpicture}[Centering,scale=.7]
        \node[RootGraph](a)at(1,1){$a$};
        \node[NodeGraph](b)at(1,-1){$b$};
        \node[NodeGraph](c)at(0,0){$c$};
        \node[NodeGraph](d)at(2,0){$d$};
        \draw[ArcGraph](a)--(d);
        \draw[ArcGraph](c)--(d);
        \draw[ArcGraph](c)--(b);
    \end{tikzpicture}.
\end{equation}

\begin{theorem}\label{th_op_graph_or}
    The species $\point{\MG_{or}}$, endowed with the preceding partial composition, is an operad.
\end{theorem}

We give a proof of this theorem in Section~\ref{sec_graph_operads} where we provide a general way to define operads on graphs and multigraphs.

It is straightforward to note that the subspecies of connected components $\point{\MG_{orc}}$ and the species $\point{\G_{or}}$ are sub-operads of $\point{\MG_{or}}$ and 
that $\point{\G_{orc}}$ is a sub-operad of $\point{\G_{or}}$. 
Given a rooted tree $t$ with root $r$, we have a natural orientation which consist of choosing the parent ends as targets and child ends as sources. 
We denote by $\ori_r=(\texttt{s}_r, \texttt{t}_r)$ this orientation and $t_r = (t, \ori_r)$ the associated oriented tree. This induces an operad monomorphism 
$(t,r)\mapsto(t_r,r)$ from $\PLie$ to $\point{\G_{orc}}$ and hence makes $\PLie$ a sub-operad of $\point{\G}_{orc}$.

\medskip

Our second solution in extending $\PLie$ is to ignore the third step of the construction of the partial composition which involves the root. The resulting
partial composition is then much more natural than the previous one over $\point{\MG}_{or}$:
let $g_1\in\MG'[V_1]$ and $g_2\in\MG[V_2]$ be two multigraphs and define the partial composition $g_1\comp g_2$ as the sum over all multigraphs obtained as follows:
\begin{enumerate}
        \item take the disjoint union of $g_1$ and $g_2$;

        \item remove the vertex $\ast$ from $g_1$;

        \item for each neighbour of $\ast$ in $g_1$, add an edge between this vertex and any vertex of $g_2$.
\end{enumerate}
For instance, we have:
\begin{equation}\begin{split}
    \begin{tikzpicture}[Centering,scale=1]
        \tikzset{every loop/.style={}}
        \node[NodeGraph](a)at(0,0){$a$};
        \node[NodeGraph](s)at(1,0){$\ast$};
        \draw[EdgeGraph](a)edge[bend left=40](s);
        \draw[EdgeGraph](a)edge[bend right=40](s);
        \draw[EdgeGraph](s)edge[loop above](s);
        \draw[EdgeGraph,draw=ColWhite](s)edge[loop below](s);
    \end{tikzpicture}
    \enspace \comp \enspace
    \begin{tikzpicture}[Centering,scale=1]
        \tikzset{every loop/.style={}}
        \node[NodeGraph](b)at(0,0){$b$};
        \node[NodeGraph](c)at(1,0){$c$};
        \draw[EdgeGraph](b)--(c);
        \draw[EdgeGraph](c)edge[loop above](c);
        \draw[EdgeGraph,draw=ColWhite](s)edge[loop below](s);
    \end{tikzpicture}
    & \enspace = \enspace
    \begin{tikzpicture}[Centering,scale=1]
        \tikzset{every loop/.style={}}
        \node[NodeGraph](a)at(0,0){$a$};
        \node[NodeGraph](b)at(1,0){$b$};
        \node[NodeGraph](c)at(2,0){$c$}; 
        \draw[EdgeGraph](a)edge[bend left=40](b);
        \draw[EdgeGraph](a)edge[bend right=40](b);
        \draw[EdgeGraph](b)edge[loop below](b);
        \draw[EdgeGraph](b)--(c);
        \draw[EdgeGraph](c)edge[loop above](c);
    \end{tikzpicture}
    \enspace + \enspace
    \begin{tikzpicture}[Centering,scale=1]
        \tikzset{every loop/.style={}}
        \node[NodeGraph](a)at(0,0){$a$};
        \node[NodeGraph](b)at(1,0){$b$};
        \node[NodeGraph](c)at(2,0){$c$};
        \draw[EdgeGraph](a)edge[bend left=40](b);
        \draw[EdgeGraph](a)edge[bend right=40](b);
        \draw[EdgeGraph](c)edge[loop below](c);
        \draw[EdgeGraph](b)--(c);
        \draw[EdgeGraph](c)edge[loop above](c);
    \end{tikzpicture}
    \enspace + \enspace
    2\,
    \begin{tikzpicture}[Centering,scale=1]
        \tikzset{every loop/.style={}}
        \node[NodeGraph](a)at(0,0){$a$};
        \node[NodeGraph](b)at(1,0){$b$};
        \node[NodeGraph](c)at(2,0){$c$};
        \draw[EdgeGraph](a)edge[bend left=40](b);
        \draw[EdgeGraph](a)edge[bend right=40](b);
        \draw[EdgeGraph](b)edge[bend left=40](c);
        \draw[EdgeGraph](b)edge[bend right=40](c);
        \draw[EdgeGraph](c)edge[loop above](c);
        \draw[EdgeGraph,draw=ColWhite](c)edge[loop below](c);
    \end{tikzpicture}
    \\
    & \quad + \enspace
    2\,
    \begin{tikzpicture}[Centering,scale=1]
        \tikzset{every loop/.style={}}
        \node[NodeGraph](a)at(0,0){$a$};
        \node[NodeGraph](b)at(1,0){$b$};
        \node[NodeGraph](c)at(2,0){$c$};
        \draw[EdgeGraph](a)--(b);
        \draw[EdgeGraph](a)edge[bend right=40](c);
        \draw[EdgeGraph](b)edge[loop above](b);
        \draw[EdgeGraph](b)--(c);
        \draw[EdgeGraph](c)edge[loop above](c);
    \end{tikzpicture}
    \enspace + \enspace
    2\,
    \begin{tikzpicture}[Centering,scale=1]
        \tikzset{every loop/.style={}}
        \node[NodeGraph](a)at(0,0){$a$};
        \node[NodeGraph](b)at(1,0){$b$};
        \node[NodeGraph](c)at(2,0){$c$};
        \draw[EdgeGraph](a)--(b);
        \draw[EdgeGraph](a)edge[bend right=40](c);
        \draw[EdgeGraph](c)edge[loop below](c);
        \draw[EdgeGraph](b)--(c);
        \draw[EdgeGraph](c)edge[loop above](c);
    \end{tikzpicture}
    \enspace + \enspace
    4\,
    \begin{tikzpicture}[Centering,scale=1]
        \tikzset{every loop/.style={}}
        \node[NodeGraph](a)at(0,0){$a$};
        \node[NodeGraph](b)at(1,0){$b$};
        \node[NodeGraph](c)at(2,0){$c$};
        \draw[EdgeGraph](a)--(b);
        \draw[EdgeGraph](a)edge[bend right=40](c);
        \draw[EdgeGraph](b)edge[bend right=40](c);
        \draw[EdgeGraph](b)edge[bend left=40](c);
        \draw[EdgeGraph](c)edge[loop above](c);
    \end{tikzpicture}
    \\
    & \quad + \enspace
    \begin{tikzpicture}[Centering,scale=1]
        \tikzset{every loop/.style={}}
        \node[NodeGraph](a)at(0,0){$a$};
        \node[NodeGraph](b)at(1,0){$b$};
        \node[NodeGraph](c)at(2,0){$c$};
        \draw[EdgeGraph](a)edge[bend right=40](c);
        \draw[EdgeGraph](a)edge[bend left=40](c);
        \draw[EdgeGraph](b)edge[loop left](b);
        \draw[EdgeGraph](b)--(c);
        \draw[EdgeGraph](c)edge[loop above](c);
    \end{tikzpicture}
    \enspace + \enspace
    \begin{tikzpicture}[Centering,scale=1]
        \tikzset{every loop/.style={}}
        \node[NodeGraph](a)at(0,0){$a$};
        \node[NodeGraph](b)at(1,0){$b$};
        \node[NodeGraph](c)at(2,0){$c$};
        \draw[EdgeGraph](a)edge[bend right=40](c);
        \draw[EdgeGraph](a)edge[bend left=40](c);
        \draw[EdgeGraph](c)edge[loop below](c);
        \draw[EdgeGraph](b)--(c);
        \draw[EdgeGraph](c)edge[loop above](c);
    \end{tikzpicture}
    \enspace + \enspace
    2\,
    \begin{tikzpicture}[Centering,scale=1]
        \tikzset{every loop/.style={}}
        \node[NodeGraph](a)at(0,0){$a$};
        \node[NodeGraph](b)at(1,0){$b$};
        \node[NodeGraph](c)at(2,0){$c$};
        \draw[EdgeGraph](a)edge[bend right=40](c);
        \draw[EdgeGraph](a)edge[bend left=40](c);
        \draw[EdgeGraph](b)edge[bend left=20](c);
        \draw[EdgeGraph](b)edge[bend right=20](c);
        \draw[EdgeGraph](c)edge[loop above](c);
    \end{tikzpicture}.
\end{split}\end{equation}

\begin{theorem} \label{th_graphop_cano}
    The species $\MG$ endowed with the preceding partial composition is an operad.
\end{theorem}

As for Theorem~\ref{th_op_graph_or}, we give a proof of this theorem in Section~\ref{sec_graph_operads}.

This operad structure makes all the species of the diagram~\ref{graph_diagram} operads and its maps operad monomorphisms. In particular, we recover 
the Kontsevich-Willwacher operad~\cite{MV19} on $\G$. Recall now from Section~\ref{sec_intro} that we can identify multigraphs with polynomials. The partial 
composition we just defined can then be formally written as $g_1\comp g_2 = g_1|_{\ast \take \sum V_2}\oplus g_2$ (using the same notation for the
composition of polynomials than in \eqref{pol_compositioin}), which can then be expanded as follows:
\begin{equation}\begin{split}\label{comp_multigraphs}
    g_1\comp g_2 &= g_1|_{\ast \take \sum V_2}\oplus g_2\\
    &= g_1|_{V_1} \oplus\bigoplus_{v\in n(\ast)} v(\sum V_2)\oplus((\sum V_2)^2)^{\oplus g_1(\ast\ast)}\oplus g_2 \\
    &= \sum_{f:n(\ast)\to V_2}\sum_{l:[g_1(\ast\ast)]\to V_2V_2} g_1|_{V_1}\oplus\bigoplus_{v\in n(\ast)}vf(v)\oplus\bigoplus_{i=1}^{g_1(\ast\ast)}l(i) \oplus g_2,
\end{split}\end{equation}
where $n(\ast)$ is the multiset of neighbours of $\ast$ in $g_1$ and $g_1(\ast\ast)$ is the number of loops on $\ast$ in $g_1$. Each of the three terms 
of the second line, without counting $g_2$, have a combinatorial interpretation: $g_1|_{V_1}$ is $g_1$ to which we removed $\ast$, $\bigoplus_{n\in n(\ast)} v(\sum V_2)$
can be understood as ``for all vertices in $n(\ast)$, sum over the ways of connecting it to $g_2$'' and the term $((\sum V_2)^2)^{\oplus g_1(\ast\ast)}$ as ``for
each loop over $\ast$, add an edge between any two elements of $V_2$''. This partial composition expands in a simpler way on $\G$ because of the absence of loops. 
Indeed, if $g_1$ and $g_2$ are now graphs, Equ.~\eqref{comp_multigraphs} rewrites as
\begin{equation}\begin{split}
    g_1\comp g_2 &= g_1|_{\ast \take \sum V_2}\oplus g_2\\
    &= g_1|_{V_1} \oplus\bigoplus_{n\in n(\ast)} v(\sum V_2)\oplus g_2 \\
    &= \sum_{f:n(\ast)\rightarrow V_2} g_1|_{V_1}\oplus\bigoplus_{v\in n(\ast)}vf(v)\oplus g_2.
\end{split}\end{equation}
For instance, we have:
\begin{equation}
    \begin{tikzpicture}[Centering,scale=.5]
        \node[NodeGraph](a)at(1,1){$a$};
        \node[NodeGraph](s)at(0,0){$\ast$};
        \node[NodeGraph](b)at(1,-1){$b$};
        \draw[EdgeGraph](a)--(s);
        \draw[EdgeGraph](s)--(b);
    \end{tikzpicture}
    \enspace \comp \enspace
    \begin{tikzpicture}[Centering,scale=.7]
        \node[NodeGraph](c)at(0,0){$c$};
        \node[NodeGraph](d)at(1,0){$d$};
        \draw[EdgeGraph](c)--(d);
    \end{tikzpicture}
    \enspace = \enspace
    \begin{tikzpicture}[Centering,scale=.5]
        \node[NodeGraph](a)at(1,1){$a$};
        \node[NodeGraph](b)at(1,-1){$b$};
        \node[NodeGraph](c)at(0,0){$c$};
        \node[NodeGraph](d)at(2,0){$d$};
        \draw[EdgeGraph](c)--(a);
        \draw[EdgeGraph](c)--(d);
        \draw[EdgeGraph](c)--(b);
    \end{tikzpicture}
    \enspace + \enspace
    \begin{tikzpicture}[Centering,scale=.5]
        \node[NodeGraph](a)at(1,1){$a$};
        \node[NodeGraph](b)at(1,-1){$b$};
        \node[NodeGraph](c)at(0,0){$c$};
        \node[NodeGraph](d)at(2,0){$d$};
        \draw[EdgeGraph](c)--(d);
        \draw[EdgeGraph](d)--(a);
        \draw[EdgeGraph](c)--(b);
    \end{tikzpicture}
    \enspace + \enspace
    \begin{tikzpicture}[Centering,scale=.5]
        \node[NodeGraph](a)at(1,1){$a$};
        \node[NodeGraph](b)at(1,-1){$b$};
        \node[NodeGraph](c)at(0,0){$c$};
        \node[NodeGraph](d)at(2,0){$d$};
        \draw[EdgeGraph](c)--(d);
        \draw[EdgeGraph](c)--(a);
        \draw[EdgeGraph](d)--(b);
    \end{tikzpicture}
    \enspace + \enspace
    \begin{tikzpicture}[Centering,scale=.5]
        \node[NodeGraph](a)at(1,1){$a$};
        \node[NodeGraph](b)at(1,-1){$b$};
        \node[NodeGraph](c)at(0,0){$c$};
        \node[NodeGraph](d)at(2,0){$d$};
        \draw[EdgeGraph](c)--(d);
        \draw[EdgeGraph](d)--(a);
        \draw[EdgeGraph](d)--(b);
    \end{tikzpicture}.
\end{equation}
In particular, we observe that all graphs appearing in $g_1 \comp g_2$ have $1$ as coefficient.


\subsection{Link with $\PLie$}
While $\PLie$ is a sub-operad of $\point{\MG}_{or}$, the operad structure on $\MG$ does not seems to keep any relation with $\PLie$. In fact, as we will see at the
end of this subsection, there is a non trivial link between the four operads $\point{\MG}_{or}$, $\MG$, $\PLie$, and the Kontsevich-Willwacher operad $\G$.
Let us begin with the following result which gives a link between the sub-operad $\T$, of $\MG$ and $\PLie$.

\begin{proposition} \label{prop_prelie}
    The monomorphism of species $\psi : \T \to \point{\T}$ defined by, for any tree $t \in \T[V]$,
    \begin{equation}
        \psi(t) =  \sum_{r \in V} (t, r),
    \end{equation}
    is a monomorphism of operads from $\T$ to $\PLie$.
\end{proposition}

Before giving the proof of this proposition, we illustrate it on an example:
\begin{equation}\begin{split}
    \psi&\left(\enspace
    \begin{tikzpicture}[Centering, scale=0.6]
        \node[NodeGraph, minimum size=3mm](a)at(0,0){$a$};
        \node[NodeGraph, minimum size=3mm](b)at(0.5,0.86){$b$};
        \node[NodeGraph, minimum size=3mm](s)at(1,0){$\ast$};
        \draw[EdgeGraph](a)--(b);
        \draw[EdgeGraph](b)--(s);
    \end{tikzpicture}
    \enspace\right) \comp
    \psi\left(\enspace
    \begin{tikzpicture}[Centering, scale=0.6]
        \node[NodeGraph, minimum size=3mm](c)at(0,1){$c$};
        \node[NodeGraph, minimum size=3mm](d)at(0,0){$d$};
        \draw[EdgeGraph](c)--(d);
    \end{tikzpicture}
    \enspace\right)=\left(\enspace
    \begin{tikzpicture}[Centering, scale=0.6]
        \node[RootGraph, minimum size=3mm](a)at(0,0){$a$};
        \node[NodeGraph, minimum size=3mm](b)at(0.5,0.86){$b$};
        \node[NodeGraph, minimum size=3mm](s)at(1,0){$\ast$};
        \draw[EdgeGraph](a)--(b);
        \draw[EdgeGraph](b)--(s);
    \end{tikzpicture}
    \enspace + \enspace
    \begin{tikzpicture}[Centering, scale=0.6]
        \node[NodeGraph, minimum size=3mm](a)at(0,0){$a$};
        \node[RootGraph, minimum size=3mm](b)at(0.5,0.86){$b$};
        \node[NodeGraph, minimum size=3mm](s)at(1,0){$\ast$};
        \draw[EdgeGraph](a)--(b);
        \draw[EdgeGraph](b)--(s);
    \end{tikzpicture}
    \enspace + \enspace
    \begin{tikzpicture}[Centering, scale=0.6]
        \node[NodeGraph, minimum size=3mm](a)at(0,0){$a$};
        \node[NodeGraph, minimum size=3mm](b)at(0.5,0.86){$b$};
        \node[RootGraph, minimum size=3mm](s)at(1,0){$\ast$};
        \draw[EdgeGraph](a)--(b);
        \draw[EdgeGraph](b)--(s);
    \end{tikzpicture}
    \enspace\right)\comp\left(\enspace
    \begin{tikzpicture}[Centering, scale=0.6]
        \node[RootGraph, minimum size=3mm](c)at(0,1){$c$};
        \node[NodeGraph, minimum size=3mm](d)at(0,0){$d$};
        \draw[EdgeGraph](c)--(d);
    \end{tikzpicture}
    \enspace + \enspace
    \begin{tikzpicture}[Centering, scale=0.6]
        \node[NodeGraph, minimum size=3mm](c)at(0,0){$c$};
        \node[RootGraph, minimum size=3mm](d)at(0,1){$d$};
        \draw[EdgeGraph](c)--(d);
    \end{tikzpicture}
    \enspace\right) \\
    &= 
    \begin{tikzpicture}[Centering, scale=0.6]
        \node[RootGraph, minimum size=3mm](a)at(0,0){$a$};
        \node[NodeGraph, minimum size=3mm](b)at(0.5,0.86){$b$};
        \node[NodeGraph, minimum size=3mm](s)at(1,0){$\ast$};
        \draw[EdgeGraph](a)--(b);
        \draw[EdgeGraph](b)--(s);
    \end{tikzpicture}
    \comp\left(\enspace
    \begin{tikzpicture}[Centering, scale=0.6]
        \node[RootGraph, minimum size=3mm](c)at(0,1){$c$};
        \node[NodeGraph, minimum size=3mm](d)at(0,0){$d$};
        \draw[EdgeGraph](c)--(d);
    \end{tikzpicture}
    \enspace + \enspace
    \begin{tikzpicture}[Centering, scale=0.6]
        \node[NodeGraph, minimum size=3mm](c)at(0,0){$c$};
        \node[RootGraph, minimum size=3mm](d)at(0,1){$d$};
        \draw[EdgeGraph](c)--(d);
    \end{tikzpicture}
    \enspace\right)
    \enspace + \enspace
    \begin{tikzpicture}[Centering, scale=0.6]
        \node[NodeGraph, minimum size=3mm](a)at(0,0){$a$};
        \node[RootGraph, minimum size=3mm](b)at(0.5,0.86){$b$};
        \node[NodeGraph, minimum size=3mm](s)at(1,0){$\ast$};
        \draw[EdgeGraph](a)--(b);
        \draw[EdgeGraph](b)--(s);
    \end{tikzpicture}
    \comp\left(\enspace
    \begin{tikzpicture}[Centering, scale=0.6]
        \node[RootGraph, minimum size=3mm](c)at(0,1){$c$};
        \node[NodeGraph, minimum size=3mm](d)at(0,0){$d$};
        \draw[EdgeGraph](c)--(d);
    \end{tikzpicture}
    \enspace + \enspace
    \begin{tikzpicture}[Centering, scale=0.6]
        \node[NodeGraph, minimum size=3mm](c)at(0,0){$c$};
        \node[RootGraph, minimum size=3mm](d)at(0,1){$d$};
        \draw[EdgeGraph](c)--(d);
    \end{tikzpicture}
    \enspace\right)
    \enspace + \enspace
    \begin{tikzpicture}[Centering, scale=0.6]
        \node[NodeGraph, minimum size=3mm](a)at(0,0){$a$};
        \node[NodeGraph, minimum size=3mm](b)at(0.5,0.86){$b$};
        \node[RootGraph, minimum size=3mm](s)at(1,0){$\ast$};
        \draw[EdgeGraph](a)--(b);
        \draw[EdgeGraph](b)--(s);
    \end{tikzpicture}
    \comp\left(\enspace
    \begin{tikzpicture}[Centering, scale=0.6]
        \node[RootGraph, minimum size=3mm](c)at(0,1){$c$};
        \node[NodeGraph, minimum size=3mm](d)at(0,0){$d$};
        \draw[EdgeGraph](c)--(d);
    \end{tikzpicture}
    \enspace + \enspace
    \begin{tikzpicture}[Centering, scale=0.6]
        \node[NodeGraph, minimum size=3mm](c)at(0,0){$c$};
        \node[RootGraph, minimum size=3mm](d)at(0,1){$d$};
        \draw[EdgeGraph](c)--(d);
    \end{tikzpicture}
    \enspace\right) \\
    &=
    \begin{tikzpicture}[Centering, scale=0.6]
        \node[RootGraph, minimum size=3mm](a)at(0,0){$a$};
        \node[NodeGraph, minimum size=3mm](b)at(0.5,0.86){$b$};
        \node[NodeGraph, minimum size=3mm](c)at(1,0){$c$};
        \draw[EdgeGraph](a)--(b);
        \draw[EdgeGraph](b)--(c);
        \node[NodeGraph, minimum size=3mm](d)at(1,-1){$d$};
        \draw[EdgeGraph](c)--(d);
    \end{tikzpicture}
    \enspace + \enspace
    \begin{tikzpicture}[Centering, scale=0.6]
        \node[RootGraph, minimum size=3mm](a)at(0,0){$a$};
        \node[NodeGraph, minimum size=3mm](b)at(0.5,0.86){$b$};
        \node[NodeGraph, minimum size=3mm](c)at(1,0){$d$};
        \draw[EdgeGraph](a)--(b);
        \draw[EdgeGraph](b)--(c);
        \node[NodeGraph, minimum size=3mm](d)at(1,-1){$c$};
        \draw[EdgeGraph](c)--(d);
    \end{tikzpicture}
    \enspace + \enspace
    \begin{tikzpicture}[Centering, scale=0.6]
        \node[NodeGraph, minimum size=3mm](a)at(0,0){$a$};
        \node[RootGraph, minimum size=3mm](b)at(0.5,0.86){$b$};
        \node[NodeGraph, minimum size=3mm](c)at(1,0){$c$};
        \draw[EdgeGraph](a)--(b);
        \draw[EdgeGraph](b)--(c);
        \node[NodeGraph, minimum size=3mm](d)at(1,-1){$d$};
        \draw[EdgeGraph](c)--(d);
    \end{tikzpicture}
    \enspace + \enspace
    \begin{tikzpicture}[Centering, scale=0.6]
        \node[NodeGraph, minimum size=3mm](a)at(0,0){$a$};
        \node[RootGraph, minimum size=3mm](b)at(0.5,0.86){$b$};
        \node[NodeGraph, minimum size=3mm](c)at(1,0){$d$};
        \draw[EdgeGraph](a)--(b);
        \draw[EdgeGraph](b)--(c);
        \node[NodeGraph, minimum size=3mm](d)at(1,-1){$c$};
        \draw[EdgeGraph](c)--(d);
    \end{tikzpicture}
    \enspace + \enspace
    \begin{tikzpicture}[Centering, scale=0.6]
        \node[NodeGraph, minimum size=3mm](a)at(0,0){$a$};
        \node[NodeGraph, minimum size=3mm](b)at(0.5,0.86){$b$};
        \node[RootGraph, minimum size=3mm](c)at(1,0){$c$};
        \draw[EdgeGraph](a)--(b);
        \draw[EdgeGraph](b)--(c);
        \node[NodeGraph, minimum size=3mm](d)at(1,-1){$d$};
        \draw[EdgeGraph](c)--(d);
    \end{tikzpicture}
    \enspace + \enspace
    \begin{tikzpicture}[Centering, scale=0.6]
        \node[NodeGraph, minimum size=3mm](a)at(0,0){$a$};
        \node[NodeGraph, minimum size=3mm](b)at(0.5,0.86){$b$};
        \node[NodeGraph, minimum size=3mm](c)at(1,0){$d$};
        \draw[EdgeGraph](a)--(b);
        \draw[EdgeGraph](b)--(c);
        \node[RootGraph, minimum size=3mm](d)at(1,-1){$c$};
        \draw[EdgeGraph](c)--(d);
    \end{tikzpicture}
    \enspace + \enspace
    \begin{tikzpicture}[Centering, scale=0.6]
        \node[NodeGraph, minimum size=3mm](a)at(0,0){$a$};
        \node[NodeGraph, minimum size=3mm](b)at(0.5,0.86){$b$};
        \node[RootGraph, minimum size=3mm](c)at(1,0){$d$};
        \draw[EdgeGraph](a)--(b);
        \draw[EdgeGraph](b)--(c);
        \node[NodeGraph, minimum size=3mm](d)at(1,-1){$c$};
        \draw[EdgeGraph](c)--(d);
    \end{tikzpicture}
    \enspace + \enspace
    \begin{tikzpicture}[Centering, scale=0.6]
        \node[NodeGraph, minimum size=3mm](a)at(0,0){$a$};
        \node[NodeGraph, minimum size=3mm](b)at(0.5,0.86){$b$};
        \node[NodeGraph, minimum size=3mm](c)at(1,0){$c$};
        \draw[EdgeGraph](a)--(b);
        \draw[EdgeGraph](b)--(c);
        \node[RootGraph, minimum size=3mm](d)at(1,-1){$d$};
        \draw[EdgeGraph](c)--(d);
    \end{tikzpicture} \\
    &=\left(\enspace
    \begin{tikzpicture}[Centering, scale=0.6]
        \node[RootGraph, minimum size=3mm](a)at(0,0){$a$};
        \node[NodeGraph, minimum size=3mm](b)at(0.5,0.86){$b$};
        \node[NodeGraph, minimum size=3mm](c)at(1,0){$c$};
        \draw[EdgeGraph](a)--(b);
        \draw[EdgeGraph](b)--(c);
        \node[NodeGraph, minimum size=3mm](d)at(1,-1){$d$};
        \draw[EdgeGraph](c)--(d);
    \end{tikzpicture}
    \enspace + \enspace
    \begin{tikzpicture}[Centering, scale=0.6]
        \node[NodeGraph, minimum size=3mm](a)at(0,0){$a$};
        \node[RootGraph, minimum size=3mm](b)at(0.5,0.86){$b$};
        \node[NodeGraph, minimum size=3mm](c)at(1,0){$c$};
        \draw[EdgeGraph](a)--(b);
        \draw[EdgeGraph](b)--(c);
        \node[NodeGraph, minimum size=3mm](d)at(1,-1){$d$};
        \draw[EdgeGraph](c)--(d);
    \end{tikzpicture}
    \enspace + \enspace
    \begin{tikzpicture}[Centering, scale=0.6]
        \node[NodeGraph, minimum size=3mm](a)at(0,0){$a$};
        \node[NodeGraph, minimum size=3mm](b)at(0.5,0.86){$b$};
        \node[RootGraph, minimum size=3mm](c)at(1,0){$c$};
        \draw[EdgeGraph](a)--(b);
        \draw[EdgeGraph](b)--(c);
        \node[NodeGraph, minimum size=3mm](d)at(1,-1){$d$};
        \draw[EdgeGraph](c)--(d);
    \end{tikzpicture}
    \enspace + \enspace
    \begin{tikzpicture}[Centering, scale=0.6]
        \node[NodeGraph, minimum size=3mm](a)at(0,0){$a$};
        \node[NodeGraph, minimum size=3mm](b)at(0.5,0.86){$b$};
        \node[NodeGraph, minimum size=3mm](c)at(1,0){$c$};
        \draw[EdgeGraph](a)--(b);
        \draw[EdgeGraph](b)--(c);
        \node[RootGraph, minimum size=3mm](d)at(1,-1){$d$};
        \draw[EdgeGraph](c)--(d);
    \end{tikzpicture}
    \enspace\right)
    +
    \left(\enspace
    \begin{tikzpicture}[Centering, scale=0.6]
        \node[RootGraph, minimum size=3mm](a)at(0,0){$a$};
        \node[NodeGraph, minimum size=3mm](b)at(0.5,0.86){$b$};
        \node[NodeGraph, minimum size=3mm](c)at(1,0){$d$};
        \draw[EdgeGraph](a)--(b);
        \draw[EdgeGraph](b)--(c);
        \node[NodeGraph, minimum size=3mm](d)at(1,-1){$c$};
        \draw[EdgeGraph](c)--(d);
    \end{tikzpicture}
    \enspace + \enspace
    \begin{tikzpicture}[Centering, scale=0.6]
        \node[NodeGraph, minimum size=3mm](a)at(0,0){$a$};
        \node[RootGraph, minimum size=3mm](b)at(0.5,0.86){$b$};
        \node[NodeGraph, minimum size=3mm](c)at(1,0){$d$};
        \draw[EdgeGraph](a)--(b);
        \draw[EdgeGraph](b)--(c);
        \node[NodeGraph, minimum size=3mm](d)at(1,-1){$c$};
        \draw[EdgeGraph](c)--(d);
    \end{tikzpicture}
    \enspace + \enspace
    \begin{tikzpicture}[Centering, scale=0.6]
        \node[NodeGraph, minimum size=3mm](a)at(0,0){$a$};
        \node[NodeGraph, minimum size=3mm](b)at(0.5,0.86){$b$};
        \node[RootGraph, minimum size=3mm](c)at(1,0){$d$};
        \draw[EdgeGraph](a)--(b);
        \draw[EdgeGraph](b)--(c);
        \node[NodeGraph, minimum size=3mm](d)at(1,-1){$c$};
        \draw[EdgeGraph](c)--(d);
    \end{tikzpicture}
    \enspace + \enspace
    \begin{tikzpicture}[Centering, scale=0.6]
        \node[NodeGraph, minimum size=3mm](a)at(0,0){$a$};
        \node[NodeGraph, minimum size=3mm](b)at(0.5,0.86){$b$};
        \node[NodeGraph, minimum size=3mm](c)at(1,0){$d$};
        \draw[EdgeGraph](a)--(b);
        \draw[EdgeGraph](b)--(c);
        \node[RootGraph, minimum size=3mm](d)at(1,-1){$c$};
        \draw[EdgeGraph](c)--(d);
    \end{tikzpicture}
    \enspace\right)\\
    &= \psi\left(\enspace
        \begin{tikzpicture}[Centering, scale=0.6]
        \node[NodeGraph, minimum size=3mm](a)at(0,0){$a$};
        \node[NodeGraph, minimum size=3mm](b)at(0.5,0.86){$b$};
        \node[NodeGraph, minimum size=3mm](c)at(1,0){$c$};
        \draw[EdgeGraph](a)--(b);
        \draw[EdgeGraph](b)--(c);
        \node[NodeGraph, minimum size=3mm](d)at(1,-1){$d$};
        \draw[EdgeGraph](c)--(d);
    \end{tikzpicture}
    \enspace + \enspace
    \begin{tikzpicture}[Centering, scale=0.6]
        \node[NodeGraph, minimum size=3mm](a)at(0,0){$a$};
        \node[NodeGraph, minimum size=3mm](b)at(0.5,0.86){$b$};
        \node[NodeGraph, minimum size=3mm](c)at(1,0){$d$};
        \draw[EdgeGraph](a)--(b);
        \draw[EdgeGraph](b)--(c);
        \node[NodeGraph, minimum size=3mm](d)at(1,-1){$c$};
        \draw[EdgeGraph](c)--(d);
    \end{tikzpicture}
    \enspace\right) =
    \psi\left(\enspace
    \begin{tikzpicture}[Centering, scale=0.6]
        \node[NodeGraph, minimum size=3mm](a)at(0,0){$a$};
        \node[NodeGraph, minimum size=3mm](b)at(0.5,0.86){$b$};
        \node[NodeGraph, minimum size=3mm](s)at(1,0){$\ast$};
        \draw[EdgeGraph](a)--(b);
        \draw[EdgeGraph](b)--(s);
    \end{tikzpicture}
    \comp
    \begin{tikzpicture}[Centering, scale=0.6]
        \node[NodeGraph, minimum size=3mm](c)at(0,1){$c$};
        \node[NodeGraph, minimum size=3mm](d)at(0,0){$d$};
        \draw[EdgeGraph](c)--(d);
    \end{tikzpicture}
    \enspace\right).
\end{split}\end{equation}
We can note that the case when $\ast$ is the root plays a particular role.

\begin{proof}
    For $t\in \T[V]$ a tree and $r,v\in V$, we denote by $n_t(v)$ the set of neighbours of $v$ in $t$, by $c_{t,r}(v)$ the set of children of $v$ in the rooted tree 
    $(t,r)$ and if $r\not = v$, we further denote by $p_{t,r}(v)$ the parent of $v$ in~$(t,r)$.

    Let be $t_1\in \T'[V_1]$ and $t_2\in \T[V_2]$. We now make full use of the correspondence between graphs and polynomials:
    \begin{equation}\begin{split}
        \psi_{V_1}(t_1)\comp \psi_{V_2}(t_2) 
        &= \sum_{r_1\in V_1+\set{\ast}} (t_1,r_1) \comp \sum_{r_2\in V_2}(t_2,r_2) \\
        &= \sum_{r_1\in V_1+\set{\ast}}\sum_{r_2\in V_2} (t_1,r_1)\comp (t_2,r_2) \\ 
        &= \sum_{r_1\in V_1}\sum_{r_2\in V_2} \left(t_1|_{V_1}\oplus p_{t_1,r_1}(\ast)r_2
        \oplus t_2 \oplus \bigoplus_{v\in c_{t_1,r_1}(\ast)} v\left(\sum V_2\right), r_1\right) \\
        &\quad + \sum_{r_2\in V_2}\left(t_1|_{V_1}\oplus t_2 \oplus \bigoplus_{v\in c_{t_1,\ast}(\ast)} v\left(\sum V_2\right),r_2\right) \\
        &= \sum_{r_1\in V_1} \left(t_1|_{V_1}\oplus p_{t_1,r_1}(\ast)\left(\sum V_2\right) \oplus t_2 \oplus \bigoplus_{v\in c_{t_1,r_1}(\ast)}v\left(\sum V_2\right), r_1\right) \\
        &\quad+ \sum_{r_2\in V_2} \left(t_1|_{V_1}\oplus t_2 \oplus \bigoplus_{v\in c_{t_1,\ast}(\ast)} v\left(\sum V_2\right),r_2\right) \\
        &= \sum_{r\in V_1+V_2}\left(t_1|_{V_1}\oplus\bigoplus_{v\in n_{t_1}(\ast)}v\left(\sum V_2\right)\oplus t_2,r\right) \\
        &= \sum_{r\in V_1+V_2} \left(t_1|_{\ast \take \sum V_2}\oplus t_2,r\right) \\
        &= \psi_{V_1+V_2}(t_1\comp t_2).
    \end{split}\end{equation}
\end{proof}


The map $\psi$ naturally extends to a morphism of species $\MG_c\to \point{\MG}_c$. A natural question to ask is if it also possible to find an operad structure on
$\point{\MG}_c$ in order to make the following commutative diagram of species a commutative diagram of operads:
\begin{equation}\label{diagram_psi}
    \begin{tikzcd}[column sep=huge, ampersand replacement=\&]
        \T \arrow[r, hook, "\psi"] \arrow[d, hook] \& \PLie \arrow[d, hook] \\
        \MG_c \arrow[r, hook, "\psi"] \& \point{\MG}_c
    \end{tikzcd}.
\end{equation}
But as explained before, there does not seem to be a natural way to extend $\PLie$ to $\point{\MG}_c$ and we must rather consider $\point{\MG}_{orc}$, from which $\PLie$ is
indeed a sub-operad. By doing this we are now faced with the problem to find a natural way to embed $\point{\MG}_c$ in $\point{\MG}_{orc}$ which would make the species
$\psi(\MG_c)$ a sub-operad of $\point{\MG}_{orc}$ containing $\PLie$. This would then require to find a canonical orientation for each pointed multigraph compatible with
$\point{\MG}_{orc}$ operad structure, which again does not seems possible.

Fortunately, while it does not seems possible to make the diagram~\eqref{diagram_psi} a diagram of operads, we can obtain a similar result albeit with a more
involved diagram. To do this let us first introduce three new species. 
For $g\in \MG_c[V]$, $r\in V$, and $t\in \T[V]$ a spanning tree of $g$, we denote by $\ori_{t,r}$ the orientation defined as follows: the targets and sources of
the edges in $t$ are the same than in $t_r$ and the all the ends of the remaining edges are targets.

Let us define $\mathcal{I}\subset \Operad\subset\ST$ ($\ST$ standing for ``spanning tree'') three sub-species of $\point{\MG_{orc}}$ by
\begin{equation}
    \ST[V]=\K\left\{
        (g, \ori_{(t,r)},r) : g\in \MG_c[V], r\in V \text{ and $t$ a
    spanning tree of $g$}\right\},
\end{equation}
\begin{equation}
    \Operad[V] = \K\left\{\sum_{r\in V} (g, \ori_{(t(r),r)},r) : g\in
    \MG_c[V]\text{ and for each $r$, $t(r)$ a spanning tree of $g$}\right\},
\end{equation}
\begin{multline}
    \mathcal{I}[V]=
    \K\left\{(g, \ori_{(t_1,r)},r)-(g, \ori_{(t_2,r)},r) : g\in
    \MG_c[V],r\in V,
    \right. \\ \left.
    \text{ and $t_1$ and $t_2$ two spanning trees of $g$}\right\}.
\end{multline}

\begin{example}
    Let $V=\set{a,b,c,d}$. We give example of elements in $\ST[V]$, $\Operad[V]$ and $\mathcal{I}[V]$. 
    For the sake of an easier reading, instead of representing the orientations of the edges, we just colored the spanning trees in red. The blue edges should have
    both ends with an arrow shape and the red edges only the end directing to the root by only considering red dashed edges.
    \begin{equation}
        \begin{tikzpicture}[Centering,scale=1.2]
            \tikzset{every loop/.style={}}
            \node[RootGraph](a)at(0,0){$a$};
            \node[NodeGraph](b)at(0.5,0.5){$b$};
            \node[NodeGraph](c)at(1,0){$c$};
            \node[NodeGraph](d)at(0.5,-0.5){$d$};
            \draw[EdgeGraph, DarkRed, dashed](a)--(b);
            \draw[EdgeGraph, DarkRed, dashed](b)--(c);
            \draw[EdgeGraph](c)edge[bend left=20](d);
            \draw[EdgeGraph, DarkRed, dashed](c)edge[bend right=20](d);
            \draw[EdgeGraph](a)--(d);
            \draw[EdgeGraph](b)--(d);
        \end{tikzpicture}
        \enspace \in \ST[V]
    \end{equation}
    \begin{equation}
        \begin{tikzpicture}[Centering,scale=1.2]
            \tikzset{every loop/.style={}}
            \node[RootGraph](a)at(0,0){$a$};
            \node[NodeGraph](b)at(0.5,0.5){$b$};
            \node[NodeGraph](c)at(1,0){$c$};
            \node[NodeGraph](d)at(0.5,-0.5){$d$};
            \draw[EdgeGraph, DarkRed, dashed](a)--(b);
            \draw[EdgeGraph, DarkRed, dashed](b)--(c);
            \draw[EdgeGraph](c)edge[bend left=20](d);
            \draw[EdgeGraph, DarkRed, dashed](c)edge[bend right=20](d);
            \draw[EdgeGraph](a)--(d);
            \draw[EdgeGraph](b)--(d);
        \end{tikzpicture}
        \enspace + \enspace
        \begin{tikzpicture}[Centering,scale=1.2]
            \tikzset{every loop/.style={}}
            \node[NodeGraph](a)at(0,0){$a$};
            \node[RootGraph](b)at(0.5,0.5){$b$};
            \node[NodeGraph](c)at(1,0){$c$};
            \node[NodeGraph](d)at(0.5,-0.5){$d$};
            \draw[EdgeGraph, DarkRed, dashed](a)--(b);
            \draw[EdgeGraph](b)--(c);
            \draw[EdgeGraph](c)edge[bend left=20](d);
            \draw[EdgeGraph, DarkRed, dashed](c)edge[bend right=20](d);
            \draw[EdgeGraph](a)--(d);
            \draw[EdgeGraph, DarkRed, dashed](b)--(d);
        \end{tikzpicture}
        \enspace + \enspace
        \begin{tikzpicture}[Centering,scale=1.2]
            \tikzset{every loop/.style={}}
            \node[NodeGraph](a)at(0,0){$a$};
            \node[NodeGraph](b)at(0.5,0.5){$b$};
            \node[RootGraph](c)at(1,0){$c$};
            \node[NodeGraph](d)at(0.5,-0.5){$d$};
            \draw[EdgeGraph](a)--(b);
            \draw[EdgeGraph, DarkRed, dashed](b)--(c);
            \draw[EdgeGraph](c)edge[bend left=20](d);
            \draw[EdgeGraph](c)edge[bend right=20](d);
            \draw[EdgeGraph, DarkRed, dashed](a)--(d);
            \draw[EdgeGraph, DarkRed, dashed](b)--(d);
        \end{tikzpicture}
        \enspace + \enspace
        \begin{tikzpicture}[Centering,scale=1.2]
            \tikzset{every loop/.style={}}
            \node[NodeGraph](a)at(0,0){$a$};
            \node[NodeGraph](b)at(0.5,0.5){$b$};
            \node[NodeGraph](c)at(1,0){$c$};
            \node[RootGraph](d)at(0.5,-0.5){$d$};
            \draw[EdgeGraph, DarkRed, dashed](a)--(b);
            \draw[EdgeGraph, DarkRed, dashed](b)--(c);
            \draw[EdgeGraph, DarkRed, dashed](c)edge[bend left=20](d);
            \draw[EdgeGraph](c)edge[bend right=20](d);
            \draw[EdgeGraph](a)--(d);
            \draw[EdgeGraph](b)--(d);
        \end{tikzpicture}
        \enspace \in\Operad_1[V]
    \end{equation}
    \begin{equation}
        \begin{tikzpicture}[Centering,scale=1.2]
            \tikzset{every loop/.style={}}
            \node[NodeGraph](a)at(0,0){$a$};
            \node[RootGraph](b)at(0.5,0.5){$b$};
            \node[NodeGraph](c)at(1,0){$c$};
            \node[NodeGraph](d)at(0.5,-0.5){$d$};
            \draw[EdgeGraph, DarkRed, dashed](a)--(b);
            \draw[EdgeGraph, DarkRed, dashed](b)--(c);
            \draw[EdgeGraph](c)edge[bend left=20](d);
            \draw[EdgeGraph, DarkRed, dashed](c)edge[bend right=20](d);
            \draw[EdgeGraph](a)--(d);
            \draw[EdgeGraph](b)--(d);
        \end{tikzpicture}
        \enspace - \enspace
        \begin{tikzpicture}[Centering,scale=1.2]
            \tikzset{every loop/.style={}}
            \node[NodeGraph](a)at(0,0){$a$};
            \node[RootGraph](b)at(0.5,0.5){$b$};
            \node[NodeGraph](c)at(1,0){$c$};
            \node[NodeGraph](d)at(0.5,-0.5){$d$};
            \draw[EdgeGraph, DarkRed, dashed](a)--(b);
            \draw[EdgeGraph](b)--(c);
            \draw[EdgeGraph](c)edge[bend left=20](d);
            \draw[EdgeGraph, DarkRed, dashed](c)edge[bend right=20](d);
            \draw[EdgeGraph](a)--(d);
            \draw[EdgeGraph, DarkRed, dashed](b)--(d);
        \end{tikzpicture} \enspace 
        \in \Operad_2[V]
    \end{equation}

%
\end{example}

\begin{lemma} \label{lemma_op_fond}
    The following properties hold:
    \begin{enumerate}[label=(\roman*)]
        \item $\ST$ is a sub-operad of $\point{\MG}_{orc}$ isomorphic to $\MG\times\PLie$,
        \item $\Operad$ is a sub-operad of $\ST$, 
        \item $\mathcal{I}$ is an ideal of $\Operad$.
    \end{enumerate}
\end{lemma}

\begin{proof}
    Before proving these three items, we first give two equalities which will help us for the two last items. Let $U:\MG_{or} \rightarrow \MG$ 
    be the forgetful functor which sends an oriented graph on the graph obtained by forgetting the orientation. Let $g_1\in \MG_c'[V_1]$ and 
    $g_2\in \MG_c[V_2]$ be two connected multigraphs, $t$ a spanning tree of $g_1$ and for each $v\in V_2$, $t(v)$ a spanning tree 
    of $g_2$. When $\ast$ is the root of the spanning tree $t$, all the ends pointing to $\ast$ are targets. Since the target ends in an oriented multigraph 
    behave the same than the normal ends in a multigraph, the forgetful functor preserves the partial composition. For example we have: 
    \begin{equation}\begin{split}
        &U\times\id\left(\enspace
        \begin{tikzpicture}[Centering,scale=.6]
            \node[RootGraph, minimum size=3mm](s)at(0,0){$\ast$};
            \node[NodeGraph, minimum size=3mm](a)at(1,1){$a$};
            \node[NodeGraph, minimum size=3mm](b)at(1,-1){$b$};
            \draw[ArcGraph](a)--(s);
            \draw[ArcGraph](b)--(s);
        \end{tikzpicture}
        \enspace \comp \enspace
        \begin{tikzpicture}[Centering,scale=.6]
            \node[RootGraph, minimum size=3mm](c)at(0,0){$c$};
            \node[NodeGraph, minimum size=3mm](d)at(1,0){$d$};
            \draw[ArcGraph](d)--(c);
        \end{tikzpicture}
        \enspace\right)\\
        =&  U\times\id\left(\enspace
        \begin{tikzpicture}[Centering,scale=.6]
            \node[NodeGraph, minimum size=3mm](a)at(1,1){$a$};
            \node[NodeGraph, minimum size=3mm](b)at(1,-1){$b$};
            \node[RootGraph, minimum size=3mm](c)at(0,0){$c$};
            \node[NodeGraph, minimum size=3mm](d)at(2,0){$d$};
            \draw[ArcGraph](a)--(c);
            \draw[ArcGraph](b)--(c);
            \draw[ArcGraph](d)--(c);
        \end{tikzpicture}
        \enspace + \enspace
        \begin{tikzpicture}[Centering,scale=.6]
            \node[NodeGraph, minimum size=3mm](a)at(1,1){$a$};
            \node[NodeGraph, minimum size=3mm](b)at(1,-1){$b$};
            \node[RootGraph, minimum size=3mm](c)at(0,0){$c$};
            \node[NodeGraph, minimum size=3mm](d)at(2,0){$d$};
            \draw[ArcGraph](a)--(d);
            \draw[ArcGraph](d)--(c);
            \draw[ArcGraph](b)--(c);
        \end{tikzpicture}
        \enspace + \enspace
        \begin{tikzpicture}[Centering,scale=.6]
            \node[NodeGraph, minimum size=3mm](a)at(1,1){$a$};
            \node[NodeGraph, minimum size=3mm](b)at(1,-1){$b$};
            \node[RootGraph, minimum size=3mm](c)at(0,0){$c$};
            \node[NodeGraph, minimum size=3mm](d)at(2,0){$d$};
            \draw[ArcGraph](a)--(c);
            \draw[ArcGraph](d)--(c);
            \draw[ArcGraph](b)--(d);
        \end{tikzpicture}
        \enspace + \enspace
        \begin{tikzpicture}[Centering,scale=.6]
            \node[NodeGraph, minimum size=3mm](a)at(1,1){$a$};
            \node[NodeGraph, minimum size=3mm](b)at(1,-1){$b$};
            \node[RootGraph, minimum size=3mm](c)at(0,0){$c$};
            \node[NodeGraph, minimum size=3mm](d)at(2,0){$d$};
            \draw[ArcGraph](a)--(d);
            \draw[ArcGraph](d)--(c);
            \draw[ArcGraph](b)--(d);
        \end{tikzpicture}
        \enspace\right)\\
        =&\enspace
        \begin{tikzpicture}[Centering,scale=.6]
            \node[NodeGraph, minimum size=3mm](a)at(1,1){$a$};
            \node[NodeGraph, minimum size=3mm](b)at(1,-1){$b$};
            \node[RootGraph, minimum size=3mm](c)at(0,0){$c$};
            \node[NodeGraph, minimum size=3mm](d)at(2,0){$d$};
            \draw[EdgeGraph](a)--(c);
            \draw[EdgeGraph](b)--(c);
            \draw[EdgeGraph](d)--(c);
        \end{tikzpicture}
        \enspace + \enspace
        \begin{tikzpicture}[Centering,scale=.6]
            \node[NodeGraph, minimum size=3mm](a)at(1,1){$a$};
            \node[NodeGraph, minimum size=3mm](b)at(1,-1){$b$};
            \node[RootGraph, minimum size=3mm](c)at(0,0){$c$};
            \node[NodeGraph, minimum size=3mm](d)at(2,0){$d$};
            \draw[EdgeGraph](a)--(d);
            \draw[EdgeGraph](d)--(c);
            \draw[EdgeGraph](b)--(c);
        \end{tikzpicture}
        \enspace + \enspace
        \begin{tikzpicture}[Centering,scale=.6]
            \node[NodeGraph, minimum size=3mm](a)at(1,1){$a$};
            \node[NodeGraph, minimum size=3mm](b)at(1,-1){$b$};
            \node[RootGraph, minimum size=3mm](c)at(0,0){$c$};
            \node[NodeGraph, minimum size=3mm](d)at(2,0){$d$};
            \draw[EdgeGraph](a)--(c);
            \draw[EdgeGraph](d)--(c);
            \draw[EdgeGraph](b)--(d);
        \end{tikzpicture}
        \enspace + \enspace
        \begin{tikzpicture}[Centering,scale=.6]
            \node[NodeGraph, minimum size=3mm](a)at(1,1){$a$};
            \node[NodeGraph, minimum size=3mm](b)at(1,-1){$b$};
            \node[RootGraph, minimum size=3mm](c)at(0,0){$c$};
            \node[NodeGraph, minimum size=3mm](d)at(2,0){$d$};
            \draw[EdgeGraph](a)--(d);
            \draw[EdgeGraph](d)--(c);
            \draw[EdgeGraph](b)--(d);
        \end{tikzpicture}\\
        =&\enspace 
        \begin{tikzpicture}[Centering,scale=.6]
            \node[RootGraph, minimum size=3mm](s)at(0,0){$\ast$};
            \node[NodeGraph, minimum size=3mm](a)at(1,1){$a$};
            \node[NodeGraph, minimum size=3mm](b)at(1,-1){$b$};
            \draw[EdgeGraph](a)--(s);
            \draw[EdgeGraph](b)--(s);
        \end{tikzpicture}
        \enspace \comp \enspace
        \begin{tikzpicture}[Centering,scale=.6]
            \node[RootGraph, minimum size=3mm](c)at(0,0){$c$};
            \node[NodeGraph, minimum size=3mm](d)at(1,0){$d$};
            \draw[EdgeGraph](d)--(c);
        \end{tikzpicture} 
        \enspace = 
        U\times\id\left(\enspace
        \begin{tikzpicture}[Centering,scale=.6]
            \node[RootGraph, minimum size=3mm](s)at(0,0){$\ast$};
            \node[NodeGraph, minimum size=3mm](a)at(1,1){$a$};
            \node[NodeGraph, minimum size=3mm](b)at(1,-1){$b$};
            \draw[ArcGraph](a)--(s);
            \draw[ArcGraph](b)--(s);
        \end{tikzpicture}
        \enspace\right) \comp U\times\id\left(\enspace
        \begin{tikzpicture}[Centering,scale=.6]
            \node[RootGraph, minimum size=3mm](c)at(0,0){$c$};
            \node[NodeGraph, minimum size=3mm](d)at(1,0){$d$};
            \draw[ArcGraph](d)--(c);
        \end{tikzpicture}
        \enspace\right).
    \end{split}\end{equation}
    More formally, for $r\in V_2$, we have:
    \begin{equation} \begin{split}\label{lemma_op_eq1}
        U\times\id&\left((g_1, \ori_{(t,\ast)},\ast)\comp (g_2, \ori_{(t(r),r)},r)\right) \\
        &= \left(g_1|_{V_1} \oplus\bigoplus_{v\in n(\ast)}v\left(\sum V_2\right)\oplus((\sum V_2)^2)^{\oplus g_1(\ast\ast)}\oplus g_2, r\right) \\
        &= (g_1\comp g_2, r).
    \end{split} \end{equation}
    Let now $r$ be a vertex in $V_1$. Denote by $p$ the parent of $\ast$ in the rooted tree $(t,r)$, by $c(\ast)$ its children, by $n_{g_1\setminus t}(\ast)$ 
    the multiset of neighbours of $\ast$ in $g_1\setminus t$ and by $n(\ast)$ the multiset of neighbours of $\ast$ in $g_1$, so that 
    $n(\ast)=n_{g_1\setminus t}(\ast)\cup c(\ast)\cup \{p\}$. We then have
    \begin{equation} \begin{split}\label{lemma_op_eq2}
        U\times\id&\left((g_1,\ori_{(t,r)},r)\comp \sum_{v\in V_2}(g_2,\ori_{(t(v),v)},v) \right) \\
        &= \sum_{v\in V_2} U\times\id\left((g_1,\ori_{(t,r)},r)\comp (g_2,\ori_{(t(v),v)},v) \right) \\
        &= \sum_{v\in V_2}\left(g_1|_{V_1}\oplus pv \oplus \bigoplus_{v\in c(\ast)}v\left(\sum V_2\right)
            \oplus\bigoplus_{v\in n_{g_1\setminus t}(\ast)}v\left(\sum V_2\right) \oplus ((\sum V_2)^2)^{\oplus g_1(\ast\ast)}\oplus g_2, r\right)\\
        &=\left(g_1|_{V_1}\oplus p\left(\sum V_2\right) \oplus \bigoplus_{v\in c(\ast)}v \left(\sum V_2\right)
            \oplus\bigoplus_{v\in n_{g_1\setminus t}(\ast)}v\left(\sum V_2\right) \oplus ((\sum V_2)^2)^{\oplus g_1(\ast\ast)}\oplus g_2, r\right) \\
        &= \left(g_1|_{V_1}\oplus \bigoplus_{v\in n(\ast)}\left(\sum V_2\right)\oplus ((\sum V_2)^2)^{\oplus g_1(\ast\ast)}\oplus g_2, r\right) \\
        &= (g_1\comp g_2, r).
    \end{split} \end{equation}

    {\em Proof of (i)} The species morphism from $\MG\times \PLie$ to $\point{\MG_{orc}}$ given by $(g,t,r) \mapsto (g,\ori_{(t,r)}, r)$ 
    is an operad morphism and hence its image $\ST$ is a sub-operad of $ \K \point{\MG_{orc}}$.

    \begin{figure}[htbp]
        \begin{equation*}
            \begin{tikzpicture}[Centering,scale=1.3]
                \tikzset{every loop/.style={}}
                \node[NodeGraph, minimum size=3mm](a)at(0,0){$a$};
                \node[NodeGraph, minimum size=3mm](b)at(0.5,0.5){$b$};
                \node[NodeGraph, minimum size=3mm](c)at(1,0){$c$};
                \node[NodeGraph, minimum size=3mm](d)at(0.5,-0.5){$d$};
                \draw[EdgeGraph](a)--(b);
                \draw[EdgeGraph](c)edge[bend left=20](d);
                \draw[EdgeGraph](c)edge[bend right=20](d);
            \end{tikzpicture}
            \enspace \otimes \enspace
            \begin{tikzpicture}[Centering,scale=1.3]
                \tikzset{every loop/.style={}}
                \node[NodeGraph, minimum size=3mm](a)at(0,0){$a$};
                \node[NodeGraph, minimum size=3mm](b)at(0.5,0.5){$b$};
                \node[RootGraph, minimum size=3mm](c)at(1,0){$c$};
                \node[NodeGraph, minimum size=3mm](d)at(0.5,-0.5){$d$};
                \draw[EdgeGraph, DarkRed, dashed](b)--(c);
                \draw[EdgeGraph, DarkRed, dashed](a)--(d);
                \draw[EdgeGraph, DarkRed, dashed](b)--(d);
            \end{tikzpicture}
            \enspace \mapsto \enspace
            \begin{tikzpicture}[Centering,scale=1.3]
                \tikzset{every loop/.style={}}
                \node[NodeGraph, minimum size=3mm](a)at(0,0){$a$};
                \node[NodeGraph, minimum size=3mm](b)at(0.5,0.5){$b$};
                \node[RootGraph, minimum size=3mm](c)at(1,0){$c$};
                \node[NodeGraph, minimum size=3mm](d)at(0.5,-0.5){$d$};
                \draw[ArcGraph](a)--(b);
                \draw[ArcGraph](b)--(a);
                \draw[ArcGraph, DarkRed](b)--(c);
                \draw[ArcGraph](c)edge[bend left=30](d);
                \draw[ArcGraph](c)edge[bend right=30](d);
                \draw[ArcGraph](d)edge[bend left=30](c);
                \draw[ArcGraph](d)edge[bend right=30](c);
                \draw[ArcGraph, DarkRed](a)--(d);
                \draw[ArcGraph, DarkRed](d)--(b);
            \end{tikzpicture}
        \end{equation*}
        \caption{An example of the isomorphism of item \textit{i}.}
    \end{figure}

    {\em Proof of (ii)} Let $V_1$ and $V_2$ be two disjoint sets, $g_1\in \MG'_c[V_1]$ and $g_2\in \MG_c[V_2]$ be two connected multigraphs and for each 
    $v\in V_1+\set{\ast}$, $t(v)$ a spanning tree of $g_1$ and for each $v\in V_2$, $t(v)$ a spanning tree of $g_2$. We have
    \begin{equation}\begin{split}
        \sum_{r_1\in V_1+\set{\ast}} (g_1,\ori_{(t(r_1),r_1)}, r_1) &\comp \sum_{r_2\in V_2} (g_2,\ori_{(t(r_2),r_2)}, r_2)
        = \sum_{r_1\in V_1+\set{\ast}}\sum_{r_2\in V_2} (g_1,\ori_{(t(r_1),r_1)}, r_1)\comp (g_2,\ori_{(t(r_2),r_2)}, r_2) 
    \end{split}\end{equation}
    Then from \eqref{lemma_op_eq1} and \eqref{lemma_op_eq2} we know that applying $U\times\id$ to the preceding sum gives us:
    \begin{equation}
        \sum_{r\in V_1+V_2} (g_1\comp g_2,r).
    \end{equation}
    To conclude note that $\Operad[V]$ is the reciprocal image of $\K\{\sum_{v\in V}(g,v)\,|\,g\in \MG_c[V]\}$ by 
    $U\times\id: \ST\rightarrow \point{\MG}$.

    {\em Proof of (iii)} It is easy to see that $\mathcal{I}$ is a left ideal of $\ST$ and hence of $\Operad$. Let be $g_1\in \MG'_c[V_1]$, $g_2\in \MG_c[V_2]$, 
    $r\in V_1$, $t$ a spanning tree of $g_1$ and for every $v\in V_2$, $t(v)$ a spanning tree of $g_2$. Then from 
    \eqref{lemma_op_eq1} and \eqref{lemma_op_eq2} we know that $U\times\id((g_1,\ori_{(t,r)}, r)\comp \sum_{v\in V_2} (g_2,\ori_{(t(v),v))}, v)$ 
    is of the form $(g_1\comp g_2, r)$ if $r\not = \ast$, and of the form $\sum_{v\in V_2} (g_1\comp g_2, v)$ otherwise. In both cases it does not depend on $t$.
    This concludes this proof since $\mathcal{I}[V]$ is the kernel of $(U\times\id)_V: \ST[V]\rightarrow \MG_c^{\bullet}[V]$.
\end{proof}

We can see $\PLie$ as a sub-operad of $\ST$ by the monomorphism $(t,r)\mapsto (t, \ori_r,r)$. The image of the operad morphism $\psi$ of Proposition~\ref{prop_prelie} is
then $\Operad\cap \PLie$ and we have that $\mathcal{I}\cap \PLie = \set{0}$ and hence $\Operad\cap \PLie/\mathcal{I}\cap \PLie = \Operad\cap \PLie$.

\begin{proposition}
    The operad isomorphism $\psi: \T \to \PLie\cap\Operad$ extends into an operad isomorphism $\psi: \MG_c \to \Operad/\mathcal{I}$ satisfying,
    for any $g \in \MG_c[V]$,
    \begin{equation}
        \psi(g) = \sum_{r\in V} (g,\ori_{(t(r),r)}, r),
    \end{equation}
    where for each $r\in V$, $t(r)$ is a spanning tree of $g$. Furthermore, this isomorphism restricts itself to an isomorphism 
    $\G_c \to \Operad\cap\point{\G_{orc}}/\mathcal{I}\cap\point{\G_{orc}}$.
\end{proposition}

\begin{proof}
    This statement is a direct consequence of Lemma~\ref{lemma_op_fond} and its proof.
\end{proof}

The last results are summarized in the following commutative diagram of operad morphisms.
\begin{equation}
    \begin{tikzcd}
        \T \arrow[r, "\sim"] \arrow[d, hook]
        & \PLie\cap\K\Operad/\mathcal{I} \arrow[r, equal] \arrow[d, hook]
        & \PLie\cap\Operad \arrow[d,hook] \arrow[r, hook]
        & \PLie \arrow[d,hook]\\
        \G_c \arrow[r, "\sim"] \arrow[d, hook]
        & \Operad\cap\point{\G}_{orc}/\mathcal{I}\cap\point{\G_{orc}} \arrow[d, hook]
        & \point{\G}_{orc}\cap\Operad \arrow[l, two heads] \arrow[r, hook] \arrow[d, hook]
        & \point{\G}_{orc}\cap\ST \arrow[d, hook] \\
        \MG_c \arrow[r, "\sim"]
        & \Operad/\mathcal{I}
        & \Operad \arrow[l, two heads] \arrow[r, hook]
        & \MG\times\PLie
    \end{tikzcd}
\end{equation}


\subsection{The $\MG$ operad}
As shown in the previous subsection, $\MG$ and its sub-operads have some interesting properties and links with $\PLie$. While $\MG$ was defined by explicitly 
describing $\MG[V]$ for each $V$, the operad $\PLie$ was first defined by its generator and the pre-Lie relation and later proved to be an operad on rooted
tree in \cite{CL01}. It is then natural to search for generators and relations of $\MG$ and its sub-operads. 

We first search for a smallest family of generators of $\G$. The search for such a family is computationally hard. Using computer algebra, we obtain a family 
which generates simple graphs of arity less than $5$:
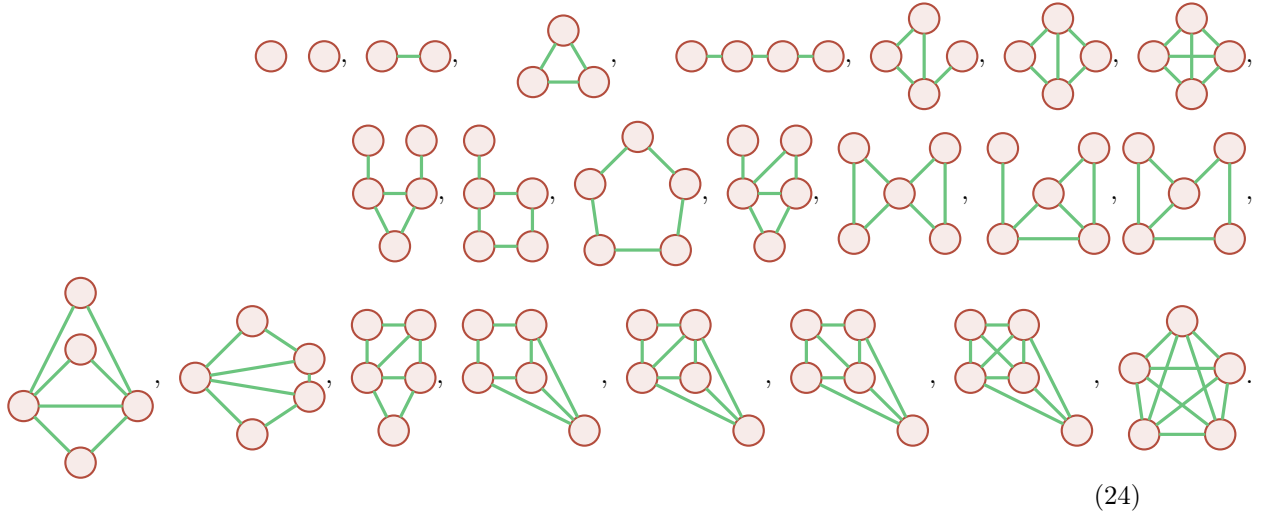
\begin{equation}\begin{split} \label{equ:generators_G}
    &\begin{tikzpicture}[Centering,scale=.6]
        \node[UnlabeledNodeGraph](a)at(0,0){};
        \node[UnlabeledNodeGraph](b)at(1,0){};
    \end{tikzpicture},
    \enspace
    \begin{tikzpicture}[Centering,scale=.6]
        \node[UnlabeledNodeGraph](a)at(0,0){};
        \node[UnlabeledNodeGraph](b)at(1,0){};
        \draw[EdgeGraph](a)--(b);
    \end{tikzpicture},
    \qquad \quad
    \begin{tikzpicture}[Centering,scale=.3]
        \node[UnlabeledNodeGraph](a)at(-1,-1){};
        \node[UnlabeledNodeGraph](b)at(1,-1){};
        \node[UnlabeledNodeGraph](c)at(0,.707){};
        \draw[EdgeGraph](a)--(b);
        \draw[EdgeGraph](a)--(c);
        \draw[EdgeGraph](b)--(c);
    \end{tikzpicture},
    \qquad \quad
    \begin{tikzpicture}[Centering,scale=.4]
        \node[UnlabeledNodeGraph](a)at(0,0){};
        \node[UnlabeledNodeGraph](b)at(1,0){};
        \node[UnlabeledNodeGraph](c)at(2,0){};
        \node[UnlabeledNodeGraph](d)at(3,0){};
        \draw[EdgeGraph](a)--(b);
        \draw[EdgeGraph](b)--(c);
        \draw[EdgeGraph](c)--(d);
    \end{tikzpicture},
    \enspace
    \begin{tikzpicture}[Centering,scale=.4]
        \node[UnlabeledNodeGraph](a)at(-1,-1){};
        \node[UnlabeledNodeGraph](b)at(-1,1){};
        \node[UnlabeledNodeGraph](c)at(0,0){};
        \node[UnlabeledNodeGraph](d)at(-2,0){};
        \draw[EdgeGraph](a)--(b);
        \draw[EdgeGraph](a)--(c);
        \draw[EdgeGraph](a)--(d);
        \draw[EdgeGraph](b)--(d);
    \end{tikzpicture},
    \enspace
    \begin{tikzpicture}[Centering,scale=.4]
        \node[UnlabeledNodeGraph](a)at(-1,-1){};
        \node[UnlabeledNodeGraph](b)at(-1,1){};
        \node[UnlabeledNodeGraph](c)at(0,0){};
        \node[UnlabeledNodeGraph](d)at(-2,0){};
        \draw[EdgeGraph](a)--(b);
        \draw[EdgeGraph](a)--(c);
        \draw[EdgeGraph](a)--(d);
        \draw[EdgeGraph](b)--(c);
        \draw[EdgeGraph](b)--(d);
    \end{tikzpicture},
    \enspace
    \begin{tikzpicture}[Centering,scale=.4]
        \node[UnlabeledNodeGraph](a)at(-1,-1){};
        \node[UnlabeledNodeGraph](b)at(-1,1){};
        \node[UnlabeledNodeGraph](c)at(0,0){};
        \node[UnlabeledNodeGraph](d)at(-2,0){};
        \draw[EdgeGraph](a)--(b);
        \draw[EdgeGraph](a)--(c);
        \draw[EdgeGraph](a)--(d);
        \draw[EdgeGraph](b)--(c);
        \draw[EdgeGraph](b)--(d);
        \draw[EdgeGraph](c)--(d);
    \end{tikzpicture},
    \\
    &
    \begin{tikzpicture}[Centering,scale=.3]
        \node[UnlabeledNodeGraph](a)at(-2,0){};
        \node[UnlabeledNodeGraph](b)at(0,0){};
        \node[UnlabeledNodeGraph](c)at(-2,2){};
        \node[UnlabeledNodeGraph](d)at(0,2){};
        \node[UnlabeledNodeGraph](e)at(-1,-2){};
        \draw[EdgeGraph](a)--(b);
        \draw[EdgeGraph](a)--(c);
        \draw[EdgeGraph](a)--(e);
        \draw[EdgeGraph](b)--(d);
        \draw[EdgeGraph](b)--(e);
    \end{tikzpicture},
    \enspace
    \begin{tikzpicture}[Centering,scale=.3]
        \node[UnlabeledNodeGraph](a)at(-2,0){};
        \node[UnlabeledNodeGraph](b)at(0,-2){};
        \node[UnlabeledNodeGraph](c)at(-2,2){};
        \node[UnlabeledNodeGraph](d)at(0,0){};
        \node[UnlabeledNodeGraph](e)at(-2,-2){};
        \draw[EdgeGraph](a)--(c);
        \draw[EdgeGraph](a)--(d);
        \draw[EdgeGraph](a)--(e);
        \draw[EdgeGraph](b)--(d);
        \draw[EdgeGraph](b)--(e);
    \end{tikzpicture},
    \enspace
    \begin{tikzpicture}[Centering,scale=.3]
        \node[UnlabeledNodeGraph](a)at(-1,-1){};
        \node[UnlabeledNodeGraph](b)at(-1.25,.75){};
        \node[UnlabeledNodeGraph](c)at(0,2){};
        \node[UnlabeledNodeGraph](d)at(1.25,.75){};
        \node[UnlabeledNodeGraph](e)at(1,-1){};
        \draw[EdgeGraph](a)--(b);
        \draw[EdgeGraph](b)--(c);
        \draw[EdgeGraph](c)--(d);
        \draw[EdgeGraph](d)--(e);
        \draw[EdgeGraph](e)--(a);
    \end{tikzpicture},
    \enspace
    \begin{tikzpicture}[Centering,scale=.3]
        \node[UnlabeledNodeGraph](a)at(-2,0){};
        \node[UnlabeledNodeGraph](b)at(0,0){};
        \node[UnlabeledNodeGraph](c)at(-2,2){};
        \node[UnlabeledNodeGraph](d)at(0,2){};
        \node[UnlabeledNodeGraph](e)at(-1,-2){};
        \draw[EdgeGraph](a)--(b);
        \draw[EdgeGraph](a)--(c);
        \draw[EdgeGraph](a)--(d);
        \draw[EdgeGraph](a)--(e);
        \draw[EdgeGraph](b)--(d);
        \draw[EdgeGraph](b)--(e);
    \end{tikzpicture},
    \enspace
    \begin{tikzpicture}[Centering,scale=.4]
        \node[UnlabeledNodeGraph](a)at(0,0){};
        \node[UnlabeledNodeGraph](b)at(1,-1){};
        \node[UnlabeledNodeGraph](c)at(-1,1){};
        \node[UnlabeledNodeGraph](d)at(1,1){};
        \node[UnlabeledNodeGraph](e)at(-1,-1){};
        \draw[EdgeGraph](a)--(b);
        \draw[EdgeGraph](a)--(c);
        \draw[EdgeGraph](a)--(d);
        \draw[EdgeGraph](a)--(e);
        \draw[EdgeGraph](b)--(d);
        \draw[EdgeGraph](c)--(e);
    \end{tikzpicture},
    \enspace
    \begin{tikzpicture}[Centering,scale=.4]
        \node[UnlabeledNodeGraph](a)at(0,0){};
        \node[UnlabeledNodeGraph](b)at(1,-1){};
        \node[UnlabeledNodeGraph](c)at(-1,1){};
        \node[UnlabeledNodeGraph](d)at(1,1){};
        \node[UnlabeledNodeGraph](e)at(-1,-1){};
        \draw[EdgeGraph](a)--(b);
        \draw[EdgeGraph](a)--(d);
        \draw[EdgeGraph](a)--(e);
        \draw[EdgeGraph](b)--(d);
        \draw[EdgeGraph](b)--(e);
        \draw[EdgeGraph](c)--(e);
    \end{tikzpicture},
    \begin{tikzpicture}[Centering,scale=.4]
        \node[UnlabeledNodeGraph](a)at(0,0){};
        \node[UnlabeledNodeGraph](b)at(1,-1){};
        \node[UnlabeledNodeGraph](c)at(-1,1){};
        \node[UnlabeledNodeGraph](d)at(1,1){};
        \node[UnlabeledNodeGraph](e)at(-1,-1){};
        \draw[EdgeGraph](a)--(c);
        \draw[EdgeGraph](a)--(d);
        \draw[EdgeGraph](a)--(e);
        \draw[EdgeGraph](b)--(d);
        \draw[EdgeGraph](b)--(e);
        \draw[EdgeGraph](c)--(e);
    \end{tikzpicture},
    \enspace
    \begin{tikzpicture}[Centering,scale=.3]
        \node[UnlabeledNodeGraph](a)at(0,-1){};
        \node[UnlabeledNodeGraph](b)at(3,-1){};
        \node[UnlabeledNodeGraph](c)at(1.5,2){};
        \node[UnlabeledNodeGraph](d)at(1.5,.5){};
        \node[UnlabeledNodeGraph](e)at(1.5,-2.5){};
        \draw[EdgeGraph](a)--(b);
        \draw[EdgeGraph](a)--(c);
        \draw[EdgeGraph](a)--(d);
        \draw[EdgeGraph](a)--(e);
        \draw[EdgeGraph](b)--(c);
        \draw[EdgeGraph](b)--(d);
        \draw[EdgeGraph](b)--(e);
    \end{tikzpicture},
    \enspace
    \begin{tikzpicture}[Centering,scale=.3]
        \node[UnlabeledNodeGraph](a)at(0,-1.5){};
        \node[UnlabeledNodeGraph](b)at(3,-1){};
        \node[UnlabeledNodeGraph](c)at(3,-2){};
        \node[UnlabeledNodeGraph](d)at(1.5,0){};
        \node[UnlabeledNodeGraph](e)at(1.5,-3){};
        \draw[EdgeGraph](a)--(b);
        \draw[EdgeGraph](a)--(c);
        \draw[EdgeGraph](a)--(d);
        \draw[EdgeGraph](a)--(e);
        \draw[EdgeGraph](b)--(c);
        \draw[EdgeGraph](b)--(d);
        \draw[EdgeGraph](c)--(e);
    \end{tikzpicture},
    \enspace
    \begin{tikzpicture}[Centering,scale=.6]
        \node[UnlabeledNodeGraph](a)at(0,0){};
        \node[UnlabeledNodeGraph](b)at(1,1){};
        \node[UnlabeledNodeGraph](c)at(.5,-1){};
        \node[UnlabeledNodeGraph](d)at(0,1){};
        \node[UnlabeledNodeGraph](e)at(1,0){};
        \draw[EdgeGraph](a)--(b);
        \draw[EdgeGraph](a)--(c);
        \draw[EdgeGraph](a)--(d);
        \draw[EdgeGraph](a)--(e);
        \draw[EdgeGraph](b)--(d);
        \draw[EdgeGraph](b)--(e);
        \draw[EdgeGraph](c)--(e);
    \end{tikzpicture},
    \\
    &
    \quad
    \begin{tikzpicture}[Centering,scale=.6]
        \node[UnlabeledNodeGraph](a)at(0,0){};
        \node[UnlabeledNodeGraph](b)at(1,1){};
        \node[UnlabeledNodeGraph](c)at(2,-1){};
        \node[UnlabeledNodeGraph](d)at(0,1){};
        \node[UnlabeledNodeGraph](e)at(1,0){};
        \draw[EdgeGraph](a)--(c);
        \draw[EdgeGraph](a)--(d);
        \draw[EdgeGraph](a)--(e);
        \draw[EdgeGraph](b)--(c);
        \draw[EdgeGraph](b)--(d);
        \draw[EdgeGraph](b)--(e);
        \draw[EdgeGraph](c)--(e);
    \end{tikzpicture},
    \enspace
    \begin{tikzpicture}[Centering,scale=.6]
        \node[UnlabeledNodeGraph](a)at(0,0){};
        \node[UnlabeledNodeGraph](b)at(1,1){};
        \node[UnlabeledNodeGraph](c)at(2,-1){};
        \node[UnlabeledNodeGraph](d)at(0,1){};
        \node[UnlabeledNodeGraph](e)at(1,0){};
        \draw[EdgeGraph](a)--(b);
        \draw[EdgeGraph](a)--(c);
        \draw[EdgeGraph](a)--(d);
        \draw[EdgeGraph](a)--(e);
        \draw[EdgeGraph](b)--(c);
        \draw[EdgeGraph](b)--(d);
        \draw[EdgeGraph](b)--(e);
        \draw[EdgeGraph](c)--(e);
    \end{tikzpicture},
    \enspace
    \begin{tikzpicture}[Centering,scale=.6,rotate=90]
        \node[UnlabeledNodeGraph](a)at(0,0){};
        \node[UnlabeledNodeGraph](b)at(1,1){};
        \node[UnlabeledNodeGraph](c)at(-1,-1){};
        \node[UnlabeledNodeGraph](d)at(0,1){};
        \node[UnlabeledNodeGraph](e)at(1,0){};
        \draw[EdgeGraph](a)--(b);
        \draw[EdgeGraph](a)--(c);
        \draw[EdgeGraph](a)--(d);
        \draw[EdgeGraph](a)--(e);
        \draw[EdgeGraph](b)--(d);
        \draw[EdgeGraph](b)--(e);
        \draw[EdgeGraph](c)--(d);
        \draw[EdgeGraph](c)--(e);
    \end{tikzpicture},
    \enspace
    \begin{tikzpicture}[Centering,scale=.6,rotate=90]
        \node[UnlabeledNodeGraph](a)at(0,0){};
        \node[UnlabeledNodeGraph](b)at(1,1){};
        \node[UnlabeledNodeGraph](c)at(-1,-1){};
        \node[UnlabeledNodeGraph](d)at(0,1){};
        \node[UnlabeledNodeGraph](e)at(1,0){};
        \draw[EdgeGraph](a)--(b);
        \draw[EdgeGraph](a)--(c);
        \draw[EdgeGraph](a)--(d);
        \draw[EdgeGraph](a)--(e);
        \draw[EdgeGraph](b)--(d);
        \draw[EdgeGraph](b)--(e);
        \draw[EdgeGraph](c)--(d);
        \draw[EdgeGraph](c)--(e);
        \draw[EdgeGraph](d)--(e);
    \end{tikzpicture},
    \enspace
    \begin{tikzpicture}[Centering,scale=.35]
        \node[UnlabeledNodeGraph](a)at(-1,-1){};
        \node[UnlabeledNodeGraph](b)at(-1.25,.75){};
        \node[UnlabeledNodeGraph](c)at(0,2){};
        \node[UnlabeledNodeGraph](d)at(1.25,.75){};
        \node[UnlabeledNodeGraph](e)at(1,-1){};
        \draw[EdgeGraph](a)--(b);
        \draw[EdgeGraph](a)--(c);
        \draw[EdgeGraph](a)--(d);
        \draw[EdgeGraph](a)--(e);
        \draw[EdgeGraph](b)--(c);
        \draw[EdgeGraph](b)--(d);
        \draw[EdgeGraph](b)--(e);
        \draw[EdgeGraph](c)--(d);
        \draw[EdgeGraph](c)--(e);
        \draw[EdgeGraph](d)--(e);
    \end{tikzpicture}.
\end{split}\end{equation}
Due to the symmetric group action on $\G$, only the knowledge of the shapes of the graphs is significant. While~\eqref{equ:generators_G} does not provide to us 
any particular insight on a possible characterisation of the generators, it does suggest that any graph with ``enough'' edges must be a generator. 
This is confirmed by the following lemma. We say that a graph $g\in \G$ is generated by a set $E$ of graphs if $g$ is in the sub-operad generated by $E$.

\begin{lemma} \label{lem:infinite_number_generators_G}
    Let $\spe{S}$ be a sub-species of $\G$ and let $g$ be a graph in $\G[V]$ with at least $\binom{n-1}{2} +1$ edges, where $n = \card{V}$. 
    Then $g$ is generated by $\spe{S}$ if and only if $g\in\spe{S}[V]$.
\end{lemma}

\begin{proof}
    Suppose that $g\not\in\spe{S}[V]$. It is sufficient to show that $g$ cannot appear in the support of any vector of the form $g_1\comp g_2$ for any 
    $g_1$ and $g_2$ different of $g$. Hence let $g_1\in \G'[V_1]$ and $g_2\in \G[V_2]$ be two graphs, and 
    denote by $e_1$ the number of edges of $g_1$ and by $e_2$ the number of edges of $g_2$. Then the graphs in the sum $g_1\comp g_2$ have $e_1+e_2$ edges. 
    This is maximal when $g_1$ and $g_2$ are both complete graphs and is then equal to $\binom{x}{2}+\binom{n+1-x}{2} = x^2-(n+1)x+\binom{n+1}{2}$ where 
    $1\leq x=\card{V_1}\leq n$.

    If $x = 1$ then necessarily $g_1 =\emptyset_{\ast}$ and $g_1\comp g_2 = g_2$ and $g$ appears in the sum $g_1\comp g_2$ if and only if $g=g_2$. This is impossible, 
    hence $x\not = 1$. Similarly we have $x\not = n$. The expression $x^2-(n+1)x+\binom{n+1}{2}$ is then maximal for $x=2$ or $x=n-1$ and is equal in both cases 
    to $\binom{n-1}{2}<\binom{n-1}{2} +1$. This implies that $g$ can not be part of the sum of graphs $g_1\comp g_2$ and concludes the proof.
\end{proof}

\begin{proposition} \label{gc} 
    The operad $\G$ is not free and has an infinite number of generators.
\end{proposition}

\begin{proof}
    The fact that $\G$ has an infinite number of generators is a direct consequence of Lemma~\ref{lem:infinite_number_generators_G}. Moreover, the relation
    \begin{equation}\begin{split} \label{nf}
        \begin{tikzpicture}[Centering,scale=.7]
            \node[NodeGraph](a)at(0,0){$a$};
            \node[NodeGraph](s)at(1,0){$\ast$};
            \draw[EdgeGraph](a)--(s);
        \end{tikzpicture}
        &
        \comp
        \begin{tikzpicture}[Centering,scale=.7]
            \node[NodeGraph](b)at(0,0){$b$};
            \node[NodeGraph](c)at(1,0){$c$};
            \draw[EdgeGraph](b)--(c);
        \end{tikzpicture}
        \enspace + \enspace
        \begin{tikzpicture}[Centering,scale=.7]
            \node[NodeGraph](c)at(0,0){$c$};
            \node[NodeGraph](s)at(1,0){$\ast$};
            \draw[EdgeGraph](c)--(s);
        \end{tikzpicture}
        \comp
        \begin{tikzpicture}[Centering,scale=.7]
            \node[NodeGraph](b)at(0,0){$b$};
            \node[NodeGraph](a)at(1,0){$a$};
            \draw[EdgeGraph](b)--(a);
        \end{tikzpicture}
        \enspace - \enspace
        \begin{tikzpicture}[Centering,scale=.7]
            \node[NodeGraph](b)at(0,0){$b$};
            \node[NodeGraph](s)at(1,0){$\ast$};
            \draw[EdgeGraph](b)--(s);
        \end{tikzpicture}
        \comp
        \begin{tikzpicture}[Centering,scale=.7]
            \node[NodeGraph](a)at(0,0){$a$};
            \node[NodeGraph](c)at(1,0){$c$};
            \draw[EdgeGraph](a)--(c);
        \end{tikzpicture}
        \enspace
        - 2 \,
        \begin{tikzpicture}[Centering,scale=.7]
            \node[NodeGraph](a)at(0,0){$a$};
            \node[NodeGraph](b)at(1,0){$b$};
            \node[NodeGraph](c)at(2,0){$c$};
            \draw[EdgeGraph](a)--(b);
            \draw[EdgeGraph](b)--(c);
        \end{tikzpicture}
        \\[.5em]
        & =
        \begin{tikzpicture}[Centering,scale=.7]
            \node[NodeGraph](a)at(0,0){$a$};
            \node[NodeGraph](b)at(1,0){$b$};
            \node[NodeGraph](c)at(2,0){$c$};
            \draw[EdgeGraph](a)--(b);
            \draw[EdgeGraph](b)--(c);
        \end{tikzpicture}
        \enspace + \enspace
        \begin{tikzpicture}[Centering,scale=.7]
            \node[NodeGraph](b)at(0,0){$b$};
            \node[NodeGraph](c)at(1,0){$c$};
            \node[NodeGraph](a)at(2,0){$a$};
            \draw[EdgeGraph](b)--(c);
            \draw[EdgeGraph](c)--(a);
        \end{tikzpicture}
        \enspace + \enspace
        \begin{tikzpicture}[Centering,scale=.7]
            \node[NodeGraph](c)at(0,0){$c$};
            \node[NodeGraph](b)at(1,0){$b$};
            \node[NodeGraph](a)at(2,0){$a$};
            \draw[EdgeGraph](c)--(b);
            \draw[EdgeGraph](b)--(a);
        \end{tikzpicture}
        \enspace + \enspace
        \begin{tikzpicture}[Centering,scale=.7]
            \node[NodeGraph](b)at(0,0){$b$};
            \node[NodeGraph](a)at(1,0){$a$};
            \node[NodeGraph](c)at(2,0){$c$};
            \draw[EdgeGraph](b)--(a);
            \draw[EdgeGraph](a)--(c);
        \end{tikzpicture}
        \\[.5em]
        & \qquad - \enspace
        \begin{tikzpicture}[Centering,scale=.7]
            \node[NodeGraph](b)at(0,0){$b$};
            \node[NodeGraph](a)at(1,0){$a$};
            \node[NodeGraph](c)at(2,0){$c$};
            \draw[EdgeGraph](b)--(a);
            \draw[EdgeGraph](a)--(c);
        \end{tikzpicture}
        \enspace - \enspace
        \begin{tikzpicture}[Centering,scale=.7]
            \node[NodeGraph](a)at(0,0){$a$};
            \node[NodeGraph](c)at(1,0){$c$};
            \node[NodeGraph](b)at(2,0){$b$};
            \draw[EdgeGraph](a)--(c);
            \draw[EdgeGraph](c)--(b);
        \end{tikzpicture}
        - 2 \,
        \begin{tikzpicture}[Centering,scale=.7]
            \node[NodeGraph](a)at(0,0){$a$};
            \node[NodeGraph](b)at(1,0){$b$};
            \node[NodeGraph](c)at(2,0){$c$};
            \draw[EdgeGraph](a)--(b);
            \draw[EdgeGraph](b)--(c);
        \end{tikzpicture}
        \\[.5em]
        & = 0
    \end{split}\end{equation}
    shows that $\G$ is not free.
\end{proof}

As a consequence of Proposition~\ref{gc}, it seems particularly involved to find a definition of $\G$ by generator and relation. We did not have any more success in describing
$\T$ in such a way even exploiting the monomorphism $\psi:\T\to\PLie$. We at least managed to compute a family of generators of $\T$ with arity less than $6$:
\begin{equation}\begin{split}
    \begin{tikzpicture}[Centering,scale=.6]
        \node[UnlabeledNodeGraph](a)at(0,0){};
        \node[UnlabeledNodeGraph](b)at(1,0){};
        \draw[EdgeGraph](a)--(b);
    \end{tikzpicture},
    \qquad \quad
    \begin{tikzpicture}[Centering,scale=.35]
        \node[UnlabeledNodeGraph](a)at(0,0){};
        \node[UnlabeledNodeGraph](b)at(1,0){};
        \node[UnlabeledNodeGraph](c)at(2,0){};
        \node[UnlabeledNodeGraph](d)at(3,0){};
        \draw[EdgeGraph](a)--(b);
        \draw[EdgeGraph](b)--(c);
        \draw[EdgeGraph](c)--(d);
    \end{tikzpicture},
    \qquad \quad
    \begin{tikzpicture}[Centering,scale=.35]
        \node[UnlabeledNodeGraph](a)at(0,0){};
        \node[UnlabeledNodeGraph](b)at(-1,0){};
        \node[UnlabeledNodeGraph](c)at(0,1){};
        \node[UnlabeledNodeGraph](d)at(1,0){};
        \node[UnlabeledNodeGraph](e)at(-2,0){};
        \draw[EdgeGraph](a)--(b);
        \draw[EdgeGraph](a)--(c);
        \draw[EdgeGraph](a)--(d);
        \draw[EdgeGraph](b)--(e);
    \end{tikzpicture},
    \qquad \quad
    \begin{tikzpicture}[Centering,scale=.35]
        \node[UnlabeledNodeGraph](a)at(0,0){};
        \node[UnlabeledNodeGraph](b)at(-1,0){};
        \node[UnlabeledNodeGraph](c)at(0,1){};
        \node[UnlabeledNodeGraph](d)at(1,0){};
        \node[UnlabeledNodeGraph](e)at(0,-1){};
        \node[UnlabeledNodeGraph](f)at(-2,0){};
        \draw[EdgeGraph](a)--(b);
        \draw[EdgeGraph](a)--(c);
        \draw[EdgeGraph](a)--(d);
        \draw[EdgeGraph](a)--(e);
        \draw[EdgeGraph](b)--(f);
    \end{tikzpicture},
    \enspace
    \begin{tikzpicture}[Centering,scale=.4]
        \node[UnlabeledNodeGraph](a)at(0,0){};
        \node[UnlabeledNodeGraph](b)at(-1,0){};
        \node[UnlabeledNodeGraph](c)at(1,-1){};
        \node[UnlabeledNodeGraph](d)at(1,1){};
        \node[UnlabeledNodeGraph](e)at(-2,-1){};
        \node[UnlabeledNodeGraph](f)at(-2,1){};
        \draw[EdgeGraph](a)--(b);
        \draw[EdgeGraph](a)--(c);
        \draw[EdgeGraph](a)--(d);
        \draw[EdgeGraph](b)--(e);
        \draw[EdgeGraph](b)--(f);
    \end{tikzpicture},
    \enspace
    \begin{tikzpicture}[Centering,scale=.4]
        \node[UnlabeledNodeGraph](a)at(0,0){};
        \node[UnlabeledNodeGraph](b)at(1,0){};
        \node[UnlabeledNodeGraph](c)at(2,0){};
        \node[UnlabeledNodeGraph](d)at(3,0){};
        \node[UnlabeledNodeGraph](e)at(4,0){};
        \node[UnlabeledNodeGraph](f)at(2,1){};
        \draw[EdgeGraph](a)--(b);
        \draw[EdgeGraph](b)--(c);
        \draw[EdgeGraph](c)--(d);
        \draw[EdgeGraph](d)--(e);
        \draw[EdgeGraph](c)--(f);
    \end{tikzpicture}.
\end{split}\end{equation}
As with the previous family~\eqref{equ:generators_G}, this does not give to us any insight on a possible family of generators.


\subsection{Finitely generated sub-operads}
Since finding family of generators seems out of reach, let us now directly focus on finitely generated sub-operads of $\MG$. In particular we will study the 
operads generated by:
\begin{enumerate}
\item 
\begin{math}
\left\{
\begin{tikzpicture}[Centering,scale=.6]
    \node[NodeGraph](a)at(0,0){$a$};
    \node[NodeGraph](b)at(1,0){$b$};
\end{tikzpicture}
\right\}
\end{math} which we denote by $\G_{\emptyset}$ and which is isomorphic to $\Com$,
\item \begin{math}\left\{
\begin{tikzpicture}[Centering,scale=.6]
    \node[NodeGraph](a)at(0,0){$a$};
    \node[NodeGraph](b)at(1,0){$b$};
    \draw[EdgeGraph](a)--(b);
\end{tikzpicture}
\right\}
\end{math} which we denote by $\Seg$ and which is isomorphic to $\ComMag$,
\item 
\begin{math}
\left\{
\begin{tikzpicture}[Centering,scale=.6]
    \node[NodeGraph](a)at(0,0){$a$};
    \node[NodeGraph](b)at(1,0){$b$};
\end{tikzpicture},
\begin{tikzpicture}[Centering,scale=.6]
    \node[NodeGraph](a)at(0,0){$a$};
    \node[NodeGraph](b)at(1,0){$b$};
    \draw[EdgeGraph](a)--(b);
\end{tikzpicture}
\right\}
\end{math} which we denote by $\SP$,
\item \begin{math}
\left\{
\begin{tikzpicture}[Centering,scale=.6]
    \tikzset{every loop/.style={}}
    \node[NodeGraph](a)at(0,0){$a$};
    \draw[EdgeGraph](a)edge[loop](a);
\end{tikzpicture},
\begin{tikzpicture}[Centering,scale=.6]
    \node[NodeGraph](a)at(0,0){$a$};
    \node[NodeGraph](b)at(1,0){$b$};
\end{tikzpicture}
\right\}
\end{math} which we denote by $\LP$.
\end{enumerate}

First, note that the sub-operad $\G_{\emptyset}$ generated by
\begin{math}
\left\{
\begin{tikzpicture}[Centering,scale=.6]
    \node[NodeGraph](a)at(0,0){$a$};
    \node[NodeGraph](b)at(1,0){$b$};
\end{tikzpicture}
\right\}
\end{math}
is isomorphic to the commutative operad $\Com$. Indeed, recall from Example~\ref{ex_gen_com} that $\Com$ is the quotient of the free operad over one symmetric
element of size two $\Sym_2$ by the associativity relation. By definition $\G_{\emptyset}$ is also generated by one symmetric element of size $2$, and furthermore we have:
\begin{equation}\label{eq_iso_com}
    \begin{tikzpicture}[Centering,scale=.7]
        \node[NodeGraph](a)at(0,0){$a$};
        \node[NodeGraph](s)at(1,0){$\ast$};
    \end{tikzpicture}
    \enspace \comp \enspace
    \begin{tikzpicture}[Centering,scale=.7]
        \node[NodeGraph](b)at(0,0){$b$};
        \node[NodeGraph](c)at(1,0){$c$};
    \end{tikzpicture}
    \enspace = \enspace
    \begin{tikzpicture}[Centering,scale=.7]
        \node[NodeGraph](a)at(0,0){$a$};
        \node[NodeGraph](b)at(1,0){$b$};
        \node[NodeGraph](c)at(2,0){$c$};
    \end{tikzpicture}
    \enspace = \enspace
    \begin{tikzpicture}[Centering,scale=.7]
        \node[NodeGraph](s)at(0,0){$\ast$};
        \node[NodeGraph](c)at(1,0){$c$};
    \end{tikzpicture}
    \enspace \comp \enspace
    \begin{tikzpicture}[Centering,scale=.7]
        \node[NodeGraph](a)at(0,0){$a$};
        \node[NodeGraph](b)at(1,0){$b$};
    \end{tikzpicture},
\end{equation}
which is the associativity relation. Hence $\G_{\emptyset}\cong\Com$. This could also be observed from the fact that we clearly have $\G_{\emptyset}[V]=\K\emptyset_V$ and
hence the map $\emptyset_V\mapsto\mu_V$ implies an isomorphism from $\G_{\emptyset}$ to $\Com$. 

The other three cases are more involved.
\begin{proposition} \label{prop_commag}
    The sub-operad $\Seg$ of $\G$ generated by
    \begin{math}
        \left\{
        \begin{tikzpicture}[Centering,scale=.6]
            \node[NodeGraph](a)at(0,0){$a$};
            \node[NodeGraph](b)at(1,0){$b$};
            \draw[EdgeGraph](a)--(b);
        \end{tikzpicture}
        \right\}
    \end{math}
    is isomorphic to $\ComMag$.
\end{proposition}

\begin{proof}
    We know from Proposition~\ref{prop_prelie} that $\Seg$ is isomorphic to the sub-operad of $\PLie$ generated by 
    \begin{equation}
        \left\{
        \begin{tikzpicture}[Centering,scale=.6]
            \node[RootGraph](a)at(0,0){$a$};
            \node[NodeGraph](b)at(0,-1){$b$};
            \draw[EdgeGraph](a)--(b);
        \end{tikzpicture}
        +
        \begin{tikzpicture}[Centering,scale=.6]
            \node[RootGraph](a)at(0,0){$b$};
            \node[NodeGraph](b)at(0,-1){$a$};
            \draw[EdgeGraph](a)--(b);
        \end{tikzpicture}
        \right\}
    \end{equation}
    Then~\cite{BL11} gives us that the map which sends the above element to $s_{\set{a,b}}\in\Sym_2[\set{a,b}]$ induces an isomorphism between the sub-operad and $\ComMag$. 
    This concludes the proof.
\end{proof}

Remark that while $\ComMag\cong\Seg$, the image of an element in the canonical basis of $\ComMag[V]$, i.e. a binary tree with $V$ as set of leaves, 
is not a single graph but a sum of trees with vertex set $V$.

This result suggests the following more general conjecture.
\begin{conjecture}
    The sub-operad of $\G$ generated by the complete graphs $K_V$ is isomorphic to $\free{\Sym_n}$, the free operad over one symmetric element of size $n=\card{V}$.
\end{conjecture}
Proving this would require showing that there are no relations involving only complete graphs of size $n$ in $\G$ which is highly non trivial. In fact this was avoided
in the proof of Proposition~\ref{prop_commag} by using the results of \cite{BL11} which is somewhat equivalent in the case of $n=2$.
    
The isomorphisms $\Com\cong\G_{\emptyset}$ and $\ComMag\cong\Seg$ allow us to see $\Com$ and $\ComMag$ as disjoint sub-operads of $\G$ and hence gives us a natural 
way to define the smallest operad containing these two as disjoint sub-operads. Denote by $\mathcal{G}$ the sub-species of $\G$ generated by
\begin{math}
    \left\{
    \begin{tikzpicture}[Centering,scale=.6]
        \node[NodeGraph](a)at(0,0){$a$};
        \node[NodeGraph](b)at(1,0){$b$};
    \end{tikzpicture},
    \begin{tikzpicture}[Centering,scale=.6]
        \node[NodeGraph](a)at(0,0){$a$};
        \node[NodeGraph](b)at(1,0){$b$};
        \draw[EdgeGraph](a)--(b);
    \end{tikzpicture}
    \right\}
\end{math}
and $\SP$ the sub-operad of $\mathcal{G}$ generated by these two elements. This operad has some interesting properties. Recall from Remark~\ref{sec_operads} 
that we use the notation $\comp^{\xi}$ for the grafting of tree in a free operad and that we denote the equivalence class of $x$ by $\iso{x}$.

\begin{proposition}\label{prop_sp}
    The three following operads are isomorphic
    \begin{itemize}
        \item $\SP$
        \item $\Ope(\mathcal{G},\spe{R})$ where $\spe{R}$ is the subspecies of $\free{\mathcal{G}}$ generated by
        \begin{subequations}
            \begin{equation} \label{equ:rel_1}
                \Points{c}{\ast} \comp^{\xi} \Points{a}{b}
                \enspace - \enspace
                \Points{a}{\ast} \comp^{\xi} \Points{b}{c},
            \end{equation}
            \begin{center}
                and
            \end{center}
            \begin{equation} \label{equ:rel_2}
                \Segment{a}{\ast} \comp^{\xi} \Points{b}{c}
                \enspace - \enspace
                \Points{c}{\ast} \comp^{\xi} \Segment{a}{b}
                \enspace - \enspace
                \Points{b}{\ast} \comp^{\xi} \Segment{a}{c}.
            \end{equation}
        \end{subequations}
        \item $\Com(\ComMag)$, the assemblies of $\ComMag$.
    \end{itemize}
    In particular, $\SP$ is binary and quadratic.
\end{proposition}

\begin{proof}
    The element $\iso{\Points{a}{b}}$ of $\Ope(\mathcal{G},\mathcal{R})$ is symmetric of size 2 and follows the associativity relation~\eqref{equ:rel_1}.
    Hence the sub-operad of $\Ope(\mathcal{G},\mathcal{R})$ generated by $\iso{\Points{a}{b}}$ is equal to $\Com$ and all the trees in $\free{\mathcal{G}}$ 
    with $V$ as leaves and whose labels are all empty graphs over two points are sent over $\mu_V$ when passing to the quotient. In the same way, the 
    element $\iso{\Segment{a}{b}}$ of $\Ope(\mathcal{G},\mathcal{R})$ is symmetric of size 2 and do not follow any relation involving only itself, 
    hence the sub-operad of $\Ope(\mathcal{G},\mathcal{R})$ generated by $\iso{\Segment{a}{b}}$ is equal to $\ComMag$ and 
    $\iso{\Segment{a}{b}}=s_{\set{a,b}}\in\ComMag[\set{a,b}]$.

    There is a natural epimorphism $\phi$ from $\free{\mathcal{G}}$ to $\SP$ which is the identity on 
    \begin{math}
        \begin{tikzpicture}[Centering,scale=.6]
            \node[NodeGraph](a)at(0,0){$a$};
            \node[NodeGraph](b)at(1,0){$b$};
        \end{tikzpicture}
    \end{math}
    and
    \begin{math}
        \begin{tikzpicture}[Centering,scale=.6]
            \node[NodeGraph](a)at(0,0){$a$};
            \node[NodeGraph](b)at(1,0){$b$};
            \draw[EdgeGraph](a)--(b);
        \end{tikzpicture}
    \end{math}
    and which sends a partial composition $g_1\comp^{\xi} g_2$ on the partial composition $g_1\comp g_2$. We already proved by \eqref{eq_iso_com} that
    the vector~\eqref{equ:rel_1} is in the kernel $\phi$. The case of \eqref{equ:rel_2} is also straightforward:
    \begin{equation}\begin{split}
        \begin{tikzpicture}[Centering,scale=.7]
            \node[NodeGraph](a)at(0,0){$a$};
            \node[NodeGraph](s)at(1,0){$\ast$};
            \draw[EdgeGraph](a)--(s);
        \end{tikzpicture}
        \enspace \comp \enspace
        \begin{tikzpicture}[Centering,scale=.7]
            \node[NodeGraph](b)at(0,0){$b$};
            \node[NodeGraph](c)at(1,0){$c$};
        \end{tikzpicture}
        \enspace &= \enspace
        \begin{tikzpicture}[Centering,scale=.7]
            \node[NodeGraph](a)at(0,0){$a$};
            \node[NodeGraph](b)at(1,0){$b$};
            \node[NodeGraph](c)at(2,0){$c$};
            \draw[EdgeGraph](a)--(b);
        \end{tikzpicture}
        \enspace + \enspace
        \begin{tikzpicture}[Centering,scale=.7]
            \node[NodeGraph](a)at(0,0){$a$};
            \node[NodeGraph](c)at(1,0){$c$};
            \node[NodeGraph](b)at(2,0){$b$};
            \draw[EdgeGraph](a)--(c);
        \end{tikzpicture}\\
        \enspace &= \enspace
        \begin{tikzpicture}[Centering,scale=.7]
            \node[NodeGraph](c)at(0,0){$c$};
            \node[NodeGraph](s)at(1,0){$\ast$};
        \end{tikzpicture}
        \enspace \comp \enspace
        \begin{tikzpicture}[Centering,scale=.7]
            \node[NodeGraph](a)at(0,0){$a$};
            \node[NodeGraph](b)at(1,0){$b$};
            \draw[EdgeGraph](a)--(b);
        \end{tikzpicture}
        \enspace + \enspace
        \begin{tikzpicture}[Centering,scale=.7]
            \node[NodeGraph](b)at(0,0){$b$};
            \node[NodeGraph](s)at(1,0){$\ast$};
        \end{tikzpicture}
        \enspace \comp \enspace
        \begin{tikzpicture}[Centering,scale=.7]
            \node[NodeGraph](a)at(0,0){$a$};
            \node[NodeGraph](c)at(1,0){$c$};
            \draw[EdgeGraph](a)--(c);
        \end{tikzpicture}.
    \end{split}\end{equation}
    To conclude that $\SP\cong \Ope(\mathcal{G},\spe{R})$, we must now show that for any $w\in\Ope(\K\mathcal{G}, \mathcal{R})[V]$, $\phi(w)=0$ implies $w=0$.
    To do this, we first prove the bijection $\Ope(\mathcal{G},\spe{R})\cong \Com(\ComMag)$.
    Because of \eqref{equ:rel_2}, we have that the vector $s_{\set{a,\ast}}\comp^{\xi}\mu_{\set{b,c}}$ is equal to the vector 
    $\mu_{\set{b,\ast}}\comp^{\xi}s_{\set{a,c}} + \mu_{\set{c,\ast}}\comp^{\xi}s_{\set{a,b}}$ in $\Ope(\mathcal{G}, \mathcal{R})$. Hence, by iterating 
    this process, we get that all elements of $\Ope(\mathcal{G}, \mathcal{R})$ can be written as a sum of equivalence classes of trees where no $s$ vertex has 
    a $\mu$ vertex as descendent. This means that $\Ope(\mathcal{G}, \mathcal{R})[V]$ has the following generating family (cf Figure~\ref{fig_gen_fam} for
    an example): 
    \begin{equation}\label{sp_eq_com_commag}
        \set{\mu_{\pi}\comp[]^\xi (t_1, t_2\dots t_k)\,|\, \mu_{\pi}\in\Com[\pi],\, t_i\in \ComMag[V_i]}_{\pi = \set{V_1,\dots V_k}\text{ partition of $V$}},
    \end{equation}
    where $\mu_{\pi}\comp[]^\xi (t_1, t_2\dots t_k)$ stands for $(\dots((\mu_{\pi}\comp[V_1]^{\xi} t_1)\comp[V_2]^{\xi} t_2)\dots)\comp[V_k]^{\xi}t_k$.
    \begin{figure}[htbp]
        \begin{equation*}
            \left(\left(
            \begin{tikzpicture}[Centering,xscale=0.6,yscale=0.35]
                \node[NodeFree](x)at(0,0){$\mu$};
                \node[NodeFree](y)at(-1,2){$\mu$};
                \node[NodeFree, draw=white](1)at(-2,4){$1$};
                \node[NodeFree, draw=white](2)at(0,4){$2$};
                \node[NodeFree, draw=white](3)at(1,2){$3$};
                \draw[EdgeFree](x)--(y);
                \draw[EdgeFree](x)--(3);
                \draw[EdgeFree](1)--(y);
                \draw[EdgeFree](2)--(y);
            \end{tikzpicture}
            \enspace \comp[1]^\xi \enspace
            \begin{tikzpicture}[Centering,xscale=0.6,yscale=0.35]
                \node[NodeFree](x)at(0,0){$s$};
                \node[NodeFree, draw=white](a)at(-1,2){$a$};
                \node[NodeFree, draw=white](b)at(1,2){$b$};
                \draw[EdgeFree](x)--(a);
                \draw[EdgeFree](x)--(b);
            \end{tikzpicture}
            \enspace \right) \comp[2]^\xi \enspace
            \begin{tikzpicture}[Centering,scale=0.6]
                \node[NodeFree, draw=white](a)at(0,0){$c$};
            \end{tikzpicture}
            \enspace \right) \comp[3]^\xi \enspace
            \begin{tikzpicture}[Centering,xscale=0.6,yscale=0.35]
                \node[NodeFree](x)at(0,0){$s$};
                \node[NodeFree](y)at(1,2){$s$};
                \node[NodeFree, draw=white](d)at(0,4){$d$};
                \node[NodeFree, draw=white](e)at(2,4){$e$};
                \node[NodeFree, draw=white](f)at(-1,2){$f$};
                \draw[EdgeFree](x)--(y);
                \draw[EdgeFree](x)--(f);
                \draw[EdgeFree](d)--(y);
                \draw[EdgeFree](e)--(y);
            \end{tikzpicture}
            = 
        \end{equation*}
        \begin{equation}\label{fig_gen_fam}
            \begin{tikzpicture}[Centering,xscale=0.6,yscale=0.35]
                \node[NodeFree](x)at(0,0){$\mu$};
                \node[NodeFree](y)at(-1,2){$\mu$};
                \node[NodeFree, draw=white](2)at(-0.5,4){$c$};
                \node[NodeFree, draw=white](3)at(1,2){$3$};
                \node[NodeFree](1)at(-2,4){$s$};
                \node[NodeFree, draw=white](a)at(-3,6){$a$};
                \node[NodeFree, draw=white](b)at(-1,6){$b$};
                \draw[EdgeFree](1)--(a);
                \draw[EdgeFree](1)--(b);
                \node[NodeFree](3)at(1,2){$s$};
                \node[NodeFree](4)at(2,4){$s$};
                \node[NodeFree, draw=white](d)at(1,6){$d$};
                \node[NodeFree, draw=white](e)at(3,6){$e$};
                \node[NodeFree, draw=white](f)at(0.5,4){$f$};
                \draw[EdgeFree](3)--(4);
                \draw[EdgeFree](3)--(f);
                \draw[EdgeFree](d)--(4);
                \draw[EdgeFree](e)--(4);
                \draw[EdgeFree](x)--(y);
                \draw[EdgeFree](x)--(3);
                \draw[EdgeFree](1)--(y);
                \draw[EdgeFree](2)--(y);
            \end{tikzpicture}
        \end{equation}
        \caption{An element in the generating family of $\Ope(\mathcal{G}, \mathcal{R})[\set{a,b,c,d,e,f}]$.}
    \end{figure}
    To conclude that $\Ope(\mathcal{G},\spe{R})\cong \Com(\ComMag)$ just remark that the partial composition of $\Ope(\mathcal{G},\spe{R})$ acts the same
    way than the partial composition over assemblies of $\ComMag$.

    Let now be $w$ of the form $\sum_{i=1}^{l}a_iw_i$ where for each $1\leq i\leq l$, $a_i\in \K$ and there is a partition of $\pi_i=\set{V_{i,1},\dots, V_{i,k_i}}$ 
    of $V$ such that $w_i=(\dots(\mu_{\pi_i}\comp[ ]^\xi (t_{i,1}\dots t_{i,k_i})$ with $t_{i,j}\in\ComMag[V_{i,j}]$.
    For $i\in[l]$, the image of $\mu_{\pi_i}$ by $\phi$ is the empty graph over $\pi_i$, and so the image of $w_i$ is equal to: 
    \begin{math}
        \bigsqcup_{j=1}^{k_i}\phi(t_{i,j}).
    \end{math}
    Hence, for $i\not=j$ two indices, if $\pi_i\not= \pi_j$ the support of $\phi(w_i)$ and $\phi(w_j)$ are disjoint. We can then restrict ourselves 
    to the case where all the $w_i$ are on the same partition of $V$ \textit{i.e.} $\pi_i=\set{V_1,\dots, V_k}$ for all indices $i$.

    Denote by $\G[V_1,\dots,V_k]$ the vector space $\K\set{g_1\cup\dots\cup g_k\,|\,g_i\in\G[V_i]}$. Then there is an isomorphism
    from $\G[V_1,\dots,V_k]$ to $\G[V_1]\otimes\dots\otimes\G[V_k]$ defined by $g_1\cup\dots\cup g_k \mapsto g_1\otimes\dots\otimes g_k$
    which sends $\phi(w)$ on $\sum_{i=1}^l a_i \bigotimes_{j=1}^{k}\phi(t_{i,j})$. By definition, $\phi$ sends the elements $t_{i,j}$ on
    images of the basis elements of $\ComMag[V_i]\hookrightarrow\G[V_i]$ and hence form a free family in $\G[V_i]$. Hence the tensor products 
    $\bigotimes_{j=1}^{k}\phi(t_{i,j})$ also form a free family of $\G[V_1]\otimes\dots\otimes\G[V_k]$ and so $\phi(w)=0$ if and only if $w=0$.
    This concludes the proof.
\end{proof}

\begin{remark}
    The elements of $\Com(\ComMag)$ are assemblies of $\ComMag$ which can be interpreted as forest of binary trees.
\end{remark}

From now on we identify $\Points{a}{b}$ and $\Segment{a}{b}$ with their respective image in $\Com$ and $\ComMag$: $\mu_{\set{a,b}}$ and $s_{\set{a,b}}$. 
We now exhibit the Koszul dual of $\SP$. 

\begin{proposition}\label{sp_dual}
    The operad $\SP$ admits as Koszul dual the operad $\SP^!$ which is isomorphic to the operad $\Ope((\mathcal{G})^{\vee}, \spe{R})$ where 
    $\spe{R}$ is the subspecies of $\free{\mathcal{G}^\vee}$ generated by
    \begin{subequations}
        \begin{equation} \label{equ:rel_dual_1}
            \Segment{a}{\ast}^{\vee} \comp^{\xi} \Segment{b}{c}^{\vee},
        \end{equation}
        \begin{equation} \label{equ:rel_dual_2}
            \Points{a}{\ast}^{\vee} \comp^{\xi} \Segment{b}{c}^{\vee}
            \enspace + \enspace
            \Segment{c}{\ast}^{\vee} \comp^{\xi} \Points{a}{b}^{\vee}
            \enspace + \enspace
            \Segment{b}{\ast}^{\vee} \comp^{\xi} \Points{a}{c}^{\vee},
        \end{equation}
        \begin{equation} \label{equ:rel_dual_3}
            \Points{a}{\ast}^{\vee} \comp^{\xi} \Points{b}{c}^{\vee}
            +
            \Points{c}{\ast}^{\vee} \comp^{\xi} \Points{a}{b}^{\vee}
            +
            \Points{b}{\ast}^{\vee} \comp^{\xi} \Points{c}{a}^{\vee}.
        \end{equation}
    \end{subequations}
\end{proposition}

\begin{proof}
    Let us respectively denote by $r_1$ and $r_2$ and $r'_1$, $r'_2$, and $r'_3$ the vectors \eqref{equ:rel_1}, \eqref{equ:rel_2}, \eqref{equ:rel_dual_1}, 
    \eqref{equ:rel_dual_2}, and~\eqref{equ:rel_dual_3}. Denote by $\mathcal{I}$ the operad ideal generated by~$r_1$ and~$r_2$. As a vector space, 
    $\mathcal{I}[[\set{a,b,c}]]$ is then the linear span of the set
    \begin{equation}
        \set{r_1, (ab)\cdot r_1, r_2, (abc) \cdot r_2, (acb) \cdot r_2},
    \end{equation}
    where $\cdot$ is the action of the symmetric group, \textit{e.g} $r_1\cdot (ab) = \free{\mathcal{G}}[(ab)](r_1)$. 
    This space is a sub-space of dimension $5$ of $\free{\mathcal{G}}[\set{a,b,c}]$, which is of dimension $12$. Hence, since as a vector space we have
    \begin{equation}
        \free{\mathcal{G}^\vee}[\set{a,b,c}]
        \cong \free{\mathcal{G}^*}[\set{a,b,c}]\cong
        \free{\mathcal{G}}[\set{a,b,c}],
    \end{equation}
    we conclude that $\mathcal{I}^{\bot}[\set{a,b,c}]$ must be of dimension 7.

    Denote by $\mathcal{J}$ the ideal generated by $r_1'$, $r_2'$ and $r_3'$. As a vector space ,$\mathcal{J}[\set{a,b,c}]$ is then the linear span of the set
    \begin{equation}
        \set{ r_1', (ab) \cdot r_1', (ac) \cdot r_1',r_2', (abc)\cdot r_2', (acb)\cdot r_2', r_3'}.
    \end{equation}
    This vector space is of dimension 7. To conclude, we need to show that for any elements $f$ of $\mathcal{J}[\set{a,b,c}]$ and $x$ of $\mathcal{I}[\set{a,b,c}]$ we have 
    $\scalar{f}{x}=0$. For every pair $\alpha < \beta$  in $\set{a,b,c,\ast}$ ordered with the alphabetical order ($c<\ast$) denote 
    by $s_{\alpha\beta}^{\vee}$ the dual of $s_{\set{\alpha,\beta}}$ and by $\mu_{\alpha\beta}^{\vee}$ the dual of $\mu_{\set{\alpha,\beta}}$. Among the 21 cases to 
    check, we have for example:
    \begin{equation}\begin{split}
        \scalar{ r_1'}{r_1 } &=
        \scalar{ s_{a\ast}^{\vee}\comp^{\xi} s_{bc}^{\vee}}{\mu_{\set{\ast,c}}\comp^{\xi}\mu_{\set{a,b}} -
        \mu_{\set{a,\ast}}\comp^{\xi} \mu_{\set{b,c}} } \\
        &= \scalar{ s_{a\ast}^{\vee}\comp^{\xi} s_{bc}^{\vee}}{ \mu_{\set{\ast,c}}\comp^{\xi}\mu_{\set{a,b}} } - \scalar{ s_{a\ast}^{\vee}\comp^{\xi} s_{bc}^{\vee}}{\mu_{\set{a,\ast}}\comp^{\xi} \mu_{\set{b,c}}} \\
        &= s_{a\ast}^{\vee}(\mu_{\set{\ast,c}})s_{bc}^{\vee}(\mu_{\set{a,b}}) -
        s_{a\ast}^{\vee}(\mu_{\set{a,\ast}})s_{bc}^{\vee}(\mu_{\set{b,c}}) = 0,
    \end{split}\end{equation}
    and
    \begin{equation}\begin{split}
        \scalar{ (abc)\cdot r_2')&}{r_2}  = \\ 
        &\scalar{
            \mu_{b\ast}^{\vee}\comp^{\xi}s_{ca}^{\vee}+s_{a\ast}^{\vee}\comp^{\xi}\mu_{bc}^{\vee}+s_{c\ast}^{\vee}\comp^{\xi}\mu_{ab}^{\vee}
            }{
            s_{\set{a,\ast}}\comp \mu_{\set{b,c}}-\mu_{\set{c,\ast}})\comp s_{\set{a,b}}-\mu_{\set{b,\ast}}\comp s_{\set{c,a}}
            }  \\
        &= \mu_{b\ast}^{\vee}(s_{\set{a,\ast}})s_{ca}^{\vee}(\mu_{\set{b,c}}) 
        -\mu_{b\ast}^{\vee}(\mu_{\set{c,\ast}}))s_{ca}^{\vee}(s_{\set{a,b}}) 
        -\mu_{b\ast}^{\vee}(\mu_{\set{b,\ast}})s_{ca}^{\vee}(s_{\set{c,a}}) \\
        &+ s_{a\ast}^{\vee}(s_{\set{a,\ast}})\mu_{bc}^{\vee}(\mu_{\set{b,c}}) 
        - s_{a\ast}^{\vee}(\mu_{\set{c,\ast}}))\mu_{bc}^{\vee}(s_{\set{a,b}}) 
        - s_{a\ast}^{\vee}(\mu_{\set{b,\ast}})\mu_{bc}^{\vee}(s_{\set{c,a}}) \\
        &+ s_{c\ast}^{\vee}(s_{\set{a,\ast}})\mu_{ab}^{\vee}(\mu_{\set{b,c}}) 
        - s_{c\ast}^{\vee}(\mu_{\set{c,\ast}}))\mu_{ab}^{\vee}(s_{\set{a,b}}) 
        - s_{c\ast}^{\vee}(\mu_{\set{b,\ast}})\mu_{ab}^{\vee}(s_{\set{c,a}}) \\
        &= -1 +1 = 0.
    \end{split}\end{equation}
    We leave the verification of the 19 remaining cases as an exercise to the interested reader.
\end{proof}

In order to compute the Hilbert series of $\SP^!$ we need to use identity \eqref{hdual}
and hence to prove that the operad $\SP$ is Koszul.

\begin{proposition}\label{sp_koszul}
    The operad $\SP$ is Koszul.
\end{proposition}

\begin{proof}
Let $\mathcal{R}$ be the species defined as in Proposition~\ref{prop_sp} so that $\SP\cong\Ope(\K\mathcal{G},\mathcal{R})$.
Denote by $\mathfrak{p}(a,b)$ and $\mathfrak{s}(a,b)$ the elements of $\mathcal{G}^{\mathscr{F}}[\set{a,b},ab]$. Then the following vectors form a basis 
$\mathcal{B}$ of $\mathcal{R}^{\mathscr{F}}[\set{a,b,c},abc]$:
\begin{align}
    \mathfrak{v}_1 =\mathfrak{p}(\mathfrak{p}(a,b),c) - \mathfrak{p}(a,\mathfrak{p}(b,c))  \enspace &,\enspace 
    \mathfrak{v}_2=\mathfrak{p}(\mathfrak{p}(a,c),b) - \mathfrak{p}(a,\mathfrak{p}(b,c))\label{vec_points}\\
    \mathfrak{v'}_1=\mathfrak{s}(\mathfrak{p}(a,b),c) - &\mathfrak{p}(\mathfrak{s}(a,c),b) - \mathfrak{p}(a,\mathfrak{s}(b,c)) \label{vec_seg1} \\
    \mathfrak{v'}_2=\mathfrak{s}(\mathfrak{p}(a,c),b) - &\mathfrak{p}(\mathfrak{s}(a,b),c) - \mathfrak{p}(a, \mathfrak{s}(b,c))  \label{vec_seg2}\\
    \mathfrak{v'}_3=\mathfrak{s}(a,\mathfrak{p}(b,c)) - &\mathfrak{p}(\mathfrak{s}(a,b),c) - \mathfrak{p}(\mathfrak{s}(a,c),b).\label{vec_seg3}
\end{align}
We need to show that it is a Gröbner bases of $(\mathcal{R}^{\mathscr{F}})$. 
Let now consider the path-lexicographic ordering presented in subsubsection~\ref{sec_koszul} with $s>p$. Then the leading terms of $\mathfrak{v}_1$, $\mathfrak{v}_2$, $\mathfrak{v'}_1$, $\mathfrak{v'}_2$ 
and $\mathfrak{v'}_3$ are respectively $\mathfrak{p}(\mathfrak{p}(a,b),c)$, $\mathfrak{p}(\mathfrak{p}(a,c),b)$, $\mathfrak{s}(\mathfrak{p}(a,b),c)$, 
$\mathfrak{s}(\mathfrak{p}(a,c),b)$ and $\mathfrak{s}(a,p(b,c))$. We conclude with Proposition~\ref{prop_spol}. Indeed, it is shown in \cite{DK10} 
that the S-polynomials of pairs of elements in $\set{\mathfrak{v}_1,\mathfrak{v}_2}$ are congruent to zero modulo $\mathcal{B}$. We show for example that 
the S-polynomial of $ \mathfrak{v}_1$ and $\mathfrak{v'}_1$ corresponding to 
$\mathfrak{c}=\mathfrak{s}(\mathfrak{p}(\mathfrak{p}(a,b),c),d)\in\free{\mathcal{G}}^{sh}[\set{a,b,c,d},abcd]$ is congruent to zero.
We have
\begin{equation}\begin{split}
    m_{\mathfrak{c},\text{lt}(\mathfrak{v}_1)}(\mathfrak{v}_1) &= \mathfrak{c}-\mathfrak{s}(\mathfrak{p}(a,\mathfrak{p}(b,c)),d)\\
    m_{\mathfrak{c},\text{lt}(\mathfrak{v'}_1)}(\mathfrak{v'}_1) &= \mathfrak{c} - \mathfrak{p}(\mathfrak{s}(\mathfrak{p}(a,b),d),c)
    -\mathfrak{p}(\mathfrak{p}(a,b),\mathfrak{s}(c,d)),
\end{split}\end{equation}
which gives us
\begin{equation}\label{proof_spol}
    s_{\mathfrak{c}}(\mathfrak{v}_1,\mathfrak{v'}_1) =  \mathfrak{p}(\mathfrak{s}(\mathfrak{p}(a,b),d),c) + \mathfrak{p}(\mathfrak{p}(a,b),\mathfrak{s}(c,d))
    -\mathfrak{s}(\mathfrak{p}(a,\mathfrak{p}(b,c)),d).
\end{equation}
Let us look how each of the terms of $s_{\mathfrak{c}}(\mathfrak{v}_1,\mathfrak{v'}_1)$ reduces modulo $\mathcal{B}$:
\begin{equation}\begin{split}
    \mathfrak{p}(\mathfrak{s}(\mathfrak{p}(a,b),d),c)
    &\equiv_{\mathfrak{v'}_1} \mathfrak{p}(\mathfrak{p}(\mathfrak{s}(a,d),b),c) + \mathfrak{p}(\mathfrak{p}(a,\mathfrak{s}(b,d)),c) \\
    &\equiv_{\mathfrak{v}_1} \mathfrak{p}(\mathfrak{s}(a,d),\mathfrak{p}(b,c)) + \mathfrak{p}(a, \mathfrak{p}(\mathfrak{s}(b,d),c), \\
    & \\
    \mathfrak{p}(\mathfrak{p}(a,b),\mathfrak{s}(c,d)) &\equiv_{\mathfrak{v}_1} \mathfrak{p}(a, \mathfrak{p}(b,\mathfrak{s}(c,d)), \\
    & \\
    \mathfrak{s}(\mathfrak{p}(a,\mathfrak{p}(b,c)),d) &\equiv_{\mathfrak{v'}_1} \mathfrak{p}(\mathfrak{s}(a,d),\mathfrak{p}(b,c)) +
    \mathfrak{p}(a,\mathfrak{s}(\mathfrak{p}(b,c),d))\\
    &\equiv_{\mathfrak{v'}_1} \mathfrak{p}(\mathfrak{s}(a,d),\mathfrak{p}(b,c)) + \mathfrak{p}(a,\mathfrak{p}(\mathfrak{s}(b,d),c)) +
    \mathfrak{p}(a,\mathfrak{p}(b,\mathfrak{s}(c,d))).
\end{split}\end{equation}
Putting this together in \eqref{proof_spol} gives us that $s_{\mathfrak{c}}(\mathfrak{v}_1,\mathfrak{v'}_1)$ reduces to $0$ modulo $\mathcal{B}$.
We leave the verification of the other cases to the interested reader.
\end{proof}

\begin{proposition} \label{prop_sp_hilbert}
    The Hilbert series of $\SP^{!}$ is given
    \begin{equation}
        \hilbert{\SP^!}(x) = \dfrac{(1-\log(1-x))^2-1}{2}.
    \end{equation}
\end{proposition}

\begin{proof}
    The Hilbert series of $\ComMag$ is $\hilbert{\ComMag}(x) = 1-\sqrt{1-2x}$ 
    hence the Hilbert series of $\SP\cong \Com(\ComMag)$ is $\mathcal{H}_{\SP}(x)=e^{1-\sqrt{1-2x}}-1$, 
    where the $-1$ comes from the fact that we consider positive species. 
    We deduce the Hilbert series of $\SP^!$ from $\hilbert{\SP}$ and the identity \eqref{hdual}.
\end{proof}
The first dimensions $\dim \SP^![[n]]$ for $n\geq 1$ are
\begin{equation}
    1, 2, 5, 17, 74, 394, 2484, 18108, 149904.
\end{equation}
This is sequence~\OEIS{A000774} of~\cite{Slo}. This sequence is in particular linked to some
pattern avoiding signed permutations and mesh patterns.

Before ending this section let us mention the sub-operad $\LP$ of $\MG$ generated by
\begin{equation}
    \left\{
    \begin{tikzpicture}[Centering,scale=.6]
        \tikzset{every loop/.style={}}
        \node[NodeGraph](a)at(0,0){$a$};
        \draw[EdgeGraph](a)edge[loop](a);
    \end{tikzpicture},
    \begin{tikzpicture}[Centering,scale=.6]
        \node[NodeGraph](a)at(0,0){$a$};
        \node[NodeGraph](b)at(1,0){$b$};
    \end{tikzpicture}
    \right\}.
\end{equation}
This operad seems particularly interesting to us since its two generators can be considered as minimal
elements in the sense that a partial composition with the two isolated vertices adds exactly
one vertex and no edge, while a partial composition with the loop adds exactly one edge and
no vertex.  A natural question to ask at this point concerns the description of the
multigraphs generated by these two minimal elements.

\begin{proposition}
The following properties hold:
\begin{itemize}
\item the operad $\SP$ is a sub-operad of $\LP$;
\item the operad $\LP$ is a strict sub-operad of $\MG$. In particular, the multigraph
\begin{equation}
    \begin{tikzpicture}[Centering,scale=.8]
        \node[NodeGraph](a)at(0,0){$a$};
        \node[NodeGraph](b)at(1,0){$b$};
        \node[NodeGraph](c)at(2,0){$c$};
        \draw[EdgeGraph](a)--(b);
        \draw[EdgeGraph](b)edge[bend left=40](c);
        \draw[EdgeGraph](b)edge[bend right=40](c);
    \end{tikzpicture}
\end{equation}
is in $\MG$ but is not in $\LP$.
\end{itemize}
\end{proposition}

\begin{proof}
    \begin{itemize}
        \item The following identity shows that 
        \begin{math}
        \begin{tikzpicture}[Centering,scale=.6]
            \node[NodeGraph](a)at(0,0){$a$};
            \node[NodeGraph](b)at(1,0){$b$};
            \draw[EdgeGraph](a)--(b);
        \end{tikzpicture}
        \end{math}
        is in $\LP[\set{a,b}]$ and hence that $\SP$ is a sub-operad of $\LP$:
        \begin{equation}
            \begin{tikzpicture}[Centering,scale=.6]
                \tikzset{every loop/.style={}}
                \node[NodeGraph](a)at(0,0){$\ast$};
                \draw[EdgeGraph](a)edge[loop](a);
            \end{tikzpicture}
            \comp \enspace
            \Points{a}{b}
             - 
            \begin{tikzpicture}[Centering,scale=.6]
                \tikzset{every loop/.style={}}
                \node[NodeGraph](a)at(0,0){$a$};
                \draw[EdgeGraph](a)edge[loop](a);
            \end{tikzpicture}
             - 
            \begin{tikzpicture}[Centering,scale=.6]
                \tikzset{every loop/.style={}}
                \node[NodeGraph](a)at(0,0){$b$};
                \draw[EdgeGraph](a)edge[loop](a);
            \end{tikzpicture}
            \enspace = \enspace
            2 \enspace \Segment{a}{b}.
        \end{equation}
        \item Using computer algebra, one generates all vectors in $\LP[\set{a,b,c}]$ with three edges
        and shows that the announced multigraph is not a linear combination of these.
    \end{itemize}
\end{proof}


\section{Graph insertion operads}\label{sec_graph_operads}
The goal of this section is to give a general construction of operads on multigraphs where the partial composition of two elements $g_1\comp g_2$ is given by
\begin{enumerate}
    \item taking the disjoint union of $g_1$ and $g_2$;
    \item removing the vertex $\ast$ from $g_1$;
    \item connecting \underline{independently} each loose ends of $g_1$ to $g_2$ in a certain way.
\end{enumerate}
\noindent What we mean by independently is that the way of connecting one end does not depend on how we connect the other ends. Note that the ``certain way'' 
in which an end can be connected may include duplication of edges.


\subsection{Constructions on species and operads}\label{sec_constructions}

We begin by defining new constructions on species and operads. We define three constructions:
the augmentation, the semi-direct product and the maps from a set to an operad.

\begin{definition}
    Let $A$ be a set and $\spe{S}$ be a species. An {\em $A$-augmentation} of $\spe{S}$ is a species
     $\augm{A}{\spe{S}}$ such that $\augm{A}{\spe{S}}[V] \cong \spe{S}[A\times V]$ for every finite set $V$.
\end{definition}

\begin{example}\label{ex_deform}
    Let $A$ be a set.
    \begin{itemize}
        \item Instead of considering an $A$-augmented multigraph on $V$ as a multigraph on $V\times A$, we consider them 
        as multigraphs on $V$ where the ends of the edges are labelled with elements of $A$. In particular, the species of oriented multigraphs
        $\MG_{or}$ is in bijection with the species of $\set{\mathtt{s,t}}$-augmented multigraphs $\augm{\set{\mathtt{s,t}}}{\MG}$. For $g\in\MG$ and $\ori$ an
        orientation of $g$, the pair $(g,\ori)$ is sent on the multigraph obtained by respectively labeling by $t$ and $s$ the targets and sources
        of the edges.
        \item Instead of seeing the elements in $\augm{A}{\Poly}[V]$ as polynomials with set of variables the couples $(v,a)\in V\times A$, we consider them as 
        polynomials with set of variables $\set{v_a\,|\,v\in V,a\in A}$ of elements of $V$ indexed by elements of $A$.
    \end{itemize}
\end{example}

\begin{remark} \label{rem_bij_graphpol}
    For $\spe{S}$ and $\spe{R}$ any two species and $f:\spe{R}\to\spe{S}$ a morphism, $f$ extends to a morphism between any two $A$-augmentation of $\spe{R}$
    and $\spe{S}$ by $\augm{A}{\spe{R}}[V]\cong\spe{R}[A\times V]\stackrel{f}{\to}\spe{S}[A\times V]\cong\augm{A}{\spe{S}}$. In particular the morphism $\MG\to\Poly$ 
    given in subsubsection~\ref{graphpol} extends to a morphism which sends an edge $(u_a,v_b)$ on the monomial $u_av_b$.
\end{remark}

In the following proposition, we give an operad structure to a Hadamard product $\spe{S}\times\Operad$ where $\spe{S}$ is a species and $\Operad$ is
an operad. 
\begin{proposition}\label{prop_semiprod}
    Let $\spe{S}$ be a linear species and $\Operad$ an operad. Let $\varphi$ be a morphism from $\spe{S}'\cdot(\spe{S}\times\Operad)$ to $\spe{S}$ and denote by
    $x\comp^f y=\varphi(x\otimes y\otimes f)$. Suppose $\varphi$ satisfies the following hypotheses. 
    \begin{description}
        \item[Commutativity.] For $x$ an element $\spe{S}''$ and $y\otimes f$ and $z\otimes g$ two elements of $\spe{S}\times\Operad$,
        \begin{equation}\label{prod_sd_eq1}
            (x\comp[\ast_1]^f y)\comp[\ast_2]^g z  = (x\comp[\ast_2]^g z)\comp[\ast_1]^f y.
        \end{equation}
        \item[Associativity.] For $x$ an element of $\spe{S}'$, $y\otimes f$ an element of $(\spe{S}\times\Operad)'$ and $z\otimes h$ an element of $\spe{S}\times\Operad$,
        \begin{equation}\label{prod_sd_eq2} 
            (x\comp[\ast_1]^f y)\comp[\ast_2]^g z = x\comp[\ast_1]^{f\comp[\ast_2] g}(y\comp[\ast_2]^g z).
        \end{equation}
        \item[Unity.] There exists a map $e:\spe{X}\rightarrow \spe{S}$ such that 
        \begin{equation}\label{prod_sd_eq3}
            x\comp^{e_{\Operad}(v)}e(v) = \spe{S}[\sigma](x) \quad \text{ and }\quad e(\ast)\comp^f x = x,
        \end{equation} 
        where $e_{\Operad}$ is the unit of $\Operad$ and $\sigma$ is the bijection which sends $\ast$ on $v$ and is the identity on the rest of the set on
        which $x$ is defined.
    \end{description}
    Then the partial composition $\comp^{\varphi}$ defined by
    \begin{equation}\begin{split}
        \comp^{\varphi}: (\spe{S}\times\Operad)'\cdot \spe{S}\times\Operad &\to \spe{S}\times\Operad \\
        (x\otimes f)\otimes (y\otimes g) &\mapsto x\comp^g y\otimes f\comp g
    \end{split}\end{equation}
    makes $\spe{S}\times\Operad$ an operad with unit $e$. We call this operad the {\em semi-direct product of $\spe{S}$ and $\Operad$ over $\varphi$} and we 
    denote it by $\spe{S}\ltimes_{\varphi}\Operad$.
\end{proposition}

\begin{proof}
    We must verify that the three diagrams~\eqref{def_op} commutes.
    \begin{itemize} 
        \item Let $V_1,V_2,V_3$ be three disjoint sets and $x\otimes f\in(\spe{S}\times\Operad)''[V_1]$, $y\otimes g\in\spe{S}\times\Operad[V_2]$ and 
        $z\otimes h\in\spe{S}\times\Operad[V_3]$. We then have
        \begin{equation}\begin{split}
            \big((x\otimes f) \comp[\ast_1]^{\varphi} (y\otimes g)\big)\comp[\ast_2]^{\varphi} (z\otimes h) 
            &= \big( (x \comp[\ast_1]^g y) \comp[\ast_2]^h z \big)  \otimes  \big((f\comp[\ast_1]g)\comp[\ast_2]h \big) \\
            &= \big( (x \comp[\ast_2]^h z) \comp[\ast_1]^g y \big)  \otimes  \big((f\comp[\ast_2]h)\comp[\ast_1]g \big) \\
            &= \big( (x\otimes f) \comp[\ast_1]^{\varphi} (z\otimes h) \big) \comp[\ast_2]^{\varphi} (y\otimes g),
        \end{split}\end{equation}
        where the second equality follows from~\eqref{prod_sd_eq1} and the fact that $\Operad$ is an operad.
        \item Let $V_1,V_2,V_3$ be three disjoint sets and $x\otimes f\in(\spe{S}\times\Operad)'[V_1]$, $y\otimes g\in(\spe{S}\times\Operad)'[V_2]$ and 
        $z\otimes h\in\spe{S}\times\Operad[V_3]$. We then have
        \begin{equation}\begin{split}
            \big((x\otimes f) \comp[\ast_1]^{\varphi} (y\otimes g)\big)\comp[\ast_2]^{\varphi} (z\otimes h) 
            &= \big( (x\comp[\ast_1]^g y)\comp[\ast_2]^h z \big)  \otimes  \big( (f\comp[\ast_1]g)\comp[\ast_2]h \big) \\
            &= \big( x\comp[\ast_1]^{g\comp[\ast_1]h}( y\comp[\ast_2]^h z) \big)  \otimes  f\comp[\ast_1](g\comp[\ast_2]h) \big) \\
            &= (x\otimes f)\comp[\ast_1]^{\varphi}\big( (y\otimes g)\comp[\ast_2]^{\varphi}(z\otimes h)\big),
        \end{split}\end{equation} 
        where the second equality follows from~\eqref{prod_sd_eq2} and the fact that $\Operad$ is an operad.
        \item Let be $x\otimes f\in(\spe{S}\times\Operad)'[V]$ and $v\not\in V$. We then have 
        \begin{equation}
            (x\otimes f)\comp^{\varphi} (e(v)\otimes e_{\Operad}(v))= x\comp^{e_{\Operad}(v)}e(v) \otimes f\comp e(v) = \spe{S}[\sigma](x\otimes f)
        \end{equation}
        where $\sigma$ is the bijection which sends $\ast$ to $v$ and is the identity on $V$. The last equality follows from
        the first equality from \eqref{prod_sd_eq3} and the fact that $\Operad$ is an operad. Let now be $x\otimes f\in\spe{S}\times\Operad[V]$.
        Then
        \begin{equation}
            (e(\ast)\otimes e_{\Operad}(\ast))\comp (x\otimes f) = e(\ast)\comp^f x \otimes e_{\Operad}(\ast)\comp f = x\otimes f
        \end{equation}
        where the last equality comes from second equality from \eqref{prod_sd_eq3} and the fact that $\Operad$ is an operad.
    \end{itemize}
\end{proof}

When it is clear from the context, we do not mention $\varphi$ and just write semi-direct product of $\spe{S}$ and $\Operad$ and denote it by $\spe{S}\ltimes \Operad$.
In practice, the operad structure $\Operad$ of $\spe{S}\ltimes\Operad$ is transparent and we are just interested in what happens on $\spe{S}$, that it is to say
the ``pseudo partial composition'' $x\comp^g y$.

\begin{example}
    For $C$ a finite set, let $\spe{C}$ be the trivial species given by $\spe{C}[V]=\K C$ for all set $V$ and $\mathcal{C}=\spe{X}+\spe{C}_{2+}$. This second
    species has an operad structure given defined by, for $c_1\in \spe{C}'[V_1]$ and $c_2\in \spe{C}[V_2]$: $c_1\comp c_2 = c_1$ 
    if $V_1\noemp$ and $\ast\comp c_2 = c_2$ when $V_1=\emptyset$ and $\ast\in\spe{X}[\set{\ast}]$. 

    Let $\mathcal{F}^C = \spe{X} + \mathcal{F}_{2+}^C$ be the species of maps with co-domain $C$: $\mathcal{F}^C[V]=\K\set{f:V\to C}$ for $\card{V} > 1$. 
    We define a semi-direct product structure $\mathcal{F}^C\ltimes_{\varphi} \mathcal{C}$. Suppose that 
    $\card{V_1+\set{\ast}},\card{V_2}>1$ and let be $f\in \mathcal{F}^C[V_1+\set{\ast}]$ and $g\otimes x\in \mathcal{F}^C\times \mathcal{C}[V_2]$. We then have 
    $f\comp^c g = 0$ if $f(\ast) \not = c$ and $f\comp^c g(v) = \left\{\begin{array}{rl}
    f(v)  & \text{if $v\in V_1$}    \\ 
    g(v)  & \text{if $v\in V_2$}\end{array}\right.$ else. When $V_1=\emptyset$ or $V_2$ is a singleton, the action of $\varphi$ is implied by the unit hypothesis. 
   
    We call this operad the {\em $C$-coloration} operad. When this operad is considered alone, one can see an element of $(f,c)\in\mathcal{F}^C\ltimes C[V]$ as a 
    corolla on $V$ with its root colored by $c$ and its leaves $v\in V$ colored by $f(v)$. The partial composition consists then in grafting two corollas if the 
    root and the leaf on which it must be grafted share the same colors. For instance we have, by representing the elements of $C$ with colors:
    \begin{equation}
        \begin{tikzpicture}[Centering, scale=0.6]
            \node[NodeGraph, draw=Red, fill=Red!20](n)at(0,0){$ $};
            \node[NodeGraph, draw=Red, fill=Red!20](a)at(-1,1){$a$};
            \node[NodeGraph, draw=Green, fill=Green!20](b)at(-0.125,1.5){$b$};
            \node[NodeGraph, draw=Blue, fill=Blue!20](s)at(0.875,1.25){$\ast$};
            \node[NodeGraph, draw=Magenta, fill=Magenta!20](c)at(1.5,0.75){$c$};
            \draw[EdgeGraph, black](n)--(a);
            \draw[EdgeGraph, black](n)--(b);
            \draw[EdgeGraph, black](n)--(s);
            \draw[EdgeGraph, black](n)--(c);
        \end{tikzpicture}
        \enspace \comp \enspace
        \begin{tikzpicture}[Centering, scale=0.6]
            \node[NodeGraph, draw=Red, fill=Red!20](n)at(0,0){$ $};
            \node[NodeGraph, draw=Blue, fill=Blue!20](a)at(-0.5,1){$1$};
            \node[NodeGraph, draw=Magenta, fill=Magenta!20](b)at(0.5,1.25){$2$};
            \draw[EdgeGraph, black](n)--(a);
            \draw[EdgeGraph, black](n)--(b);
        \end{tikzpicture}
        \enspace = \enspace 0 \text{ and}
    \end{equation}
    \begin{equation}
        \begin{tikzpicture}[Centering, scale=0.6]
            \node[NodeGraph, draw=Red, fill=Red!20](n)at(0,0){$ $};
            \node[NodeGraph, draw=Red, fill=Red!20](a)at(-1,1){$a$};
            \node[NodeGraph, draw=Green, fill=Green!20](b)at(-0.125,1.5){$b$};
            \node[NodeGraph, draw=Blue, fill=Blue!20](s)at(0.875,1.25){$\ast$};
            \node[NodeGraph, draw=Magenta, fill=Magenta!20](c)at(1.5,0.75){$c$};
            \draw[EdgeGraph, black](n)--(a);
            \draw[EdgeGraph, black](n)--(b);
            \draw[EdgeGraph, black](n)--(s);
            \draw[EdgeGraph, black](n)--(c);
        \end{tikzpicture}
        \enspace \comp \enspace
        \begin{tikzpicture}[Centering, scale=0.6]
            \node[NodeGraph, draw=Blue, fill=Blue!20](n)at(0,0){$ $};
            \node[NodeGraph, draw=Blue, fill=Blue!20](a)at(-0.5,1){$1$};
            \node[NodeGraph, draw=Magenta, fill=Magenta!20](b)at(0.5,1.25){$2$};
            \draw[EdgeGraph, black](n)--(a);
            \draw[EdgeGraph, black](n)--(b);
        \end{tikzpicture}
        \enspace = \enspace \left(
        \begin{tikzpicture}[Centering, scale=0.6]
            \node[NodeGraph, draw=Red, fill=Red!20](n)at(0,0){$ $};
            \node[NodeGraph, draw=Red, fill=Red!20](a)at(-1,1){$a$};
            \node[NodeGraph, draw=Green, fill=Green!20](b)at(-0.125,1.5){$b$};
            \node[NodeGraph, draw=Blue, fill=Blue!20](s)at(0.875,1.25){$\ast$};
            \node[NodeGraph, draw=Magenta, fill=Magenta!20](c)at(1.5,0.75){$c$};
            \node[NodeGraph, draw=Blue, fill=Blue!20](1)at(0.375,2.25){$1$};
            \node[NodeGraph, draw=Magenta, fill=Magenta!20](2)at(1.375,2.5){$2$};
            \draw[EdgeGraph, black](n)--(a);
            \draw[EdgeGraph, black](n)--(b);
            \draw[EdgeGraph, black](n)--(s);
            \draw[EdgeGraph, black](n)--(c);
            \draw[EdgeGraph, black](s)--(1);
            \draw[EdgeGraph, black](s)--(2);
        \end{tikzpicture} \right)
        \enspace = \enspace
            \begin{tikzpicture}[Centering, scale=0.6]
            \node[NodeGraph, draw=Red, fill=Red!20](n)at(0,0){$ $};
            \node[NodeGraph, draw=Red, fill=Red!20](a)at(-1,1){$a$};
            \node[NodeGraph, draw=Green, fill=Green!20](b)at(-0.25,1.5){$b$};
            \node[NodeGraph, draw=Blue, fill=Blue!20](1)at(0.5,1.25){$1$};
            \node[NodeGraph, draw=Magenta, fill=Magenta!20](2)at(1.25,0.875){$2$};
            \node[NodeGraph, draw=Magenta, fill=Magenta!20](c)at(1.5,0.25){$c$};
            \draw[EdgeGraph, black](n)--(a);
            \draw[EdgeGraph, black](n)--(b);
            \draw[EdgeGraph, black](n)--(c);
            \draw[EdgeGraph, black](n)--(1);
            \draw[EdgeGraph, black](n)--(2);
        \end{tikzpicture}
    \end{equation}
    A way to define {\em colored operads}~(see \cite{DY16} for more details on the theory of colored operads) is then to define them as any Hadamard product of a 
    $C$-coloration operad with another operad.
\end{example}

Let us now define our last construction.

\begin{definition}
    Let $A$ be a set and $\spe{S}$ be a species. The set species of {\em functions from $A$ to $\spe{S}$} is defined by $\func{A}{\spe{S}}[V] = \K\set{f:A\to \spe{S}[V]}$.
\end{definition}

The following proposition then tells us that if $\Operad$ has an operad structure, it naturally reflects on $\func{A}{\Operad}$.

\begin{proposition}\label{prop_func}
    If $\Operad$ is an operad with unit $e$, $\func{A}{\Operad}$ has an operad structure with the elements $e_v:A \to \set{e(v)}\in\func{A}{\Operad}[\set{v}]$ 
    as units and partial composition defined by $f_1\comp f_2(a) = f_1(a)\comp f_2(a)$.
\end{proposition}

\begin{proof}
    We must verify that the diagrams~\eqref{def_op} are indeed commutative.
    \begin{itemize}
        \item Let be $f_1\in(\func{A}{\Operad})''[V_1]$, $f_2\in\func{A}{\Operad}[V_2]$ and $f_3\in\func{A}{\Operad}[V_3]$. Then for all $a\in A$, we have
        \begin{equation}\begin{split}
            \big((f_1\comp[\ast_1] f_2)\comp[\ast_2] f_3\big)(a) &= (f_1\comp[\ast_1] f_2(a))\comp[\ast_2] f_3(a) \\
            &= \big(f_1(a)\comp[\ast_1] f_2(a)\big)\comp f_3(a) \\
            &= f_1(a)\comp[\ast_2] f_3(a)\comp f_2(a) \\
            &= (f_1\comp[\ast_2] f_3(a))\comp[\ast_1] f_2(a) \\
            &= \big((f_1\comp[\ast_2] f_3)\comp[\ast_1] f_2\big)(a),
        \end{split}\end{equation}
        where the third equality follows from the fact that $\Operad$ is an operad. Hence, we have $(f_1\comp[\ast_1]f_2)\comp[\ast_2] f_3 = (f_1\comp[\ast_2]f_3)\comp[\ast_1]f_2$.
        \item Let be $f_1\in(\func{A}{\Operad})'[V_1]$, $f_2\in(\func{A}{\Operad})'[V_2]$ and $f_3\in\func{A}{\Operad}[V_3]$. Then for all $a\in A$:
        \begin{equation}\begin{split}
            \big((f_1\comp[\ast_1] f_2)\comp[\ast_2] f_3\big)(a) &= (f_1\comp[\ast_1] f_2(a))\comp[\ast_2] f_3(a) \\
            &= f_1(a)\comp[\ast_1] f_2(a)\comp f_3(a) \\
            &= f_1(a)\comp[\ast_2] f_3(a)\comp f_2(a) \\
            &= (f_1\comp[\ast_2] f_3(a))\comp[\ast_1] f_2(a) \\
            &= \big((f_1\comp[\ast_2] f_3)\comp[\ast_1] f_2\big)(a),
        \end{split}\end{equation}
        where the third equality comes from the fact that $\Operad$ is an operad. Hence, we have $(f_1\comp[\ast_1]f_2)\comp[\ast_2] f_3 = (f_1\comp[\ast_2] f_3)\comp[\ast_1] f_2$.
        \item Let be $f\in(\func{A}{\Operad})'[V]$ and $v\not\in V$. Then for all $a\in A$, we have the equalities $f\comp e_v(a)= f(a)\comp e(v) = \Operad[\sigma](f(a))$ 
        and so $f\comp e_v=\Operad[\sigma](f)$, where $\sigma$ is the bijection which sends $\ast$ to $v$ and is the identity over $V$. If now $f\in\func{A}{\Operad}$, we have
        for all $a\in A$: $e_{\ast}\comp f(a)= e(\ast)\comp f(a) = f(a)$ and so $e_{\ast}\comp f = f$.
    \end{itemize}
\end{proof}

Note that if $A$ is a singleton then $\func{A}{\Operad} \cong \Operad$. Let $A,B,C,D$ four sets such that $A$ and $B$ are disjoint and $f:A\to C$ 
and $g:B\to D$ two maps. We denote by $f\uplus g$ the map from $A\sqcup B$ to $C\cup D$ defined by $f\uplus g(a) = f(a)$ for $a\in A$ and 
$f\uplus g(b) = g(b)$ for $b\in B$. 

\begin{proposition}
    Let $A$ and $B$ be two disjoint sets and $\Operad_1$ and $\Operad_2$ be two operads. Then the species $\func{A,B}{\Operad_1,\Operad_2}$ defined by 
    $\func{A,B}{\Operad_1,\Operad_2}[V] = \set{f\uplus g\,|\, f\in\func{A}{\Operad_1}, g\in\func{B}{\Operad_2}}$ is an operad with same partial composition
    than in Proposition~\ref{prop_func}.
\end{proposition}

\begin{proof}
    We remark that since $A$ and $B$ are disjoint, $f_1\uplus f_2\comp g_1\uplus g_2 = (f_1\comp g_1)\uplus (f_2\comp g_2)$. To conclude we apply what was
    already shown in the proof of Proposition~\ref{prop_func}.
\end{proof}


\subsection{Application to multigraphs}\label{subsec_graph_op}

We now use the construction of the previous subsection to define operad structures on multigraphs. 

Recall from Remark~\ref{rem_bij_graphpol} that there is a monomorphism from $\augm{A}{\MG}$ to $\augm{A}{\Poly}$. We now make use of this monomorphism
and consider the elements of $\augm{A}{\MG}$ as both multigraphs and polynomials. Let $p\in\Poly[V]$ be a sum of polynomials.
Then for $A$ a set and $a\in A$, we denote by $p_a\in\augm{A}{\Poly}$ the sum of polynomials obtained by indexing all the variables in $p$ by $a$. 
Let now $p$ be any polynomial and $x_1,\dots, x_n$ a subset of its variables. Then for $q_1,\dots, q_n$ $n$ polynomials, we expand the notation 
introduced in equation \eqref{pol_compositioin} and denote by $p|_{\set{x_i\take q_i}}$ the polynomial 
$\big(\dots((p_1\comp[x_1] q_1)\comp[x_2] q_2)\dots\big)\comp[x_n] q_n$. 
This notation generalize to sum of polynomials by recalling that the multiplication and addition of polynomials act as bilinear maps.

\begin{example}
    For $A$ the singleton $\set{a}$ and $p$ the polynomial $xy\oplus x^2 + zy\in\Poly[\set{x,y,z}]$ we have $p_a= x_ay_a\oplus x_a^2 + z_ay_a$. 
    Let now be $p=x\oplus yz$, $q_1 = u+v$ and
    $q_2=x_1\oplus x_2$. Then
    \begin{equation}\begin{split}
        p|_{x\take q_1, y\take q_2} &= q_1\oplus q_2z = (u+v)\oplus (x_1\oplus x_2)z \\
        &= u\oplus x_1z\oplus x_2z + v\oplus x_1z\oplus x_2z.
    \end{split}\end{equation}
\end{example}

\begin{theorem}\label{th_op_ins}
    Let $A$ be a set and $\varphi$ be the morphism from $\augm{A}{\MG}\cdot(\augm{A}{\MG}\times\func{A}{\K\Poly^1})$ to $\augm{A}{\MG}$ given by
    \begin{equation}
        \varphi(g_1\otimes g_2 \otimes f) = g_1\comp^f g_2 = g_1|_{\set{\ast_a\take f(a)_a}}\oplus g_2.
    \end{equation}
    Then $\varphi$ satisfies the hypotheses of Proposition~\ref{prop_semiprod} and we can consider the semidirect product of $\augm{A}{\MG}$ and 
    $\mathcal{F}_A^{\K \Poly^1}$ over $\varphi$.
\end{theorem}

\begin{proof}
    We need to check that $\varphi$ satisfies the three hypotheses of Proposition~\ref{prop_semiprod}. The first two are simply computations over polynomials.
    \begin{description}
        \item[Commutativity.] Let $g_1$ be an element of $\augm{A}{\MG}''$ and $g_2\otimes f$ and $g_3\otimes h$ two elements of $\augm{A}{\MG}\times\func{A}{\K\Poly^1}$.
        Then
        \begin{equation}\begin{split}
            (g_1\comp[\ast_1]^f g_2)\comp[\ast_2]^h g_3 
            &= (g_1|_{\set{\ast_{1a}\take f(a)_a}}\oplus g_2)|_{\{\ast_{2a}\take h(a)_a\}}\oplus g_3 \\
            &= g_1|_{\set{\ast_{1a}\take f(a)_a}}|_{\set{\ast_{2a}\take h(a)_a}}\oplus g_2\oplus g_3 \\
            &= g_1|_{\set{\ast_{2a}\take h(a)_a}}|_{\set{\ast_{1a}\take f(a)_a}}\oplus g_3\oplus g_2 \\
            &= (g_1|_{\set{\ast_{2a}\take h(a)_a}}\oplus g_3)|_{\set{\ast_{1a}\take f(a)_a}}\oplus g_2 \\
            &= (g_1\comp[\ast_2]^h g_3)\comp[\ast_1]^f g_2 .
        \end{split}\end{equation}
        \item[Associativity.] Let be $g_1$ an element of $\augm{A}{\MG}'$, $g_2\otimes f$ an element of $(\augm{A}{\MG}\times\func{A}{\K\Poly^1})'$ and $g_3\otimes h$ 
        an element of $\augm{A}{\MG}\times\func{A}{\K\Poly^1}$. Then
        \begin{equation}\begin{split} 
            (g_1\comp[\ast_1]^f g_2)\comp[\ast_2]^h g_3
            &= (g_1|_{\set{\ast_{1a}\take f(a)_a}}\oplus g_2)|_{\set{\ast_{2a}\take h(a)_a}}\oplus g_3 \\
            &= g_1|_{\set{\ast_{1a}\take f(a)_a}}|_{\set{\ast_{2a}\take h(a)_a}}\oplus g_2|_{\set{\ast_{2a}\take h(a)_a}}\oplus g_3 \\
            &= g_1|_{\set{\ast_{1a}\take f(a)_a|_{\set{\ast_{2a}\take h(a)_a}}}}\oplus g_2|_{\set{\ast_{2a}\take h(a)_a}}\oplus g_3 \\
            &= g_1|_{\set{\ast_{1a}\take f\comp[\ast_2] h(a)_a}}\oplus g_2|_{\set{\ast_{2a}\take h(a)_a}}\oplus g_3 \\
            &= g_1\comp[\ast_1]^{f\comp[\ast_1] g}(g_2\comp[\ast_2]^h g_3).
        \end{split}\end{equation}
        \item[Unity.] Let $e: \spe{X}\to \augm{A}{\MG}$ be defined by $e(v) = \emptyset_{\set{v}}$ and let $e_{\mathcal{F}}$ the unit of $\func{A}{\K\Poly^1}$.
        Let be $g\in\augm{A}{\MG}'[V]$ and $v\not\in V$. Then
        \begin{equation}\begin{split}
            g\comp^{e_{\mathcal{F}}(v)}e(v) 
            &= g|_{\set{\ast_a \take e_{\mathcal{F}}(v)(a)_a}}\oplus\emptyset_{\set{v}} \\
            &= g|_{\set{\ast_a \take v_a}} = \augm{A}{\MG}[\sigma](g),
        \end{split}\end{equation}
        where $\sigma$ is the bijection which sends $\ast$ to $v$ and is the identity over $V$. If now $g\in\augm{A}{\MG}[V]$, then for any $f$,
        \begin{equation}
            e(\ast)\comp^f g = \emptyset_{\set{\ast}}|_{\set{\ast_a \take f(a)_a}}\oplus g = g.
        \end{equation}
    \end{description}
\end{proof}

In all the following we only consider this semi-direct, albeit not exactly on $\augm{A}{\MG}$ and $\func{A}{\K\Poly^1}$, and we will hence omit the $\varphi$ index.
We call {\em graph insertion operad} any operad which can be written with this semi-direct product.

\begin{remark}
    This notion of graph insertion operad is different than the one mentioned in \cite{Kreimer:2000ja}, in the context of Feynman graph insertions in quantum field 
    theory.
\end{remark}

As mentioned above, this construction is supposed to encode partial compositions of the form: take the disjoint union, forget the vertex on which we compose,
and independently reconnect loose edges. Let $g_1\otimes f_1$ and $g_2\otimes f_2$ be two elements of $\augm{A}{\MG}\ltimes\func{A}{\K\Poly^1}$. The ends of $g_1$ are 
labelled by elements of $A$. When considering $g_1\comp^{f_2} g_2$, the map $f_2$ is there to tell us how we should connect the loose ends obtained
by forgetting $\ast$: each end labelled with $a$ will be connected to the elements of $f_2(a)$. More pratically, when defining a partial composition
of the above form, one will have certain types of ends 1,2,3\dots and for each type of edge a way to reconnect them. For example 'edges of type $i$ connect 
to vertices $a$ and $b$ or to vertex $c$'. This translates by $f(i) = a\oplus b + c$. For a partial composition defined this way, we then just need
to verify that the sums of polynomials $f(i)$ are a stable sub-species of $\K\Poly^1$ by composition of polynomials i.e. a sub-operads of $\K\Poly^1$.

Let us give two simple examples of operads which can be defined using this semi-direct product. Recall from Section~\ref{sec_intro} that we have a natural 
embedding of $\spe{\id}$ in $\K\Poly$, and three natural embeddings of $\Sym_+$ in $\K\Poly$.

\begin{example}
    $\point{\G}$ has a natural operad structure given by $\point{\G} \cong \G\times \spe{\id} \cong \augm{\set{0}}{\G}\ltimes\func{\set{0}}{\spe{\id}}$.
    For $(g_1,v_1)$ and $(g_2,v_2)$ two pointed graphs, the partial composition $(g_1,v_1)\comp(g_2,v_2)$ is then equal to $(g_3,v_1|_{\ast\take v_2})$ 
    where $g_3$ is the graph obtained by connecting all the ends on $\ast$ to $v_2$. More formally,
    \begin{equation}\begin{split}
        (g_1,v_1)\comp (g_2,v_2) &= (g_1|_{\ast\take v_2}\oplus g_2, v_1|_{\ast\take v_2}) \\
        &= (\G[\sigma](g_1)\oplus g_2,v_1|_{\ast\take v_2}),
    \end{split}\end{equation}
    where $\sigma$ is the bijection which sends $\ast$ on $v_2$ and which is the identity on the rest of its domain. For instance, we have:
    \begin{equation}
        \begin{tikzpicture}[Centering,scale=.7]
            \node[NodeGraph](s)at(0,0){$\ast$};
            \node[RootGraph](a)at(1,1){$a$};
            \node[NodeGraph](b)at(1,-1){$b$};
            \draw[EdgeGraph](a)--(s);
            \draw[EdgeGraph](s)--(b);
        \end{tikzpicture}
        \enspace \comp \enspace
        \begin{tikzpicture}[Centering,scale=1]
            \node[RootGraph](c)at(0,0){$c$};
            \node[NodeGraph](d)at(1,0){$d$};
            \draw[EdgeGraph](c)--(d);
        \end{tikzpicture}
        \enspace = \enspace
        \begin{tikzpicture}[Centering,scale=.7]
            \node[RootGraph](a)at(1,1){$a$};
            \node[NodeGraph](b)at(1,-1){$b$};
            \node[NodeGraph](c)at(0,0){$c$};
            \node[NodeGraph](d)at(2,0){$d$};
            \draw[EdgeGraph](a)--(c);
            \draw[EdgeGraph](c)--(b);
            \draw[EdgeGraph](c)--(d);
        \end{tikzpicture}.
    \end{equation}
    Remark that the operad $\NAP$~\cite{Liv06} is a sub-operad of the operad above and hence is a graph insertion operad.
\end{example}

\begin{example}
    $\G$ has a natural operad structure given by $\G\cong \G\times \Sym_+\cong \augm{\set{0}}{\G}\ltimes\func{\set{0}}{\Sym}$,
    where we consider here the embedding $V\mapsto \bigoplus_{v\in V} v$.
    For $g_1$ and $g_2$ two graphs, the partial composition $g_1\circ g_2$ is then the graph obtained by adding an edge between each neighbour of 
    $\ast$ and each vertex of $g_2$. More formally, for $g_1\in G'[V_1]$ and $g_2\in G[V_2]$:
    \begin{equation}\begin{split}
        g_1\comp g_2 &= g_1|_{\ast \take \bigoplus V_2}\oplus g_2\\
        &= g_1|_{V_1}\oplus\bigoplus_{v\in n(\ast)} v\bigoplus_{v\in V_2}\oplus g_2 \\
        &= g_1|_{V_1}\oplus\bigoplus_{v_1\in n(\ast), v_2\in V_2} v_1v_2\oplus g_2,
    \end{split}\end{equation}
    where $n(\ast)$ is the set of neighbours of $\ast$. Each of terms of the last line have a combinatorial interpretation: $g_1|_{V_1}$ is $g_1$ to 
    which we removed $\ast$, the term $\bigoplus_{v_1\in n(\ast),v_2\in V_2}v_1v_2$ means that we add an edge between any element in $n(\ast)$ and any element in $V_2$, 
    and finally the term $g_2$ means that we keep all the edges of $g_2$. For instance, we have:
    \begin{equation}
        \begin{tikzpicture}[Centering,scale=.7]
            \node[NodeGraph](s)at(0,0){$\ast$};
            \node[NodeGraph](a)at(1,1){$a$};
            \node[NodeGraph](b)at(1,-1){$b$};
            \draw[EdgeGraph](a)--(s);
            \draw[EdgeGraph](s)--(b);
        \end{tikzpicture}
        \enspace \comp \enspace
        \begin{tikzpicture}[Centering,scale=1]
            \node[NodeGraph](c)at(0,0){$c$};
            \node[NodeGraph](d)at(1,0){$d$};
            \draw[EdgeGraph](c)--(d);
        \end{tikzpicture}
        \enspace = \enspace
        \begin{tikzpicture}[Centering,scale=.7]
            \node[NodeGraph](a)at(1,1){$a$};
            \node[NodeGraph](b)at(1,-1){$b$};
            \node[NodeGraph](c)at(0,0){$c$};
            \node[NodeGraph](d)at(2,0){$d$};
            \draw[EdgeGraph](a)--(c);
            \draw[EdgeGraph](c)--(b);
            \draw[EdgeGraph](c)--(d);
            \draw[EdgeGraph](a)--(d);
            \draw[EdgeGraph](b)--(d);
        \end{tikzpicture}.
    \end{equation}
\end{example}

Let us now prove that the two partial composition introduced in subsection~\ref{subsec_cano_op} do indeed make $\point{\MG}_{or}$ and $\MG$ operads.

\begin{proof}[proof of Theorem~\ref{th_op_graph_or}]
    Recall from Example~\ref{ex_deform} that we have a bijection $\MG_{or}\cong \augm{\set{\mathtt{s,t}}}{\MG}$. The isomorphism 
    $\point{\MG_{or}}\cong \augm{\set{\mathtt{s,t}}}{\MG}\times\spe{\id}\times\Sym_+ \cong \augm{\set{\mathtt{s,t}}}{\MG}\ltimes\func{\set{s},\set{t}}{\spe{\id},\Sym_+}$, 
    give the desired operad structure on $\point{\MG_{or}}$ when considering the embedding $V\mapsto \sum_{v\in V}v$ of $\Sym_+$ in $\K\Poly^1$.
\end{proof}

\begin{proof}[proof of Theorem~\ref{th_graphop_cano}]
    This is the operad structure on $\MG$ given by $\MG\cong \MG\times\Sym_+\cong \augm{\set{0}}{\MG}\ltimes\func{\set{0}}{\Sym_+}$
    when considering the embedding $V\mapsto \sum_{v\in V}v$ of $\Sym_+$ in $\K\Poly^1$.
\end{proof}

We end this section by mentioning that while we restricted ourselves to multigraphs in this paper, all the work done in this section
naturally generalize to the very general framework of multi-hypergraphs, whose edges can contain any non null number of vertices and can be appear more than once.
This is done by considering replacing $\MG$ by $\MHG$, the species of multi-hypergraphs, and $\K\Poly^1$ by $\K\Poly$ in Theorem~\ref{th_op_ins}.
In this more general context, the connection of ends can also change the number of vertices of the edges.

\bibliographystyle{plain}
\bibliography{redac}


\appendix

\section{Koszul duality and Koszul operads}\label{ann_koszul}
As said at the end of Section~\ref{sec_intro}, 
two advantages of defining an operad by its generators and relations are
that it is possible to construct (under some conditions) its Koszul dual and 
that it is possible to check if the operad is Koszul.


\subsection*{Koszul duality}
Let us begin by defining the Koszul dual of an operad.
To do this, we consider from now on that we have an arbitrary order for every finite set $V$ in order to consider the signature of a bijection 
between two different sets. Without loss of generality, we consider these order to follow the alphabetical order when applicable.

For $\spe{S}$ a linear species, we denote by $\spe{S}^*$ the {\em dual} species of $\spe{S}$ which is defined by $\spe{S}^*[V]=\spe{S}[V]^*$ and 
$\spe{S}^*[\sigma](f) = f\circ\spe{S}[\sigma^{-1}]$. We denote by $\spe{S}^\vee$ the species defined by $\spe{S}^\vee[V]=\spe{S}^*[V]$ and
$\spe{S}^\vee[\sigma](f) = \sign(\sigma)f\circ\spe{S}[\sigma^{-1}]$, where $\sign(\sigma)$ is the signature of $\sigma$. 

\begin{definition}
    Let $\Operad =\Ope(\mathcal{G},\mathcal{R})$ be a binary quadratic operad. Define the linear form $\scalar{ -}{ -}$ on 
    $\free{\mathcal{G}^{\vee}}^{(2)}\times\free{\mathcal{G}}^{(2)}$ by
    \begin{equation}
        \scalar{f_1\comp f_2}{x_1\comp x_2} = f_1(x_1)f_2(x_2),
    \end{equation}
    The {\em Koszul dual} of $\Operad$ is then the operad $\Operad^!=\Ope(\mathcal{G}^{\vee},\mathcal{R}^{\bot})$
    where $\mathcal{R}^{\bot}$ is the orthogonal of $\mathcal{R}$ for $\scalar{-}{-}$.
\end{definition}

\begin{example}\label{ex_kdual}
    The {\em Lie operad} $\Lie$ is the quotient of the free operad over one antisymmetric generator by the Jacobi relation. More formally, denote by $[a,b]$
    the generating element of $\Sym_2^{\vee}[\set{a,b}]$ and by $[b,a]= (ab)\cdot [a,b]$ so that $[b,a] = -[a,b]$. Then $\Lie=\Ope(\Sym_2^\vee,\mathcal{R})$ with 
    $\mathcal{R}$ the sub-species of $\free{\Sym_2^\vee}^{(2)}$ generated by the {\em Jacobi relation} $[a,[b,c]]+[c,[a,b]]+[b,[c,a]$, where $[a,[b,c]]$ stands 
    for $[a,\ast]\comp [b,c]$ (one can check that this is indeed stable under the action of $\symgrp_{\set{a,b,c}}$ and hence $\mathcal{R}$ is indeed a species). 
    This operad is the Koszul dual of $\Com$. Indeed, for every pair $\alpha < \beta$  in $\set{a,b,c,\ast}$ ordered with the alphabetical order ($c<\ast$) denote 
    by $s_{\alpha\beta}^{\vee}= [\alpha,\beta]$ the dual of $s_{\set{\alpha,\beta}}\in\ComMag[\set{\alpha, \beta}]$. Then we have, for example:
    \begin{equation}\begin{split}
        \scalar{
            s^\vee_{a\ast}\comp s^\vee_{bc} &+ s^\vee_{c\ast}\comp s^\vee_{ab} + s^\vee_{b\ast}\comp s^\vee_{ca}
            }{
            s_{\set{a,\ast}}\comp s_{\set{b,c}} - s_{\set{c,\ast}}\comp s_{\set{a,b}}
            }\\
        &= \scalar{
            s^\vee_{a\ast}\comp s^\vee_{bc}
            }{
            s_{\set{a,\ast}}\comp s_{\set{b,c}}
            } 
        +  \scalar{
            s^\vee_{c\ast}\comp s^\vee_{ab}
            }{
            s_{\set{a,\ast}}\comp s_{\set{b,c}}
            } 
        + \scalar{
            s^\vee_{b\ast}\comp s^\vee_{ca}
            }{
            s_{\set{a,\ast}}\comp s_{\set{b,c}}
            } \\
        &- \scalar{
            s^\vee_{a\ast}\comp s^\vee_{bc}
            }{
            s_{\set{c,\ast}}\comp s_{\set{a,b}}
            } 
        - \scalar{
            s^\vee_{c\ast}\comp s^\vee_{ab}
            }{
            s_{\set{c,\ast}}\comp s_{\set{a,b}}
            } 
        - \scalar{
            s^\vee_{b\ast}\comp s^\vee_{ca}
            }{
            s_{\set{c,\ast}}\comp s_{\set{a,b}}
            } \\
        &= s^\vee_{a\ast}(s_{\set{a,\ast}})s^{\vee}_{bc}(s_{\set{b,c}}) 
        + s^\vee_{c\ast}(s_{\set{a,\ast}})s^{\vee}_{ab}(s_{\set{b,c}}) 
        + s^\vee_{b\ast}(s_{\set{a,\ast}})s^{\vee}_{ca}(s_{\set{b,c}}) \\
        &- s^\vee_{a\ast}(s_{\set{c,\ast}})s^{\vee}_{bc}(s_{\set{a,b}}) 
        - s^\vee_{c\ast}(s_{\set{c,\ast}})s^{\vee}_{ab}(s_{\set{a,b}}) 
        - s^\vee_{b\ast}(s_{\set{c,\ast}})s^{\vee}_{ca}(s_{\set{a,b}}) \\
        &= 1 +0+0-1-0-0 = 0.
    \end{split}\end{equation}
\end{example}


\subsection*{Koszul operads}\label{sec_koszul}
Koszulity is an important aspect of operad theory. We only give here a very quick overview of Koszulity and Gröbner bases for operads which hides a 
lot of the theory. We do not give the general results but only restricted versions which suffice for our use.
We refer the reader to the literature; for a broader approach of the topic, see for example \cite{LV12}, \cite{Men15}, \cite{Hof10} and \cite{DK10}.
In particular, all the examples presented here come from \cite{DK10}.

In order to give the characterisation which interests us, we need to introduce the concepts of $\mathscr{L}$-species, shuffle operads and Gröbner bases. 
Informally, we can see these objects as the same as species and operads, except with a total order on every set of vertices.

\bigskip
\noindent\textbf{$\mathscr{L}$-species}

A {\em linear positive $\mathscr{L}$-species}  consists of the following data:
\begin{itemize}
   \item for each finite set $V$ and total order $l$ on $V$, a 
   vector space $\spe{S}[V,l]$, such that $\spe{S}[\emptyset,\emptyset]=\set{0}$.
   \item For each increasing bijection $\sigma: (V,l)\to (V',l')$, a linear map $\spe{S}[\sigma]:\spe{S}[V,l]\to \spe{S}[V',l']$.
   These maps should be such that $\spe{S}[\sigma_1\circ\sigma_2] = \spe{S}[\sigma_1]\circ \spe{S}[\sigma_2]$ and $\spe{S}[\id] = \id$.
\end{itemize}

%

In the sequel, we write order to designate a total order and $\mathscr{L}$-species to designate linear positive $\mathscr{L}$-species. For $\spe{S}$ an
$\mathscr{L}$-species and $l$ an order on $V$, we also denote by $\spe{S}[l]=\spe{S}[V,l]$. We can do this since the data of $V$ is included in $l$.
As for species, we denote by $\spe{X}$ the $\mathscr{L}$-species defined $\spe{X}[V,l]=\set{0}$ if $V$ is not a singleton and $\spe{X}[\set{v},v]= \K{v}$ else.

As in the case of classical species, $\mathscr{L}$-species also have constructions on them, but before giving them, let us give some notations.
\begin{itemize}
    \item For $l$ an order on $V$ and $W\subseteq V$ a subset of $V$, we denote by $l_W$ the order on $W$ induced by $l$.
    \item For $l=l_1\dots l_n$ an order on a set $V$ of size $n$, $i\in[n]$ and $\ast\not\in V$, we denote by $l\stackrel{i}{\take}\ast$ the order 
    $l_1\dots l_{i-1}\ast l_i\dots l_n$ on $V+\ast$.
    \item For $l^1,\dots l^k$ $k$ orders on pairwise disjoint sets $V_1,\dots, V_k$, we denote by $sh(l^1,\dots,l^k)=\set{w\,|\, w_{V_i}=l^i}$ the set of 
    {\em shuffles} of $l^1,\dots, l^k$. Note that this is an ``associative operation'' in the sense that for $l^1,l^2,l^3$ three orders, the union of
    the shuffles of $l^1$ with the elements of $sh(l^2,l^3)$ is exactly $sh(l^1,l^2,l^3)$.
    \item  The {\em shuffle compositions} $\Comp_{sh}[V,l]$ of an ordered set $(V,l)$ are the compositions $P=P_1,\dots,P_k$ of $V$ such that, for every 
    $1\leq i<j\leq k$, $\min_l P_i<\min_l P_j$.
\end{itemize}

Let $\spe{R}$ and $\spe{S}$ be two linear $\mathscr{L}$-species and $l$ a total order on $V$. Denote by $n=\card{V}$ and let be $i\in[n]$. We define the following operations.

$$Product\quad \spe{R}\cdot\spe{S}[V,l]= \bigoplus_{l\in sh(l',l'')} \spe{R}[l']\otimes\spe{S}[l''],$$

$$ i\text{-}th Derivative\quad \spe{S}^i[V,l]= \spe{S}[V+\ast, l\stackrel{i}{\take}\ast],$$

$$Composition\quad \spe{R}(\spe{S})[V,l]=\bigoplus_{P\in\Comp_{sh}[V,l]}\spe{R}[\set{P_1,\dots,P_k},P]\otimes \spe{S}[P_1,l_{P_1}]\otimes\dots\otimes\spe{S}[P_k,l_{P_k}].$$

Since we have the $\mathscr{L}$-species $\spe{X}$ and the notion of composition, we can define Schröder tree on $\mathscr{L}$-species in the same way as
for species: if $\spe{S} = \spe{X} + \spe{S}_{2^+}$, then $\stree{\spe{S}} = \spe{X} + \spe{S}_{2^+}(\stree{\spe{S}})$. 
In this case, for $\spe{S}$ a $\mathscr{L}$-species, $\stree{\spe{S}}[V]$ is the vector space of rooted {\em planar} trees with internal vertices decorated 
with elements of $\spe{S}$ and set of leaves $V$.
This is because of the orders: instead of indexing with partitions in the composition we use shuffle compositions.

For $\spe{S}$ a $\mathscr{L}$-species and $l$ an order on $V$, the fact that we have an order enables us an easier notation of the elements of $\stree{\spe{S}}[V,l]$
as operations \textit{e.g.} $\alpha(l_1,\dots, l_n)$.

\begin{example}\label{ex_free_shuffle}
    Let $\spe{S}$ be the $\mathscr{L}$-species defined by $\spe{S}[V,l]=\set{0}$ when $\card{V}\not=2$ and $\spe{S}[V,l]=\K\set{1,\dots n}$ else, for $n$
    an integer greater than 1. Then the generators of $\stree{\spe{S}}^{sh}[[3],123]$ are of the form
    \begin{equation}
        \begin{tikzpicture}[Centering,xscale=0.6,yscale=0.4]
            \node[NodeFree](x)at(0,0){$x$};
            \node[NodeFree](y)at(-1,2){$y$};
            \node[NodeFree, draw=white](1)at(-2,4){$1$};
            \node[NodeFree, draw=white](2)at(0,4){$2$};
            \node[NodeFree, draw=white](3)at(1,2){$3$};
            \draw[EdgeFree](x)--(y);
            \draw[EdgeFree](x)--(3);
            \draw[EdgeFree](1)--(y);
            \draw[EdgeFree](2)--(y);
        \end{tikzpicture}
        \enspace , \enspace
        \begin{tikzpicture}[Centering,xscale=0.6,yscale=0.4]
            \node[NodeFree](x)at(0,0){$x$};
            \node[NodeFree](y)at(-1,2){$y$};
            \node[NodeFree, draw=white](1)at(-2,4){$1$};
            \node[NodeFree, draw=white](3)at(0,4){$3$};
            \node[NodeFree, draw=white](2)at(1,2){$2$};
            \draw[EdgeFree](x)--(y);
            \draw[EdgeFree](x)--(2);
            \draw[EdgeFree](1)--(y);
            \draw[EdgeFree](3)--(y);
        \end{tikzpicture}
        \enspace\text{and}\enspace
        \begin{tikzpicture}[Centering,xscale=0.6,yscale=0.4]
            \node[NodeFree](x)at(0,0){$x$};
            \node[NodeFree](y)at(1,2){$y$};
            \node[NodeFree, draw=white](1)at(-1,2){$1$};
            \node[NodeFree, draw=white](2)at(0,4){$2$};
            \node[NodeFree, draw=white](3)at(2,4){$3$};
            \draw[EdgeFree](x)--(y);
            \draw[EdgeFree](x)--(1);
            \draw[EdgeFree](3)--(y);
            \draw[EdgeFree](2)--(y);
        \end{tikzpicture},
    \end{equation}
    where $1\leq x,y\leq n$. This is because the only shuffle compositions of size 2 of $[3]$ are $(12,3)$, $(13,2)$ and $(1,23)$. 
    These trees can be denoted in a more compact way by $\mu_x(\mu_y(1,2),3)$, $\mu_x(\mu_y(1,3),2)$ and $\mu_x(1,\mu_y(2,3))$.
\end{example}

Let us end this brief presentation of $\mathscr{L}$-species by giving their link to classical species. Let $\mathscr{F}$ be the forgetful functor which send a  
species $\spe{S}$ on the $\mathscr{L}$-species $\spe{S}^{\mathscr{F}}$ defined by:
\begin{itemize}
    \item for $l$ an order on $V$, $\spe{S}^{\mathscr{F}}[V,l]=\spe{S}[V]$. 
    \item For $\sigma:(V,l)\to (V',l')$ an increasing bijection, $\spe{S}^{\mathscr{F}}[\sigma]$ is given by
    \begin{equation}
        \spe{S}^{\mathscr{F}}[V,l]=\spe{S}[V] \stackrel{\sigma}{\to} \spe{S}[W]= \spe{S}^{\mathscr{F}}[W,l'].
    \end{equation}
    \item For $f:\spe{S}\to\spe{R}$ a species morphism, $f^{\mathscr{F}}$ is the $\mathscr{L}$-species morphism given by
        \begin{equation}
        f^{\mathscr{F}}_{V,l}:\spe{S}^{\mathscr{F}}[V,l]= \spe{S}[V] \stackrel{f}{\to} \spe{R}[V]= \spe{R}^{\mathscr{F}}[V,l].
    \end{equation}
\end{itemize}
This is a forgetful functor in the sense that we forget the action $\sigma_V$ on $\spe{S}[V]$. We have the following fundamental proposition from
\cite{DK10}.

\begin{proposition}[Proposition 3 in \cite{DK10}]\label{prop_dk10}
    Let $\spe{S}$ and $\spe{R}$ be two species. Then 
    \begin{equation}
        (\spe{R}(\spe{S}))^{\mathscr{F}}=\spe{R}^\mathscr{F}(\spe{S}^\mathscr{F}).
    \end{equation}
\end{proposition}

\bigskip
\noindent\textbf{Shuffle operads}

We would like to define shuffle operads as $\mathscr{L}$-species satisfying the same axioms that a linear species must satisfy to be an operad. For now we can
not do this because we do not have the notion of derivative of a shuffle operad $\Operad'$. Fortunately, there is a way to make sense of the different 
species appearing in the diagrams \eqref{def_op}. 

For $l$ an order on $V$ and $v\in V$, denote by $\text{arg}_lv$ the index of $v$ in $l$: $l_{\text{arg}_lv}=v$. 
For $\spe{R}$ and $\spe{S}$ two $\mathscr{L}$-species we define the $\mathscr{L}$-species $\spe{R}'\cdot\spe{S}$ and $\spe{R}''\cdot\spe{S}^2$ as follow:
\begin{equation}\begin{split}
    \spe{R}'\cdot\spe{S}[V,l] &= \bigotimes_{l\in sh(l^1,l^2)} \spe{R}^{\text{arg}_l l_1^2}[l^1]\otimes\spe{S}[l^2], \\
    \spe{R}''\cdot\spe{S}^2[V,l] &= \bigotimes_{l\in sh(l^1,l^2,l^3)} \spe{R}^{\text{arg}_l l^2_1, \text{arg}_l l^3_1}[l^1]\otimes\spe{S}[l^2]\otimes\spe{S}[l^3].
\end{split}\end{equation}

A {\em shuffle operad} is then a $\mathscr{L}$-species $\Operad$ with a unity $e$ and a partial composition $\comp^{sh}:\Operad'\cdot\Operad\to\Operad$ 
such that the diagrams \eqref{def_op} commutes. 
For $\spe{S}$ a $\mathscr{L}$-species, the {\em free shuffle operad over $\spe{S}$} is denoted by $\free{\spe{S}}^{sh}$ and defined in the same way as the free
operad over a species. The same goes with the {\em ideal of a shuffle operad} and the notation $\Ope_{sh}(\mathcal{G},\mathcal{R})$.

\begin{remark}
    With the notations of elements of the free operad as operations, the partial composition of the free operad is then the composition of operation:
    $\mu_x(\dots, \ast,\dots)\comp^\xi \mu_y(\dots)= \mu_x(\dots,\mu_y(\dots),\dots)$.
\end{remark}

We have the following corollary from Proposition~\ref{prop_dk10}.
\begin{theorem}[Corollary 1 in \cite{DK10}]
    Let $\mathcal{G}$ be a species. The image by $\mathscr{F}$ of the free operad generated by $\mathcal{G}$ is isomorphic to the free shuffle operad generated
    by $\mathscr{F}(\mathcal{G})$. For $\mathcal{R}$ a sub-species of $\free{\mathcal{G}}$, the image by $\mathscr{F}$ of the operad ideal generated by $\mathcal{R}$ is 
    isomorphic to the shuffle operad ideal generated  by $\mathscr{F}(\mathcal{R})$.
    This writes as $\free{\mathcal{G}}^{\mathscr{F}} \cong \free{\mathcal{G}^{\mathscr{F}}}^{sh}$ and $(\mathcal{R})^{\mathscr{F}}=(\mathcal{R}^{\mathscr{F}})$. 

    Hence $\Ope(\mathcal{G},\mathcal{R})^{\mathscr{F}}\cong\Ope_{sh}(\mathcal{G}^{\mathscr{F}},\mathcal{R}^{\mathscr{F}})$.
\end{theorem}

\bigskip
\noindent\textbf{Admissible order}

Let $\spe{S}$ be a $\mathscr{L}$-species. In order to define the notion of Gröbner bases, we need to introduce an order on the trees generating $\free{\spe{S}}^{sh}[V,l]$.
Instead of giving the broader notion of admissible order defined in \cite{DK10}, we only give a small variation of the {\em path-lexicographic ordering}.

First for every order $l$ on $V$, fix a basis of $\spe{S}[l]$ and an order on this basis such that for every $x$ in the chosen basis of $\spe{S}[l]$ and $\sigma:l\to l'$,  $\sigma\cdot x$ is also in the chosen basis of $\spe{S}[l']$ and for $y$ an other element of the basis greater than $x$ we have 
$\sigma\cdot x<\sigma\cdot y$.
That is to say, the order does not depend on the labels (but it can depend on their relative order). Given an ordered basis, we also have an order on the words on elements of the basis given by the lexicographic order.
Let now be $t\in\free{\spe{S}}^{sh}[l_1\dots l_n]$ such that every internal node is labelled by an element of the chosen bases. For all $i\in[n]$, there is a unique path from $l_i$ to the root of $t$. Denote by $a_i$ the word composed, from left
to right, of the labels of the nodes of this path, from the root to the leaf. We associate to $t$ the sequence $(a_1,\dots, a_i, w)$, where $w$ is the
word obtained by reading the leaves of $t$ from left to right. 

For two trees $t,t'\in\free{\spe{S}}^{sh}$, with associated sequences $(a_1,\dots, a_n, w)$ and $(b_1\dots, b_n, w')$ we then compare $t$ and $t'$ by lexicographicaly comparing $a_1$ with $b_1$ then $a_2$ with $b_2$ etc and reverse lexicographicaly comparing $w$ with $w'$ if $a_i=b_i$ for all $i$.

\begin{example}
    The $\mathscr{L}$-species $\spe{S}$ of Example~\ref{ex_free_shuffle} have natural ordered bases equal to the set $\set{1,\dots,n}$ with the natural order.
    The sequences attached to the given trees are then respectively $(xy,xy,123)$, $(xy,x,xy,132)$ and $(x,xy,xy,123)$ and we have
    \begin{equation}
        \begin{tikzpicture}[Centering,xscale=0.6,yscale=0.4]
            \node[NodeFree](x)at(0,0){$x$};
            \node[NodeFree](y)at(-1,2){$y$};
            \node[NodeFree, draw=white](1)at(-2,4){$1$};
            \node[NodeFree, draw=white](2)at(0,4){$2$};
            \node[NodeFree, draw=white](3)at(1,2){$3$};
            \draw[EdgeFree](x)--(y);
            \draw[EdgeFree](x)--(3);
            \draw[EdgeFree](1)--(y);
            \draw[EdgeFree](2)--(y);
        \end{tikzpicture}
        \enspace > \enspace
        \begin{tikzpicture}[Centering,xscale=0.6,yscale=0.4]
            \node[NodeFree](x)at(0,0){$x'$};
            \node[NodeFree](y)at(-1,2){$y'$};
            \node[NodeFree, draw=white](1)at(-2,4){$1$};
            \node[NodeFree, draw=white](3)at(0,4){$3$};
            \node[NodeFree, draw=white](2)at(1,2){$2$};
            \draw[EdgeFree](x)--(y);
            \draw[EdgeFree](x)--(2);
            \draw[EdgeFree](1)--(y);
            \draw[EdgeFree](3)--(y);
        \end{tikzpicture}
    \end{equation}
    if $x>x'$ or $x=x'$ and $y>y'$ or $x=x'$ and $y=y'$. 
\end{example}

Now remark that trees in $\free{\spe{S}}^{sh}[V,l]$ with basis elements as internal node labels make a basis $\free{\spe{S}}^{sh}[V,l]$. Hence any $x\in\free{\spe{S}}^{sh}$ can be written as a sum of such elements and we define $\text{lt}(x)$ the {\em leading term} of $x$ as the maximal element in 
this sum.

\bigskip
\noindent\textbf{Divisibility and S-polynomials}

Let $\spe{S}$ be a $\mathscr{L}$-species. A tree $t$ of $\free{\spe{S}}^{sh}$ is {\em divisible} by another tree $t'$ of $\free{\spe{S}}^{sh}$ if $t'$
is a sub-tree of $t$. Here a sub-tree must also conserve the order of the leaves.
A tree $u$ of $\free{\spe{S}}^{sh}$ is a {\em small common multiple} of two tree $t$ and $t'$ if it is divisible by both $t$ and $t'$ and its number of vertices is 
less than the total number of vertices of $t$ and $t'$.

\begin{example}
    Let $\spe{S}$ be a $\mathscr{L}$-species, $l=l_1,\dots l_n$ an order and $\alpha(\beta(l_1,l_3),\gamma(\beta(l_2,l_6),l_4,l_5))$ an element of $\free{\spe{S}}^{sh}$.
    This tree has among its divisors $\alpha(\beta(l_1,l_3),l_2)$ and $\gamma(\beta(l_1,l_4),l_2,l_3)$ but not $\gamma(\beta(l_1,l_3),l_2,l_4)$.
\end{example}

If $t$ is divisible by $t'$, then there exists trees $\alpha$ and $\beta_1,\dots,\beta_k$ such that $t=\alpha(\dots, t'(\beta_1,\dots,\beta_k), \dots)$. We denote
by $m_{t,t'}$ the operation on any tree with same number of leaves than $t'$ which associate to a tree $u$ the tree $\alpha(\dots, u(\beta_1,\dots,\beta_k),\dots)$.
Let now $V$ be a finite set, $l$ an order on $V$ and $x,y\in\free{\spe{S}}^{sh}[V,l]$. Assume $\text{lt}(x)$ and $\text{lt}(y)$ have a small 
common multiple $u$. Then we have $m_{u,\text{lt}(x)}(\text{lt}(x)) = u = m_{u,\text{lt}(y)}(\text{lt}(y))$. We call {\em S-polynomial of $x$ and $y$ 
(corresponding to $u$)} the element
\begin{equation}
    s_u(x,y)= m_{u,\text{lt}(x)}(x) -\frac{c_x}{c_y}m_{u,\text{lt}(y)}(y),
\end{equation}
where $c_x$ and $c_y$ are the respective coefficient of the leading terms of $x$ and $y$.

\bigskip
\noindent\textbf{Gröbner bases and Koszulity}

We can finally give the definition of a Gröbner bases and a Koszul operad.

\begin{definition}[Definition 13 in \cite{DK10}]
    Let $\mathcal{G}$ be a $\mathscr{L}$-species and $\mathcal{R}$ be a $\mathscr{L}$-sub-species of $\free{\mathcal{G}}^{sh}$. Let $\mathcal{B}$ 
    be basis of $\mathcal{R}$.
    We say that $\mathcal{B}$ is a {\em Gröbner bases} of $\mathcal{R}$ if for every $x\in(\mathcal{R})$, the leading term of $x$ is divisible by 
    the leading term of one element in $\mathcal{B}$.
\end{definition}

\begin{definition}[Corollary 3 in \cite{DK10}]
    Let $\mathcal{G}$ be a $\mathscr{L}$-species and $\mathcal{R}$ be a quadratic $\mathscr{L}$-sub-species of $\free{\mathcal{G}}^{sh}$. We say that
    $\Ope_{sh}(\mathcal{G},\mathcal{R})$ is {\em Koszul} if $\mathcal{R}$ admits a Gröbner bases.

    Let $\mathcal{G}$ be a set species and $\mathcal{R}$ be a quadratic sub-species of $\free{\mathcal{G}}^{sh}$. We say that
    $\Ope(\mathcal{G},\mathcal{R})$ is {\em Koszul} if $\Ope_{sh}(\mathcal{G}^{\mathscr{F}},\mathcal{R}^{\mathscr{F}})$ is Koszul.
\end{definition}

When $\Operad$ is a Koszul symmetric operad, it admits a Koszul dual $\Operad^!$. In this case the Hilbert series of $\Operad$ and $\Operad^!$ are related by
the identity:
\begin{equation} \label{hdual}
    \hilbert{\Operad}(-\hilbert{\Operad^!}(-t)) = t.
\end{equation}

Let us finish by a characterisation of Gröbner bases.

\begin{proposition}[Theorem 1 in \cite{DK10}]\label{prop_spol}
    Let $\mathcal{G}$ be a $\mathscr{L}$-species and $\mathcal{R}$ be a $\mathscr{L}$-sub-species of $\free{\mathcal{G}}^{sh}$. Let $\mathcal{B}$ 
    be basis of $\mathcal{R}$. Then $\mathcal{B}$ is a Gröbner bases if and only if for all pair of elements in $\mathcal{B}$, their S-polynomials
    are congruent to zero modulo $\mathcal{B}$ (\textit{i.e.} they are in $(\mathcal{R})$).
\end{proposition}

\end{document}